\documentclass[11pt]{article}
\usepackage{geometry}                
\usepackage{graphicx}
\usepackage{amssymb}
\usepackage{epstopdf}
\usepackage{setspace}
\usepackage{amsmath}
\usepackage{mathrsfs}
\usepackage{amsthm}
\usepackage[all,cmtip]{xy}
\usepackage[normalem]{ulem}
\usepackage{stmaryrd}
\usepackage{enumitem}
\usepackage{color}
\usepackage{hyperref}
\usepackage[title]{appendix}
\newcommand{\C}{\mathbb{C}}
\newcommand{\Q}{\mathbb{Q}}
\newcommand{\R}{\mathbb{R}}
\newcommand{\Z}{\mathbb{Z}}
\newcommand{\A}{\mathcal{A}}
\renewcommand{\P}{\mathcal{P}}
\newcommand{\mcO}{\mathcal{O}}
\newcommand{\F}{\mathcal{F}}

\newcommand{\M}{\mathcal{M}}

\newcommand{\D}{\mathbb{D}}
\newcommand{\rot}{\text{rot}}
\newcommand{\V}{\mathcal{V}}
\newcommand{\vol}{\operatorname{vol}}

\newcommand{\ro}{\mathfrak{r}}

\newcommand{\hm}{\text{hm}}

\theoremstyle{plain}
\newtheorem{thm}{Theorem}
\newtheorem{cor}[thm]{Corollary}
\newtheorem{prop}[thm]{Proposition}

\newtheorem{lem}[thm]{Lemma}

\theoremstyle{definition}
\newtheorem{defn}[thm]{Definition}

\newtheorem{rmrk}[thm]{Remark}
\newtheorem{ex}[thm]{Example}

\numberwithin{equation}{section}
\numberwithin{thm}{section}

\title{Mean action of periodic orbits of area-preserving annulus diffeomorphisms}
\author{Morgan Weiler\footnote{Partially supported by the NSF grant EMSW21-RTG-1344991.}}
\date{}

\begin{document}
\maketitle

\begin{abstract}
\noindent An area-preserving diffeomorphism of an annulus has an ``action function" which measures how the diffeomorphism distorts curves. The average value of the action function over the annulus is known as the Calabi invariant of the diffeomorphism, while the average value of the action function over a periodic orbit of the diffeomorphism is the mean action of the orbit. If an area-preserving annulus diffeomorphism is a rotation near the boundary, and if its Calabi invariant is less than the maximum boundary value of the action function, then we show that the infimum of the mean action over all periodic orbits of the diffeomorphism is less than or equal to its Calabi invariant.
\end{abstract}

\section{Introduction}\label{sec:intro}

In this paper we study the geometry of the periodic orbits of area-preserving diffeomorphisms of the annulus with boundary. We prove that a large class of these diffeomorphisms have a periodic orbit with an upper bound on the amount by which the diffeomorphism distorts curves between the orbit and the boundary.

Our result provides a quantitative illustration of several general themes in the study of periodic points of annulus diffeomorphisms. The celebrated Poincar\'{e}-Birkhoff theorem \cite{birkhoff2} tells us that a diffeomorphism of an annulus with boundary is guaranteed to have at least two fixed points so long as it is area-preserving and twists the boundary components in opposite directions. Franks proved in 1992 in \cite{franks} that an area-preserving homeomorphism of the annulus with boundary is guaranteed to have infinitely many interior periodic points if it has any periodic points. The area-preserving condition is necessary in both theorems: there are simple examples of diffeomorphisms of the annulus which are not area-preserving, are twist maps, have periodic points along the boundary, and have no periodic points in the interior (see \cite[Chapter 8]{mcdsal}). Therefore, it is natural to seek to formulate a set of sufficient conditions for the existence of periodic points entirely in terms of the quantitative properties of the area-preserving geometry. Similarly, it is natural to investigate the geometry of the resulting periodic points in quantitative terms.

We quantify the geometry using an action function. The precise definition is given as Definition \ref{defn:af}, however, the intuition behind it is the following. Let $z$ be a fixed point of an area-preserving diffeomorphism $\psi$ of the closed annulus $A$, which preserves the boundary components of $A$ as sets.  Choose any curve $\eta$ connecting $z$ to the boundary of $A$ preferred by the action function (this is an arbitrary choice we make in the definition of the action function: see (\ref{eqn:fplus}). The image of $\eta$ under $\psi$, the original curve $\eta$, and the boundary of $A$ determine some signed area. The value of the action function at $z$ is this area, up to an integer corresponding to the choice of a lift of $\psi$ to the universal cover of $A$. The action function has a natural extension to a real-valued function on $A$.

The average of the action function over $A$ is known as the Calabi invariant of $\psi$. If we assume further that $\psi$ is a rotation near the boundary, we show in Theorem \ref{thm:main} that so long as the Calabi invariant is less than the maximum of the two boundary values of the action, then there is a periodic orbit of $\psi$ over which the average of the action is at most the Calabi invariant plus any small $\epsilon>0$. (Corollary \ref{cor:main} provides a corresponding result when the Calabi invariant is greater than the minimum of the two boundary values of the action.) Our result provides periodic orbits in examples when it is comparatively easy to compute the Calabi invariant but otherwise difficult to understand the dynamics, e.g. Example \ref{ex:backrot}. Furthermore, the periodic orbits picked out by our theorem correspond to fixed points of some iterate of $\psi$, so our upper bound on the action constrains the effect the iterate $\psi$ can have on curves near this fixed point.

Our result also provides quantitative sufficient conditions for periodic points. Frank's theorem can be read as evidence for the philosophy that if an area-preserving diffeomorphism does not have the simplest dynamics possible, then we expect it to have very complex dynamics (see Ginzburg, \cite[\S6.2]{ginzconley}). Our Theorem \ref{thm:main} and Corollary \ref{cor:main} imply that if both boundary values of the action and the Calabi invariant are not all equal and irrational, then $\psi$ will have periodic points. Therefore, not only are the dynamics of $\psi$ delicate, exhibiting this all-or-nothing dichotomy, but our hypotheses provide a necessary quantitative condition for $\psi$ to balance between ``all" and ``nothing."

In \cite{mean}, Hutchings proved an analogous result for the disk via realizing the diffeomorphism in question as the Poincar\'{e} return map of a Reeb flow on $S^3$. The main tool in \cite{mean} is a filtration on the ECH of a three-manifold with zero first homology by the linking number of the generators with a chosen elliptic orbit. We generalize the results of \cite{mean} to the annulus. There are a number of reasons why it is not possible to directly extend the techniques from \cite{mean}. Firstly, the three-manifolds involved are a family of contact lens spaces rather than spheres. This requires us to extend the construction the linking number filtration on ECH from manifolds with zero first homology to manifolds with first Betti number zero. In order to compute the sum of two such filtrations on our family of lens spaces, we extend the methods from \cite{T3} and \cite{toric}, however, a number of features differ, particularly the index computations in \S\ref{subsec:computeech}. Secondly, the additional complications introduced by the second boundary component mean that we are not able to obtain as strong a result purely through contact geometry as Hutchings is in \cite{mean} (compare our Proposition \ref{prop:penultimate} to \cite[Proposition 2.2]{mean}), and therefore must take much more care when proving our main theorem, Theorem \ref{thm:main}, at the end of \S\ref{subsec:finalbound}. Our theorem is also stronger than the theorem for the annulus which can be obtained as a corollary of the disk theorem: see \S\ref{subsec:whatpsi}.

It is likely possible to generalize our techniques and those of \cite{mean} to surfaces with higher genus or more boundary components, however, such an extension would require new computational methods for ECH.

\subsection{Definitions and main theorem}\label{subsec:defnsthm}

Denote by $(A,\omega)$ the annulus $[-1,1]_x\times(\R/2\pi\Z)_y$ with the symplectic form $\omega=\frac{1}{2\pi}\,dx\wedge dy$.

Let $\psi:(A,\omega)\to(A,\omega)$ be an area-preserving diffeomorphism which fixes the boundary components $\partial_\pm A:=\{\pm1\}\times(\R/2\pi\Z)$ as sets and which is a rotation near the boundary. That is, for any choice of a lift $\tilde\psi$ of $\psi$ to the universal cover $\tilde A=[-1,1]\times\R$ of $A$, there are $y_\pm\in\R$ for which
\[
\tilde\psi(x,y)=\begin{cases}(x,y+2\pi y_+)&\text{ for $x$ sufficiently close to 1}
\\(x,y+2\pi y_-)&\text{ for $x$ sufficiently close to $-1$}\end{cases}
\]

Note that the choices of $y_+$, $y_-$, and $\tilde\psi$ are all equivalent. Denote by $G$ the set of all pairs $(\psi,y_+)$ as above, and note that $G$ is a group under $(\psi,y_+)\circ(\psi',y_+')=(\psi\circ\psi',y_++y_+')$.  We say ``$\tilde\psi$ is a translation by $2\pi y_\pm$ near $\partial\tilde A$" to emphasize the choice of lift, and we say ``$\psi$ is a rotation by $y_\pm$ near $\partial A$" when we want to emphasize the geometry of $\psi$ in $A$ at the expense of specifying the dependence of $y_\pm$ on a choice of lift $\tilde\psi$.

Let $\beta$ be any primitive of $\omega$ for which
\begin{equation}\label{eqn:betaboundary}
\beta|_{\partial A}=\begin{cases}\frac{1}{2\pi}\,dy&\text{ if } x=1
\\-\frac{1}{2\pi}\,dy&\text{ if } x=-1\end{cases}
\end{equation}

Because $\psi$ is a symplectomorphism, the one-form $\psi^*\beta-\beta$ is closed, and therefore represents a class in $H^1(A,\R)$. Elements of $H^1(A;\R)$ are determined by the value they take on any circle generating $H_1(A;\R)$, and $\psi^*\beta-\beta$ sends that circle to zero. Therefore it is exact. We are interested in one of the functions exhibiting this exactness because it carries dynamical information about $\psi$.

\begin{defn}\label{defn:af} The \emph{action function} of $(\psi,y_+)\in G$ with respect to the primitive $\beta$ of $\omega$ satisfying (\ref{eqn:betaboundary}) is the unique function $f=f_{(\psi,y_+,\beta)}:A\to\R$ for which
\begin{align}
df&=\psi^*\beta-\beta \label{eqn:df}
\\f|_{\partial_+A}&=y_+ \label{eqn:fplus}
\end{align}
\end{defn}

Note that $f$ is constant near $\partial A$.

Let $F$ denote the flux of $\tilde\psi$ applied to a generator of $H_1(A,\partial A)$, oriented from $\partial_-A$ to $\partial_+A$. That is, if $[(x,0)]$ denotes this class, then $F$ is the $\tilde\omega:=\frac{1}{2\pi}\,dx\wedge dy$-area in $\tilde A$ between $\psi(x,0)$ and $(x,0)$. We will also refer to $F$ as the \emph{flux} of $\tilde\psi$. The flux is related to the value of the action function at $\partial_-A$.

\begin{lem} $f|_{\partial_-A}=-y_-+F$
\end{lem}
\begin{proof} Compute
\begin{align*}
f|_{\partial_+A}-f|_{\partial_-A}&=df([(x,0)])
\\&=(\psi^*\beta-\beta)([(x,0)])
\\&=-F-\int_{2\pi y_+}^0\frac{1}{2\pi}\,dy-\int_0^{2\pi y_-}-\frac{1}{2\pi}\,dy
\\&=-F+y_++y_-
\end{align*}
Therefore if $f|_{\partial_+A}=y_+$, we must have $f|_{\partial_-A}=-y_-+F$.
\end{proof}

Similarly, if $(x,y)$ is a fixed point of $\psi$ and $\eta$ is a curve from $(x,y)$ to $\partial_+A$, then $f(x,y)$ is the area of the wedge between $\eta$, $\psi(\eta)$, and $\partial_+A$ plus an integer determined by $y_+$.

Although $f$ depends on $\beta$, we use it to define two other measurements of the geometry of $\psi$ which depend only on $(\psi,y_+)$. Therefore we drop the triple $(\psi,y_+,\beta)$ from our notation.

\begin{defn} The \emph{Calabi invariant} of $(\psi,y_+)$ is the number
\[
\V(\tilde\psi)=\V(\psi,y_+):=\frac{\int_Af\omega}{\int_A\omega}
\]
\end{defn}

\begin{lem}\label{lem:calabiindep} $\V(\tilde\psi)$ is independent of the choice of primitive $\beta$ for $\omega$ satisfying (\ref{eqn:betaboundary}).
\end{lem}
\begin{proof} If $\beta'\in\Omega^1(A)$ is another primitive for $\omega$, then $\beta-\beta'$ represents a class in $H^1(A)$. Such classes are determined by their value on any generator for $H_1(A)$, so if $\beta'$ satisfies (\ref{eqn:betaboundary}), then $\beta-\beta'$ represents the zero class and therefore is exact.

Let $\beta-\beta'=dg$ for $g:A\to\R$. We have
\[
df_{(\psi,y_+,\beta')}=\psi^*\beta'-\beta'=\psi^*(\beta-dg)-(\beta-dg)=df_{(\psi,y_+,\beta)}+\left(dg-\psi^*dg\right)=df_{(\psi,y_+,\beta)}+d(g-g\circ\psi)
\]
Because of the boundary conditions (\ref{eqn:betaboundary}) on $\beta$ and $\beta'$, we know that $g$ is constant on $\partial A$. In particular, $g\circ\psi|_{\partial A}=g|_{\partial A}$, therefore
\[
f_{(\psi,y_+,\beta')}=f_{(\psi,y_+,\beta)}+g-g\circ\psi
\]
since both sides satisfy both the differential (\ref{eqn:df}) and boundary (\ref{eqn:fplus}) requirements for the action function of $(\psi,y_+,\beta')$. Now we can compute
\[
\int_Af_{(\psi,y_+,\beta')}\omega=\int_A(f_{(\psi,y_+,\beta)}+g-g\circ\psi)\omega=\int_Af_{(\psi,y_+,\beta)}\omega\int_Ag\omega-\int_A(g\circ\psi)\omega=\int_Af_{(\psi,y_+,\beta)}\omega
\]
where $\int_A(g\circ\psi)\omega=\int_Ag\omega$ because $\psi$ is an orientation-preserving diffeomorphism.
\end{proof}

When $\psi$ is the rotation $(x,y)\mapsto(x,y+2\pi y_+)$, using $\beta=\frac{x}{2\pi}\,dy$, we get $df\equiv0$, therefore $f\equiv y_+$, and $\V(\tilde\psi)=y_+$. In general, the Calabi invariant can be thought of as a kind of rotation number. For more about this perspective, see \cite{shel} and the references therein.

In the case that $\psi$ is Hamiltonian, the action function integrates the action form associated to the Hamiltonian. See \cite[Chapter 10]{mcdsal} for this perspective on symplectic manifolds without boundary, or \cite[\S2.2]{abhs} for the case of the disk.

On $A$, the Calabi invariant is a homomorphism $G\to\R$. However, we will only use
\[
\V(\psi,y_++1)=\V((id,1)\circ(\psi,y_+))=\V(\psi,y_+)+1
\]

\begin{defn} An $l$-tuple $\gamma=(\gamma_1,\dots,\gamma_l)$ of points in $A$ is a \emph{periodic orbit of $\psi$} if $\gamma_{i+1\mod l}=\psi(\gamma_i)$. It is \emph{simple} if $\gamma_i\neq \gamma_j$ when $i\neq j$.
\end{defn}

\begin{defn} The \emph{total action} of a periodic orbit $\gamma$ is $\A(\gamma):=\sum_{i=1}^lf(\gamma_i)$.
\end{defn}

\begin{lem}\label{lem:totalactionindep} The total action $\A(\gamma)$ of a periodic orbit $\gamma$ is independent of the choice of the primitive $\beta$ of $\omega$ satisfying (\ref{eqn:betaboundary}).
\end{lem}
\begin{proof} The proof is analogous to the proof of \cite[Lemma 1.1]{mean}.
\end{proof}

\begin{defn} Let $\ell(\gamma)$ denote the period of a periodic orbit, i.e. the cardinality of its underlying set. The \emph{mean action} of a periodic orbit $\gamma$ is the ratio $\frac{\A(\gamma)}{\ell(\gamma)}$.
\end{defn}

Let $\P(\psi)$ denote the set of simple periodic orbits of $\psi$. Our main result is
\begin{thm}\label{thm:main} Let $y_+$, $y_-\in\R$. Let $\psi$ be an area-preserving diffeomorphism of $(A,\omega)$, with $\tilde\psi$ a lift of $\psi$ to $\tilde A$ which is translation by $2\pi y_+$ near $\{1\}\times\R$ and by $2\pi y_-$ near $\{-1\}\times\R$. Let $F$ denote the flux of $\tilde\psi$. Assuming
\begin{equation}\label{eqn:Vmax}
\V(\tilde\psi)<\max\{y_+,-y_-+F\}
\end{equation}
or that one of $y_+$ or $y_-$ is rational, we have
\begin{equation}\label{eqn:goal}
\inf\left\{\frac{\A(\gamma)}{\ell(\gamma)}\;\middle|\;\gamma\in\P(\psi)\right\}\leq\V(\tilde\psi)
\end{equation}
\end{thm}

\begin{rmrk}\label{rmrk:rationalboundary} If one of $y_+$ or $y_-$ is rational and $\V(\tilde\psi)\geq\max\{y_+,-y_-+F\}$, we can show (\ref{eqn:goal}) directly. On the boundary component near which $\psi$ rotates by a rational number, there will be a periodic orbit through every point which has mean action $y_+$ or $-y_-+F$, respectively. $\V(\tilde\psi)$ is greater than or equal to both of these. Therefore, during the course of the proof we will assume $\V(\tilde\psi)<\max\{y_+,-y_-+F\}$. This assumption is crucial -- see Example \ref{ex:irratrot} for a simple example illustrating why.
\end{rmrk}

Because none of the quantities in Theorem \ref{thm:main} depend on $\beta$, from here on out, unless specified otherwise, we will fix
\[
\beta=\frac{x}{2\pi}\,dy
\]

\begin{ex}\label{ex:twist} Consider the map $\psi(x,y)=(x,y+\pi x)$. Figure \ref{fig:twistmap} depicts the action of $\psi$ on the $(x,0)$ curve and indicates its circle $x=0$ of fixed points.

\begin{figure}
\centering
\includegraphics[width=50mm]{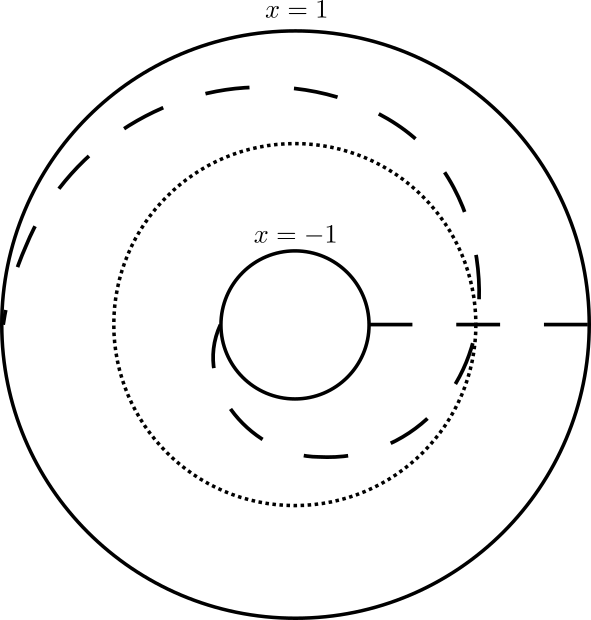}
\caption{The horizontal dashed segment is the $(x,0)$ curve, while the spiral dashed segment is the image of $\psi\circ(t\mapsto(t,0))$. Every point on the dotted circle $x=0$ is fixed.}
\label{fig:twistmap}
\end{figure}

We have $y_+=\frac{1}{2}$, $y_-=-\frac{1}{2}$, and $F=0$. Simple computations show that $f=\frac{1}{4}x^2+\frac{1}{4}$ and its average $\V(\tilde\psi)$ is $\frac{1}{3}$, which is less than $\frac{1}{2}=y_+=-y_-+F$. Moreover, $f$ achieves its minimum of $\frac{1}{4}<\V(\tilde\psi)$ at any point $(0,y)$, each of which is a fixed point.

(This map is not a rotation near the boundary, so it does not satisfy the hypotheses of Theorem \ref{thm:main}. However, there are diffeomorphisms to which our theorem does apply which are $C^0$ close to $\psi$, and for which all relevant quantities are close to those of $\psi$. See the proof of Theorem \ref{thm:main} in \S\ref{subsec:finalbound} for examples of such continuity with respect to the $C^0$ topology. Note that the Calabi invariant is not continuous with respect to the $C^0$ topology in general \cite{GambaudoGhys97}, but its values do converge on some subsequences of diffeomorphisms which converge in the $C^0$ topology, as shown in \cite{Oh16,Usher17}.)

\end{ex}

\begin{ex} Rational rotations are an illustration of the discussion in Remark (\ref{rmrk:rationalboundary}), in which every point is on a periodic orbit realizing the upper bound (\ref{eqn:goal}) on its mean action in terms of the Calabi invariant. Figure \ref{fig:halfrot} depicts the action of $\psi(x,y)=(x,y+\pi)$ on the $(x,0)$ curve and indicates a two-point periodic orbit.

\begin{figure}
\centering
\includegraphics[width=50mm]{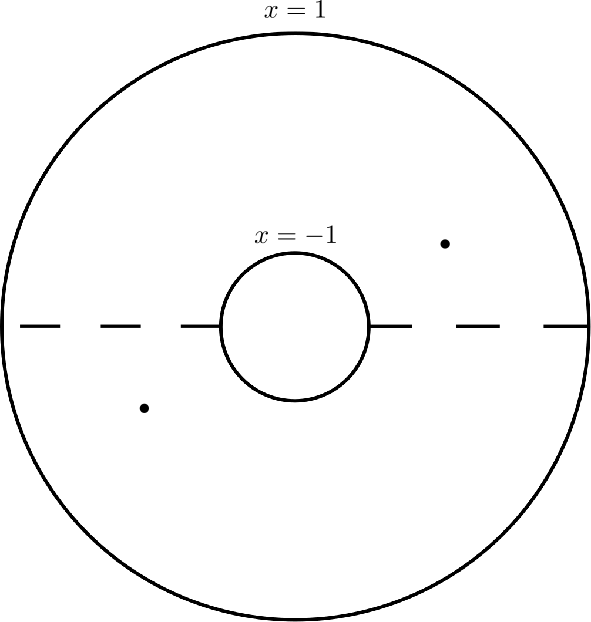}
\caption{The horizontal dashed segments are the $(x,0)$ curve and $\psi\circ(t\mapsto(t,0))$, while the two points make up an orbit of period two. Every point in $A$ is part of such an orbit.}
\label{fig:halfrot}
\end{figure}

We have $y_+=y_-=\frac{1}{2}$ and $F=1$. Because $\psi^*\beta=\beta$, the action function $f$ is identically $\frac{1}{2}$, so $\V(\tilde\psi)=\frac{1}{2}$. Every point is part of a two-point orbit with mean action $\frac{1}{2}=\V(\tilde\psi)$.
\end{ex}

\begin{ex}\label{ex:irratrot} Irrational rotations are an important example of $\psi$ to which the hypotheses of the theorem do not apply and for which its conclusion does not hold. These show that the inequality (\ref{eqn:Vmax}) is sharp when neither $y_+$ nor $y_-$ is rational.

Let $y_+\in\R-\Q$ and let $\psi(x,y)=(x,y+2\pi y_+)$. The action function $f$ is identically $y_+$ everywhere, so $\V(\tilde\psi)=y_+$ and $\psi$ does not satisfy the hypotheses of Theorem \ref{thm:main}. It also has no periodic orbits, therefore the infimum on the left-hand side of (\ref{eqn:goal}) is $+\infty$, meaning $\psi$ also does not satisfy the conclusion of Theorem \ref{thm:main}.
\end{ex}

\begin{ex}\label{ex:backrot} There are many somewhat mysterious symplectomorphisms to which our theorem applies. For example, consider the composition $\psi_2\circ\psi_1$, where $\psi_1$ is of the type described in Example \ref{ex:irratrot} with $0<y_0<2\pi$, and $\psi_2$ is a Hamiltonian symplectomorphism whose generating Hamiltonian is $y_1-x^2-y^2$ for $0<y_1<1-\epsilon$ when $x^2+y^2<y_1$, identically zero when $x^2+y^2>y_1+\epsilon$, and smooth and monotone in $x^2+y^2$ between. Intuitively, $\psi_2$ is a smoothing of a clockwise rotation in the circle $x^2+y^2<y_1$.

We can compute that the Calabi invariant of $\psi_2\circ\psi_1$ is strictly less than that of $\psi_1$, therefore our theorem proves that $\psi_2\circ\psi_1$ has periodic points.
\end{ex}

Replacing $(\psi,y_+)$ with $(\psi^{-1},-y_+)$ changes the signs of the Calabi invariant, the mean actions of periodic orbits, and the boundary values of the action function. Thus from Theorem \ref{thm:main} we immediately obtain:
\begin{cor}\label{cor:main} Let $y_+$, $y_-\in\R$. Let $\psi$ be an area-preserving diffeomorphism of $(A,\omega)$ which agrees with rotation by $2\pi y_+$ near $\partial_+A$ and by $2\pi y_-$ near $\partial_-A$, and whose flux is $F$. Assuming
\[
\V(\tilde\psi)>\min\{y_+,-y_-+F\}
\]
or that one of $y_+$ or $y_-$ is rational, we have
\[
\sup\left\{\frac{\A(\gamma)}{\ell(\gamma)}\;\middle|\;\gamma\in\P(\psi)\right\}\geq\V(\tilde\psi)
\]
\end{cor}

\subsection{Relation to area-preserving diffeomorphisms of the disk}\label{subsec:whatpsi}

Previously, Hutchings proved a version of Theorem \ref{thm:main} for the disk. The necessary definitions are entirely analogous to ours, and we have replaced Hutchings' notation with our own when it refers to equivalent objects. The $\D^2$ referred to is the unit disk in $\R^2$ with coordinates $x$, $y$.

\begin{thm}[Hutchings, {\cite[Theorem 1.2]{mean}}]\label{thm:disk} Let $\theta_0\in\R$, and let $\psi$ be an area-preserving diffeomorphism of $\left(\D^2,\frac{1}{\pi}dx\wedge dy\right)$ which agrees with rotation by angle $2\pi\theta_0$ near the boundary. Suppose that
\[
\V(\psi,\theta_0)<\theta_0
\]
Then
\begin{equation}\label{eqn:ineqdisk}
\inf\left\{\frac{\A(\gamma)}{\ell(\gamma)}\;\middle|\;\gamma\in\P(\psi)\right\}\leq\V(\psi,\theta_0)
\end{equation}
\end{thm}

In some cases it is possible to derive Theorem \ref{thm:main} from Theorem \ref{thm:disk} by collapsing the component of $\partial A$ corresponding to $\min\{y_+,-y_-+F\}$ to a point to obtain an area-preserving diffeomorphism of a disk. However, we show in Proposition \ref{prop:appendix} that it is not possible to do this in general for $\psi$ for which
\begin{equation}\label{eqn:12fv}
\frac{1}{2}F\leq\V(\tilde\psi)
\end{equation}
Therefore, Theorem \ref{thm:main} provides a large new class of area-preserving annulus diffeomorphisms for which (\ref{eqn:goal}) holds.

\subsection{Outline of the proof}\label{subsec:pfoutline}

We prove Theorem \ref{thm:main} by embedding the annulus in a contact three-manifold in such a way that the contact geometry recovers the dynamics of $\psi$. We are then able to use two filtrations on embedded contact homology, a homology theory for periodic orbits of Reeb vector fields on contact three-manifolds, to find a periodic orbit of $\psi$ which has both an upper bound on its total action and a lower bound on its period.

We begin in \S\ref{sec:reviewcontgeo} with a review of the necessary background. In \S\ref{sec:contmfld} we construct a contact form $\lambda_{\tilde\psi}$ on the lens space $L(y_+-y_-+F,y_+-y_-+F-1)$ whose geometry corresponds to the dynamics of $\psi$.

In \S\ref{sec:reviewech} and \S\ref{sec:kfech} we review embedded contact homology and its action filtration, and extend the filtration by the sum of the linking numbers with an elliptic orbit from \cite{mean}. We compute the embedded contact homology of $\lambda_{\tilde\psi}$ as well as its filtration by linking number, in \S\ref{sec:kfech}. While the fact that there are two binding components rather than one produces complications throughout the paper which do not appear in \cite{mean}, this computational section is one of the major differences from \cite{mean}. In order to prove Lemma \ref{lem:Nkinequality}, we need an understanding of the ECH chain complex which is not required in \cite{mean}.

In \S\ref{sec:proofmainthm} we use the computations from \S\ref{sec:kfech} as well as the dictionary from \S\ref{sec:contmfld} to prove the main theorem. Here is the other major difference from \cite{mean}; because we must analyze the annulus maps in much more detail, this section is quite a bit longer than the corresponding proof of \cite[Theorem 1.2]{mean} assuming \cite[Proposition 2.2]{mean}. The details are discussed in the introductory paragraphs of \S\ref{sec:proofmainthm} and in the introduction to the final proof of Theorem \ref{thm:main}.

\section{Review of contact geometry and dynamics in open book decompositions}\label{sec:reviewcontgeo}

In this section we will review some basic objects of contact geometry, paying particular attention to the relationships between the dynamics of the Reeb vector field and open book decompositions. These definitions are mostly standard.

\subsection{Basic definitions}\label{subsec:reviewdefns}

Let $Y$ be a closed oriented three-manifold.

Recall that a \emph{contact form} on $Y$ is a one-form $\lambda$ for which $\lambda\wedge d\lambda>0$. We can define its \emph{volume} to be $\vol(Y,\lambda):=\int_Y\lambda\wedge d\lambda$, and its \emph{contact structure} $\xi$ to be the oriented two-plane field $\ker\lambda$. Two contact structures $\xi_0$ and $\xi_1$ on $Y$ are \emph{contactomorphic} if there is a diffeomorphism of $Y$ which pushes one back to the other; if $\lambda_0$ and $\lambda_1$ have contactomorphic kernels, we say they are contactomorphic as well.

The \emph{Reeb vector field} of a contact form $\lambda$ is a smooth vector field on $Y$ determined by $d\lambda(R,\cdot)=0$ and $\lambda(R)=1$.

A \emph{Reeb orbit} is a smooth map $\gamma:\R/T\Z\to Y$ for some $T>0$, modulo reparameterization, for which $\dot\gamma(s)=R(\gamma(s))$. The \emph{symplectic action} of $\gamma$ is $\A(\gamma):=\int_\gamma\lambda=T$. We call $\gamma$ \emph{simple} if it is embedded. We use $\P(\lambda)$ to denote the set of simple Reeb orbits of $\lambda$.

Let $\psi_t$ denote the time $t$ flow of the Reeb vector field. A Reeb orbit $\gamma$ of action $T$ determines a symplectic linear map
\[
P_\gamma:=d_{\gamma(0)}\psi_T:(\xi_{\gamma(0)},d\lambda)\to(\xi_{\gamma(T)},d\lambda)=(\xi_{\gamma(0)},d\lambda)
\]
called the \emph{linearized return map}. If $P_\gamma$ has $1$ as an eigenvalue, then $\gamma$ is \emph{degenerate}, and it is \emph{nondegenerate} otherwise. The contact form $\lambda$ is \emph{degenerate} if any of its Reeb orbits are degenerate and \emph{nondegenerate} if all its Reeb orbits are nondegenerate.

In a choice of symplectic trivialization of $\gamma^*\xi$, the linearized return map is represented by a $2\times2$ symplectic matrix. If the eigenvalues of $P_\gamma$ are distinct complex conjugates on the unit circle, the orbit is \emph{elliptic}, and if they are distinct real multiplicative inverses, it is \emph{hyperbolic}. Further, a hyperbolic orbit is \emph{positive} or \emph{negative} according to whether its eigenvalues are. If we refer to a Reeb orbit as either elliptic or hyperbolic, we are also including the assumption that it is nondegenerate.

A choice of symplectic trivialization over an elliptic or hyperbolic Reeb orbit determines its \emph{rotation number} $\rot_\tau(\gamma)$. Let $\tau$ be a symplectic trivialization of $\gamma^*\xi$. If $\gamma$ is elliptic, then the path determined by $d_{\gamma(t)}\psi_t|_\xi$ under $\tau$ is homotopic in $Sp(2,\R)$ to a path of matrices conjugate to rotations by a continuous family of angles $2\pi\theta(t)$, where $\theta:[0,T]\to\R$, and $\theta(0)=0$. The rotation number of $\gamma$ is
\[
\rot_\tau(\gamma):=\theta(T)
\]

If $\gamma$ is hyperbolic, then let $v\in\xi_{\gamma(0)}$ be an eigenvector of $P_\gamma$. Under the trivialization $\tau$, the family $d_{\gamma(t)}\psi_t(v)|_\xi$ rotates by $2\pi\frac{k}{2}$ for some $k\in\Z$. The rotation number of $\gamma$ is
\[
\rot_\tau(\gamma):=\frac{k}{2}
\]

Assume $\gamma$ or some cover of $\gamma$ is part of the boundary of a surface $\Sigma$. We can ask that a pushoff of $\gamma$ which is constant with respect to the trivialization $\tau$ has intersection number zero with $\Sigma$. When we refer to this type of trivialization, we will say it \emph{has linking number zero with $\gamma$ with respect to $\Sigma$}, and the notation we will use is $\tau_\Sigma$. Note that this pushoff may not extend to a nonzero section of $\Sigma^*\xi$, a more common condition in contact geometry.

\subsection{Topological constructions}\label{subsec:gssobd}

We will study the dynamics of $\psi$ by embedding $A$ as a surface in a contact three-manifold whose Reeb flow sends $A$ to itself via the map $\psi$.

\begin{defn} A \emph{global surface of section} for a smooth flow $\psi_t$ of a non-vanishing vector field $V$ on a closed three-manifold $Y$ is a surface $\Sigma\subset Y$, possibly with boundary, embedded on its interior, for which
\begin{enumerate}
\item $V$ is transverse to $\Sigma-\partial\Sigma$
\item $V$ is tangent to $\partial\Sigma$; equivalently, the components of $\partial\Sigma$ are periodic orbits of $V$
\item For all $y\in Y-\Sigma$ there are $t_+>0$ and $t_-<0$ for which $\psi_{t_+}(y),\psi_{t_-}(y)\in\Sigma-\partial\Sigma$.
\end{enumerate}
The \emph{return time} is defined for $y\in\Sigma-\partial\Sigma$ as the first positive time when the flow returns $y$ to $\Sigma$:
\[
f(y):=\min\{t>0\;|\;\psi_t(y)\in\Sigma-\partial\Sigma\}
\]

The \emph{Poincar\'{e} return map} or \emph{return map} sends each point $y$ in $\Sigma-\partial\Sigma$ to its image under the flow for its return time:
\[
\psi(y):=\psi_{f(y)}(y)
\]
We will often simply refer to this map as the ``return map" of the flow of the vector field, with the global surface of section clear from context.
\end{defn}
Notice that the maps $f$ and $\psi$ do not always extend continuously to $\partial\Sigma$ (perhaps if the return time blows up approaching the boundary).

Global surfaces of section arise naturally as the closures of the pages of open book decompositions. This is a good setting in which to study the dynamics of the Reeb vector field since the contact topology is controlled, as we will explain below.

\begin{defn} An \emph{open book decomposition} of a closed oriented three-manifold $Y$ is a pair $(B,\Pi)$ where
\begin{itemize}
\item $B$ is an oriented link, called the \emph{binding}.
\item $\Pi:Y-B\to S^1$ is a fibration of the complement of $B$.
\item For each $\theta$, the surface $\Pi^{-1}(\theta)$ intersects the boundary of a small tubular neighborhood of $B$ in a longitude for the relevant component of $B$. The intersection of each $\Pi^{-1}(\theta)$ with the complement of an open tubular neighborhood of $B$ is called the \emph{$\theta$-page} of $(B,\Pi)$, and is denoted by $\Sigma_\theta$.
\end{itemize}
In particular, the closures of the surfaces $\Pi^{-1}(\theta)$ are Seifert surfaces for $B$.
\end{defn}

It is frequently helpful to focus on the page, rather than on the three-manifold. Therefore we recall the following definition as well.

\begin{defn} An \emph{abstract open book} is a pair $(\Sigma,\phi)$ where
\begin{itemize}
\item $\Sigma$ is a compact oriented surface.
\item $\phi$ is a diffeomorphism of $\Sigma$ which is the identity in a neighborhood of each boundary component. $\phi$ is called the \emph{monodromy} of $(\Sigma,\phi)$.
\end{itemize}
\end{defn}

An open book decomposition $(B,\Pi)$ of $Y$ determines an abstract open book as follows. We choose a compact oriented surface $\Sigma$ diffeomorphic to $\Sigma_0$. We obtain the monodromy by choosing a flow $\phi_t$ on $Y$ which sends the $s$ page to the $s+t$ page, and which near the binding rotates in the meridional direction about the binding as its axis. For $y\in\Sigma$, the monodromy $\phi$ is the map which sends $y$ to its image under the flow for time 1 (after conjugation with the diffeomorphism between $\Sigma_0$ and $\Sigma$).

Conversely, an abstract open book determines an open book decomposition of a three-manifold up to diffeomorphism. $Y-B$ is diffeomorphic to the mapping torus of $\phi$, and we fill in the tubular neighborhoods of the binding with solid tori by sending the meridian of the solid torus to the $S^1$ factor of the mapping torus and a longitude to the boundary of a page. If we change the monodromy to a conjugate of $\phi$ by another diffeomorphism of $\Sigma$, we obtain diffeomorphic open book decompositions. Therefore we usually consider monodromies only up to isotopy relative to the boundary.

The group of diffeomorphisms of a surface up to isotopy relative to the boundary is generated by \emph{Dehn twists}. A Dehn twist is any map isotopic relative to the boundary to the following model maps, which are supported on an annulus in the surface. In the coordinates $[-1,1]\times\R/2\pi\Z$ on that annulus, the model maps are
\[
(x,y)\mapsto\left(x,y\pm\left(2\pi\frac{x}{2}+\pi\right)\right)
\]
When we add in the above expression, we call the Dehn twist \emph{positive} or \emph{right-handed}, and when we subtract, we call it \emph{negative} or \emph{left-handed}. We usually specify a Dehn twist by the circle corresponding to $x=0$ in the above parameterization. When we are working on $A$, we refer to this circle as the ``core circle" of $A$.

\begin{defn}[Following \cite{dynamical}] An open book decomposition $(B,\Pi)$ is \emph{adapted to} a contact form $\lambda$ if the page $\Pi^{-1}(0)$ is a global surface of section for the Reeb vector field $R$ of $\lambda$. We will also say that $\lambda$ is \emph{adapted to} $(B,\Pi)$.
\end{defn}

When an open book decomposition $(B,\Pi)$ is adapted to a contact form $\lambda$, we refer to the return map of the Reeb flow from the zero page to itself as the return map of $(\lambda,B,\Pi)$. If $(\Sigma,\phi)$ is the abstract open book determined by $(B,\Pi)$, the return map of $(\lambda,B,\Pi)$ induces a diffeomorphism $\psi$ of $\Sigma$. We also refer to $\psi$ as the return map of $(\lambda,B,\Pi)$.

Open book decompositions control the contact structures of adapted contact forms:

\begin{thm}(\cite{giroux}, \cite{lecobcs})\label{thm:aobtocontacto}\label{thm:contactomorphic} If $(B_0,\Pi_0)$ is adapted to $\lambda_0$, $(B_1,\Pi_1)$ is adapted to $\lambda_1$, and both $(B_i,\Pi_i)$ induce the abstract open books with diffeomorphic pages and with monodromies which are conjugate under that diffeomorphism, then $\ker\lambda_0$ and $\ker\lambda_1$ are contactomorphic.
\end{thm}

\subsection{Example: Lens spaces}\label{subsec:stdlens} Let $p\in\Z$. The lens space $L(p,p-1)$ is the quotient of
\[
S^3=\{(z_1,z_2)\in\C^2\;\big|\;|z_1|^2+|z_2|^2=1\}
\]
by the free action of the group
\[
G_{p,p-1}=\left\{\begin{pmatrix}\zeta&0\\0&\zeta^{p-1}\end{pmatrix}\;\middle|\;\zeta\in\C,\zeta^p=1\right\}
\]
We can understand the contact geometry of $L(p,p-1)$ via this quotient.

\begin{lem}\label{lem:stdlens}
\hfill
\begin{enumerate}[label=(\alph*)]
\item $L(p,p-1)$ has an open book decomposition, defined off the image of the Hopf link under the quotient by $G_{p,p-1}$, with annulus pages and monodromy $p$ right-handed Dehn twists along the core circle. \item There is a contact form on $L(p,p-1)$ adapted to this open book decomposition, and the return map of its Reeb flow is rotation by $\frac{p}{2}$.
\end{enumerate}
\end{lem}

\begin{proof}
We will use ``polar" coordinates $(r_1,\theta_1,\theta_2)$ on $S^3-\{z_1=0\}\cup\{z_2=0\}$, where $z_1=r_1e^{i\theta_1}$ and $z_2=\sqrt{1-r_1^2}e^{i\theta_2}$. There is an open book decomposition on $S^3$ with annulus pages and binding the Hopf link $\{r_1=0\}\cup\{r_1=1\}$. Its projection map is given by $\Pi_1(r_1,\theta_1,\theta_2)=\theta_1+\theta_2$. This fibration factors through the quotient by $G_{p,p-1}$, therefore $\Pi_1$ induces a fibration of $L(p,p-1)$ off of the image $H_p$ of the Hopf link under the quotient.

Denote the projection map of the induced fibration on $L(p,p-1)-H_p$ by $\Pi_p$. Its pages are the product of intervals (parameterized by $r_1$) and the images of the circles $\theta_1+\theta_2=const$. In the parameterization by $\theta_1$, the $G_{p,p-1}$ action is precisely the action by $\Z_p$ given by $\theta_1\mapsto\theta_1+\frac{2\pi}{p}$. Therefore the pages of $\Pi_p$ are also annuli.

Consider the following flow on $S^3$:
\[
(r_1,\theta_1,\theta_2)\mapsto(r_1,\theta_1+2\pi\delta(r_1)t,\theta_2+2\pi(1-\delta(r_1))t)
\]
where $\delta:[0,1]\to[0,1]$ is a smooth function which is identically one at zero and identically zero at one. The flow is meridional near each binding component and sends pages to pages. It descends to a flow on $L(p,p-1)$ which is meridional near $H_p$ and sends the zero page to the $2\pi t$ page. The time one map of the flow, restricted to a page of $\Pi_p$, is
\[
(r_1,\theta_1)\mapsto(r_1,\theta_1+2\pi\delta(r_1))=\left(r_1,\theta_1+p\left(\frac{2\pi}{p}\delta(r_1)\right)\right)
\]
therefore the monodromy of $\Pi_p$ is $p$ right-handed Dehn twists about the core circle of the annulus page parameterized by $[0,1]\times\R/\frac{2\pi}{p}\Z$. This proves (a).

To prove (b), let $\lambda_1$ denote the restriction to $S^3$ of the standard 1-form $\frac{1}{2}r_1^2\,d\theta_1+\frac{1}{2}r_2^2\,d\theta_2$ on $\C^2-\{0\}$. Its Reeb vector field is $R=2\partial_{\theta_1}+2\partial_{\theta_2}$, which has simple orbits $\gamma^{1,2}:=\{r_{2,1}=0\}\cap S^3$. Moreover, $R$ is transverse to the pages of $\Pi_1$ and tangent to its binding, and the return times are all $\frac{\pi}{2}$.

Let $\mathfrak{q}_p:S^3\to L(p,p-1)$ denote the quotient map. Because $\lambda_1$ is invariant under the action of $G_{p,p-1}$, there is a unique contact form $\lambda_p$ on $L(p,p-1)$ for which $\mathfrak{q}_p^*\lambda_p=\lambda_1$. Its Reeb vector field is given by ${\mathfrak{q}_p}_*R$, so ${\mathfrak{q}_p}_*R$ is transverse to the pages of $\Pi_p$ and tangent to its binding. Its forward flow is given by
\[
(r_1,\theta_1+2t,\theta_2+2t)
\]
for $t\geq0$. All points in the zero page return to the zero page of $\Pi_p$ under this flow for the first time at $t=\frac{\pi}{2}$. Therefore the return map of $(\lambda_p,H_p,\Pi_p)$ is
\[
(r_1,\theta_1)\mapsto(r_1,\theta_1+\pi)
\]
which is rotation by $\frac{p}{2}$ on the annulus parameterized by $[0,1]\times\R/\frac{2\pi}{p}\Z$, proving (b).
\end{proof}

We will denote the abstract open book which can be obtained from $(H_p,\Pi_p)$ by $(A,D_p)$.

\section{Construction of the contact manifold}\label{sec:contmfld}

Now we begin the proof of Theorem \ref{thm:main}. In this section we realize $\psi$ as the return map of the Reeb flow of a contact form $\lambda_{\tilde\psi}$ adapted to an open book decomposition with annulus pages on a three-manifold in such a way that the geometry of $\lambda_{\tilde\psi}$ records the dynamics of $\psi$ (i.e., conclusions 1-4 of Proposition \ref{prop:constructcontact} hold). Therefore, the main content of Proposition \ref{prop:constructcontact} is that the three-manifold we must use is $L(\tilde p,\tilde p-1)$, where $\tilde p=y_+-y_-+F$. This relationship between the choice of lens space and the initial map $\psi$ follows from Step 4 of the proof of Proposition \ref{prop:constructcontact}, where we show that $\lambda_{\tilde\psi}$ is defined.

Proposition \ref{prop:constructcontact} is proved via a construction using contact geometry and three-manifold topology. We fix coordinates and notation in the course of the proof, which are referenced heavily later, particularly in the computational section \S\ref{sec:kfech}.

The hypotheses of Proposition \ref{prop:constructcontact} include only those those $\psi$ satisfying several additional assumptions beyond the hypotheses of Theorem \ref{thm:main}. One is that the action function is positive. It is easy to remove this assumption; see Lemma \ref{lem:Nshift}. The other assumption is that the quantities $\tilde p:=y_+-y_-+F$ and $F$ are integers. We weaken this requirement from the integers to the rationals in Lemma \ref{lem:rational}. Fully removing it requires careful perturbations, which we apply in \S\ref{subsec:finalbound} at the very end of the proof of Theorem \ref{thm:main}.

\begin{prop}\label{prop:constructcontact} Let $\psi$ be an area-preserving diffeomorphism of $(A,\omega)$ which is rotation by $2\pi y_\pm$ near $\partial_\pm A$, whose flux is $F\in\Z$, for which $y_+-y_-\in\Z$, both $y_+$ and $-y_-+F$ are irrational, and whose action function $f$ is positive. Then there is a contact form $\lambda_{\tilde\psi}$ on $L(\tilde p,\tilde p-1)$ for which
\begin{enumerate}
\item An open book decomposition $(B_{\tilde p},P_{\tilde p})$ of $L(\tilde p,\tilde p-1)$ with abstract open book $(A,D_{\tilde p})$ is adapted to $\lambda_{\tilde\psi}$. Let $A_0$ denote the closure of the zero page. The return time of the Reeb flow from $A_0$ to $A_0$ is given by the action function $f$, and $\psi$ is the return map of $(\lambda_{\tilde\psi},B_{\tilde p},P_{\tilde p})$.
\item The binding orbits have action 1, are elliptic, and have rotation numbers $y_+^{-1}$ and $(-y_-+F)^{-1}$ in the trivializations which have linking number zero with their component of $B_{\tilde p}$ with respect to $A_0$.
\item Let $\{|B_{\tilde p}|\}$ denote the set of components of $B_{\tilde p}$. There is a bijection $\P(\psi)\cup\{|B_{\tilde p}|\}\to\P(\lambda_{\tilde\psi})$. The symplectic action of the Reeb orbit $\gamma'$ corresponding to $\gamma\in\P(\psi)$ is $\A(\gamma)$, and the intersection number of $\gamma'$ with the page $A_0$ is $\ell(\gamma)$.
\item $\vol(L(\tilde p,\tilde p-1),\lambda_{\tilde\psi})=2\V(\tilde\psi)$
\end{enumerate}
\end{prop}
\begin{proof} We construct a form $\lambda_0$ on the mapping torus of $\psi$ whose Reeb vector field projects to the $S^1$ direction. We then glue in solid tori to the boundary of the mapping torus to form the lens space, and show that the contact form extends to the closed manifold and satisfies the conclusions of the proposition.

\subsubsection*{Step 1: Mapping torus}

Recall that we have fixed $\beta=\frac{x}{2\pi}\,dy$. Let
\[
M_\psi:=\frac{[0,1]_\theta\times A}{(1,x,y)\sim(0,\psi(x,y))}
\]
denote the mapping torus of $\psi$. Let $\eta_i$ be smooth functions of $\theta$. Let $\lambda_0$ denote the following one-form on $[0,1]\times A$
\begin{equation}\label{eqn:form0}
\lambda_0(\theta,x,y)=\eta_1(\theta)f(x,y)\,d\theta+\eta_2(\theta)f\circ\psi(x,y)\,d\theta+\beta+\eta_3(\theta)\,df+\eta_4(\theta)\,\psi^*df
\end{equation}
Analogous to properties (i)-(vi) in the proof of \cite[Proposition 2.1]{mean}, we'd like
\begin{itemize}
\item[(i)] $\lambda_0(\theta,x,y)=f(x,y)\,d\theta+\beta$ for $(x,y)$ near $\partial A$, which will help us glue in the binding.
\item[(ii)] $\lambda_0\in\Omega^1(M_\psi)$.
\item[(iii)] $\lambda_0$ is contact.
\item[(iv)] $R_0$ is a positive multiple of $\partial_\theta$.
\item[(v)] It takes time $f(x,y)$ to flow under $R_0$ from $(0,x,y)$ to $(1,x,y)$.
\item[(vi)] $\vol(M_\psi,\lambda_0)=\omega(A)\V(\tilde\psi)$.
\end{itemize}

The requirements (i)-(vi) determine the following conditions on the functions $\eta_i$.

In order to satisfy (i), choose $\delta_\pm$ so that $\psi$ is rotation by $2\pi y_\pm$ near $\partial_\pm A$. Since both $df\equiv0$ and $f\circ\psi\equiv f$ near $\partial A$, for $x\leq -1+\delta_-$ or $1-\delta_+\leq x$ we have
\[
\lambda_0(\theta,x,y)=(\eta_1(\theta)+\eta_2)f(x,y)\,d\theta+\beta
\]
therefore we need
\[
\eta_1(\theta)+\eta_2(\theta)=1
\]
for all $\theta$. In light of this we update (\ref{eqn:form0}) to
\begin{equation}\label{eqn:formi}
\lambda_0(\theta,x,y)=(1-\eta_2(\theta))f(x,y)\,d\theta+\eta_2(\theta)f\circ\psi(x,y)\,d\theta+\beta+\eta_3(\theta)\,df+\eta_4(\theta)\,\psi^*df
\end{equation}

To satisfy (ii), we need $\lambda_0$ to descend from $[0,1]\times A$ to $M_\psi$. That is, we need
\begin{equation}\label{eqn:pullback}
\lambda_0(1,x,y)=\psi^*\lambda_0(0,x,y)
\end{equation}
By computing both sides of (\ref{eqn:pullback}) we get
\begin{align}
&1-\eta_2(0)=\eta_2(1)\text{ and }\eta_2(0)=0\text{ and }1-\eta_2(1)=0\label{eqn:eta2}
\\&0=1-\eta_3(1)\label{eqn:beta}
\\&1-\eta_3(0)=\eta_3(1)-\eta_4(1)\label{eqn:psibeta}
\\&\eta_3(0)-\eta_4(0)=\eta_4(1)\label{eqn:psipsibeta}
\\&\eta_4(0)=0\label{eqn:psipsipsibeta}
\end{align}

In order to satisfy (iii), that $\lambda_0$ is contact, we need to check that $\lambda_0\wedge d\lambda_0>0$, where the sign agrees with that of the oriented coordinates $\theta$, $x$, $y$. We compute $d\lambda_0$:
\[
d\lambda_0(\theta,x,y)=(1-\eta_2(\theta))\,df\wedge d\theta+\eta_2(\theta)\,\psi^*df\wedge d\theta+\omega+\eta'_3(\theta)\,d\theta\wedge df+\eta'_4(\theta)\,d\theta\wedge\psi^*df
\]
The clearest way to obtain (iii)-(vi) is to set $d\lambda_0=\omega$, which follows from
\begin{equation}\label{eqn:dhatlambdaisomega}
1-\eta_2(\theta)=\eta_3'(\theta)\text{ and }\eta_2(\theta)=\eta_4'(\theta)
\end{equation}
We can therefore update (\ref{eqn:formi}) to
\begin{equation}\label{eqn:lambdawithc}
\lambda_0(\theta,x,y)=(1-\eta'_4(\theta))f(x,y)\,d\theta+\eta'_4(\theta)f\circ\psi(x,y)\,d\theta+\beta+(c+\theta-\eta_4(\theta))\,df+\eta_4(\theta)\,\psi^*df
\end{equation}
where $c$ is some constant. We can solve for $c$:
\[
1\overset{(\ref{eqn:beta})}{=}\eta_3(1)=c+1-\eta_4(1)\Rightarrow c=\eta_4(1)
\]
from which we update (\ref{eqn:lambdawithc}) to
\begin{equation}\label{eqn:formiii}
\lambda_0(\theta,x,y)=(1-\eta'_4(\theta))f(x,y)\,d\theta+\eta'_4(\theta)f\circ\psi(x,y)\,d\theta+\beta+(\eta_4(1)+\theta-\eta_4(\theta))\,df+\eta_4(\theta)\,\psi^*df
\end{equation}

Because $d\lambda_0=\omega$, we obtain (iv), that $R_0$ is a positive multiple of $\partial_\theta$.

To show that $\lambda_0$ is contact we compute $\lambda_0\wedge d\lambda_0$:
\[
\lambda_0\wedge d\lambda_0(\theta,x,y)=\left((1-\eta_4'(\theta))f(x,y)+\eta_4'(\theta)f\circ\psi(x,y)\right)\,d\theta\wedge\omega
\]
This is positive if
\begin{align}
\nonumber0&\leq(1-\eta_4'(\theta))f(x,y)+\eta_4'(\theta)f\circ\psi(x,y)
\\\nonumber&=f+\eta_4'(\theta)(f\circ\psi-f)
\\-f&<\eta_4'(\theta)(f\circ\psi-f)\label{eqn:hatcontactineq}
\end{align}
The inequality (\ref{eqn:hatcontactineq}) holds even when $f$ achieves its minimum. Moreover, since $f>0$, we have $f\circ\psi-f<\max(f)$. By rewriting (\ref{eqn:hatcontactineq}) when $f$ achieves its minimum, we get
\[
-\min(f)<\eta_4'(\theta)\max(f)\overset{\max(f)>0}{\Leftrightarrow}-\frac{\min(f)}{\max(f)}<\eta_4'(\theta)
\]
However, we can also rewrite (\ref{eqn:hatcontactineq}) as
\[
f>\eta_4'(\theta)(f-f\circ\psi)
\]
which is true if $1\geq\eta_4'(\theta)$, because $f\circ\psi>0$. Therefore we ask that
\begin{equation}\label{eqn:eta'bounds}
-\frac{\min(f)}{\max(f)}<\eta_4'(\theta)\leq1
\end{equation}

For (v) to hold, the Reeb vector field must take time $f(x,y)$ to flow from $(0,x,y)$ to $(1,x,y)$. Because the coefficient of $d\theta$ in $\lambda_0(\theta,x,y)$ is $(1-\eta_4'(\theta))f(x,y)+\eta_4'(\theta)f\circ\psi(x,y)$, the coefficient of $\partial_\theta$ in the Reeb vector field $R_0$ is its reciprocal:
\[
R_0(\theta,x,y)=\frac{1}{(1-\eta_4'(\theta))f(x,y)+\eta_4'(\theta)f\circ\psi(x,y)}\partial_\theta
\]
and the time it takes to flow from $(0,x,y)$ to $(1,x,y)$ is
\begin{align*}
\int_0^1((1-\eta_4'(\theta))f(x,y)&+\eta_4'(\theta)f\circ\psi(x,y))\,d\theta
\\&=f(x,y)-(\eta_4(1)-\eta_4(0))f(x,y)+(\eta_4(1)-\eta_4(0))f\circ\psi(x,y)
\\&\overset{\ref{eqn:psipsipsibeta}}{=}(1-\eta_4(1))f(x,y)+\eta_4(1)f\circ\psi(x,y)
\end{align*}
Therefore, if
\begin{equation}\label{eqn:eta1}
\eta_4(1)=0
\end{equation}
then we get the desired time to flow.

We get the volume required in (vi) without making any further restrictions:
\begin{align*}
\int_{M_\psi}\lambda_0\wedge d\lambda_0&=\int_{M_\psi}\left((1-\eta_4'(\theta))f(x,y)+\eta_4'(\theta)f\circ\psi(x,y)\right)\,d\theta\wedge\omega
\\&=\int_A\left(\int_0^1\left((1-\eta_4'(\theta))f(x,y)+\eta_4'(\theta)f\circ\psi(x,y)\right)\,d\theta\right)\omega
\\&=\int_Af\omega
\\&=\omega(A)\V(\tilde\psi)
\end{align*}

In summary, set $\eta=\eta_4$, which is a function with the properties
\begin{itemize}
\item $\eta(0)=0$ because of (\ref{eqn:psipsipsibeta}) and $\eta'(0)=0$ because of (\ref{eqn:dhatlambdaisomega}). We can achieve both of these conditions by asking for $\eta(\theta)\equiv0$ for $\theta$ near zero.
\item $\eta(1)=0$ because of (\ref{eqn:eta1}) and $\eta'(1)=1$ because of (\ref{eqn:dhatlambdaisomega}). We can achieve both of these conditions by asking for $\eta(\theta)=t-1$ for $\theta$ near one.
\item $\eta'$ is bounded by the inequalities in (\ref{eqn:eta'bounds}).
\end{itemize}
These properties are possible to achieve simultaneously. From $\eta$, we construct the following one-form on $M_{\psi}$, written in coordinates on $[0,1]\times A$:
\begin{equation}\label{eqn:formfinal}
\lambda_0(\theta,x,y)=(1-\eta'(\theta))f\,d\theta+\eta'(\theta)f\circ\psi\,d\theta+\beta+(\theta-\eta(\theta))\,df+\eta(\theta)\,\psi^*df
\end{equation}
$\lambda_0$ satisfies (i)-(vi). We also can check that it satsifies (\ref{eqn:eta2}) and (\ref{eqn:psibeta}), which we have not directly used in the construction so need to verify separately.

\subsubsection*{Step 2: The closed manifold}

Next we glue two solid tori to $\mathring M_\psi$ in order to obtain a closed manifold $Y_{\tilde\psi}$. Because we have assumed that both boundary rotation numbers are irrational, $y$ does not continue as a coordinate on the tori $x=const$ across $\theta=1$, but we will need global coordinates on these tori in order to glue in the solid tori. Let $\partial_\pm M_\psi$ denote the component of $\partial M_\psi$ for which $x=\pm1$, respectively. Near $\partial_\pm M_\psi$, instead of $y$ we use the coordinates
\[
\hat y=y+2\pi y_+\theta\text{ and }\check y=y+2\pi y_-\theta
\]

Let $T_\pm:=(\R/2\pi\Z)_{t_\pm}\times\D^2_{(\rho_\pm,\mu_\pm)}$, where $\rho_\pm\in[0,\delta_\pm]$ and $\mu_\pm\in\R/2\pi\Z$. Let $g_\pm:\mathring M_\psi\to T_\pm$ be given by
\begin{align*}
g_+(x,\hat y,\theta)&=\left(\sqrt{1-x},2\pi\theta,\hat y\right)=(\rho_+,\mu_+,t_+)
\\g_-(x,\theta,\check y)&=(\sqrt{x-(-1)},-F(2\pi\theta)+\check y,2\pi\theta)=(\rho_-,t_-,\mu_-)
\end{align*}

The manifold $Y_{\tilde\psi}$ is the union of $T_\pm$ and $\mathring M_\psi$, where points in $\mathring M_\psi$ are identified with their images in $T_\pm$ under the $g_\pm$.

Notice the order in which we've written the coordinates in the maps $g_\pm$. These agree with the orientations on the various components of $Y_{\tilde\psi}$. On $\mathring M_\psi$, $\theta$, $x$, $y$ are oriented coordinates, which induce the orientation $\hat y$, $\theta$ on the tori $x=const$ near $\partial_+M_\psi$. The gluing $g_+$ from $\mathring M_\psi$ and $T_+$ must send the tori $\rho_+=const$ up with the tori $x=const$, and should change orientation. Therefore the tori $\rho_+=const$ must be oriented by $\mu_+$, $t_+$ in $T_+$. That orientation is induced by the oriented coordinates $\ro_+$, $\mu_+$, $t_+$ on $T_+$. Similarly, the orientation necessary on $T_-$ is $\rho_-$, $t_-$, $\mu_-$.

At this point we could show that $Y_{\tilde\psi}$ is $L(\tilde p,\tilde p-1)$ using either Heegaard splittings or the surgery construction of lens spaces (see \cite[Chapter 2]{os}). However, this identification will follow easily from finding an open book decomposition of $Y_{\tilde\psi}$, which we do in Step 3.

\subsubsection*{Step 3: Open book decomposition}

We show that $Y_{\tilde\psi}-\{\rho_\pm=0\}$ is diffeomorphic to the mapping torus of $\tilde p$ right-handed Dehn twists about the core circle of an annulus, and that this monodromy arises as the return map of the flow of a vector field which is meridional near the circles $\rho_\pm=0$.

Denote by $B_{\tilde p}$ the set $\{\rho_\pm=0\}$. The projection map $P_{\tilde p}$ of the open book decomposition is $(\theta,x,y)\mapsto\theta$. Therefore the pages are annuli.

We'll use coordinates $x,\hat y$ on the zero page of the open book (that is, $\theta=0$). The gluing from $T_+$ to $\partial_+M_\psi$ is given by the map
\[
g_+^{-1}(\rho_+,\mu_+,t_+)=\left(1-\rho_+^2, t_+, \frac{1}{2\pi}\mu_+\right)
\]
and the flow of $\partial_\theta$ is a reparameterization of the flow of $\partial_{\mu_+}$, and so sends $\hat y$ to itself under one full rotation. Therefore the monodromy is the identity on $T_+$.

In $\mathring M_\psi$, the flow of $\partial_\theta$ brings the annulus parameterized by $x,\hat y$ back to itself. Near $\partial_+M_\psi$, the coordinate $\hat y=y+2\pi y_+\theta$ flows from itself to itself, because $M_\psi$ is the mapping torus of $\psi$, which is rotation by $y_+$ on $\partial_+A$. Near $\partial_-M_\psi$, the coordinate which is invariant under the flow by $\theta$ is $\check y=\hat y+2\pi(y_--y_+)\theta$. In other words, the coordinate $\hat y$ flows to a coordinate invariant under the flow plus $y_+-y_-$, meaning that the monodromy of the open book decomposition consists of $y_+-y_-$ right-handed Dehn twists about the core circle of the annulus, composed with the contribution to the monodromy from $T_-$.

Near $T_-$, if we think of the gluing as from $T_-$ to $\partial_-M_\psi$, we get the map
\[
g_-^{-1}(\rho_-,t_-,\mu_-)=\left(\rho_-^2-1,\frac{1}{2\pi}\mu_-,t_-+F\mu_-\right)
\]
Therefore under the flow of $\partial_{\mu_-}$ for time $2\pi$, the coordinate $\check y$ increases by $2\pi F$. Therefore the monodromy consists of $y_+-y_-+F=\tilde p$ Dehn twists about the core circle of the $x$, $\hat y$ annulus. Therefore the open book decomposition $(B_{\tilde p},\Pi_{\tilde p})$ induces the abstract open book $(A,D_{\tilde p})$. As shown in Lemma \ref{lem:stdlens}, $L(\tilde p,\tilde p-1)$ also admits an open book decomposition with abstract open book $(A,D_{\tilde p})$, so $Y_{\tilde\psi}\cong L(\tilde p,\tilde p-1)$.

\subsubsection*{Step 4: Contact geometry}

We extend $\lambda_0$ to a contact form $\lambda_{\tilde\psi}$ on $L(\tilde p,\tilde p-1)$. Near $\partial_+ M_\psi$,
\[
\lambda_0=f(x,y)\,d\theta+\frac{x}{2\pi}\,dy=(f(x,y)-xy_+)\,d\theta+\frac{x}{2\pi}\,d\hat y=\frac{\rho_+^2y_+}{2\pi}\,d\mu_++\frac{1-\rho_+^2}{2\pi}\,dt_+
\]
while near $\partial_-M_\psi$,
\[
\lambda_0=f(x,y)\,d\theta+\frac{x}{2\pi}\,dy=(f(x,y)-xy_-)\,d\theta+\frac{x}{2\pi}\,d\check y=\frac{\rho_-^2(-y_-+F)}{2\pi}\,d\mu_-+\frac{\rho_-^2-1}{2\pi}\,dt_-
\]
By converting to Cartesian coordinates, we see that $d\mu_\pm$ have poles of order one when $\rho_\pm=0$. The coefficient of $d\mu_\pm$ in both of the above reparameterizations of $\lambda_0$ has a zero of order two as $\rho_\pm\to0$. Therefore $\lambda_0$ extends over $\rho_\pm=0$ to a one-form $\lambda_{\tilde\psi}$ on $L(\tilde p,\tilde p-1)$.

In order to check that $\lambda_{\tilde\psi}$ is contact, we compute near $\partial_+M_\psi$ that
\[
d\lambda_{\tilde\psi}=\frac{\rho_+y_+}{\pi}\,d\rho_+\wedge d\mu_+-\frac{\rho_+}{\pi}\,d\rho_+\wedge dt_+\Rightarrow\lambda_{\tilde\psi}\wedge d\lambda_{\tilde\psi}=\frac{y_+}{2\pi^2}\rho_+\,dt_+\wedge d\rho_+\wedge d\mu_+
\]
and near $\partial_-M_\psi$,
\[
d\lambda_{\tilde\psi}=\frac{\rho_-(-y_-+F)}{\pi}\,d\rho_-\wedge d\mu_-+\frac{\rho_-}{\pi}\,d\rho_-\wedge dt_-\Rightarrow\lambda_{\tilde\psi}\wedge d\lambda_{\tilde\psi}=\frac{-y_-+F}{2\pi^2}(-\rho_-\,dt_-\wedge d\rho_-\wedge d\mu_-)
\]

These are both positive multiples of the volume forms $\pm\rho_\pm\,dt_\pm\wedge d\rho_\pm\wedge d\mu_\pm$ on $T_\pm$ (these are the volume forms translated from the standard volume forms in Cartesian coordinates, with signs according to the orientations of $T_\pm$ discussed in Step 2).

Near $\partial_\pm M_\psi$, respectively, the Reeb vector field is
\[
R_{\tilde\psi}=\frac{2\pi}{y_+}\partial_{\mu_+}+2\pi\partial_{t_+}\text{ and }R_{\tilde\psi}=\frac{2\pi}{-y_-+F}\partial_{\mu_-}-2\pi\partial_{t_-}
\]
By converting to Cartesian coordinates, we see that $\partial_{\mu_\pm}\to0$ as $\rho_\pm\to0$. Therefore the circles $\rho_\pm=0$ are the images of Reeb orbits, both of action 1. Denote these orbits by $e_\pm$.

It remains to show that these orbits are nondegenerate and elliptic, and to compute their rotation numbers. We first find oriented bases for $\xi_{\tilde\psi}$ over the $e_\pm$. Note
\[
\lambda_{\tilde\psi}|_{\rho_+=0}=\frac{1}{2\pi}\,dt_+\text{ and }\lambda_{\tilde\psi}|_{\rho_-=0}=-\frac{1}{2\pi}\,dt_-
\]
Therefore $\xi_{\tilde\psi}|_{e_\pm}$ is the bundle of tangent spaces to the disks $t_\pm=const$. Since $\lambda_{\tilde\psi}|_{\rho_+=0}$ orients $e_+$ in the positive $t_+$ direction and $\lambda_{\tilde\psi}\wedge d\lambda_{\tilde\psi}$ is a positive volume form, $d\lambda_{\tilde\psi}$ must give $\xi_{\tilde\psi}|_{e_+}$ the orientation which the disks $t_+=const$ inherit from the orientation $t_+,\rho_+,\mu_+$ on $T_+$. That is the standard orientation $\rho_+,\mu_+$ of the disk. On the negative side, because $\lambda_{\tilde\psi}|_{\rho_-=0}$ orients $e_-$ in the negative $t_-$ direction and $\lambda_{\tilde\psi}\wedge d\lambda_{\tilde\psi}$ is a positive volume form, $d\lambda_{\tilde\psi}$ must give $\xi_{\tilde\psi}|_{e_-}$ the orientation which the disks $t_-=const$ inherit from the orientation $-t_-,\rho_-,\mu_-$ on $T_-$. That is also the standard orientation $\rho_-,\mu_-$ of the disk.

We can now compute the rotation numbers, including signs. In the product trivialization of $\xi_{\tilde\psi}|_{e_\pm}$, the rotation numbers are the coefficients of $\partial_{\mu_\pm}$ in $R_{\tilde\psi}|_{e_\pm}$ divided by the coefficients of $\pm\partial_{t_\pm}$ in $R_{\tilde\psi}|_{e_\pm}$. These are $\frac{1}{y_+}$ and $\frac{1}{-y_-+F}$, respectively. Because these are irrational numbers, the $e_\pm$ are both elliptic.

\subsubsection*{Step 5: The desired properties}

We constructed a contact form $\lambda_{\tilde\psi}$ on a closed oriented three-manifold diffeomorphic to $L(\tilde p,\tilde p-1)$. We computed its Reeb vector field, identified the binding orbits $e_\pm$, showed they were elliptic, and computed their action and rotation numbers.

Throughout, $R_{\tilde\psi}\sim\partial_\theta$ except at $e_\pm$ where $R_{\tilde\psi}\sim\pm\partial_{t_\pm}$, and in Step 1 we computed that the the return time is given by $f$, which is always finite. Therefore $(B_{\tilde p},P_{\tilde p})$ is adapted to $\lambda_{\tilde\psi}$ and $(\lambda_{\tilde\psi},B_{\tilde p},P_{\tilde p})$ has $\psi$ as its return map. In Step 3 we computed that the abstract open book of $(B_{\tilde p},P_{\tilde p})$ is $(A,D_{\tilde p})$.

The bijection $\P(\psi)\cup\{|B_{\tilde p}|\}\to\P(\lambda_{\tilde\psi})$ follows from the fact that away from $\rho_\pm=0$, the Reeb vector field is parallel to $\partial_\theta$. The effect of the bijection on action and period follow from the computation of the return time and return map.

\end{proof}

\section{Review of ECH and the ECH spectral numbers}\label{sec:reviewech}

Embedded contact homology (ECH) is a smooth three-manifold invariant constructed using a contact form. We will use two different filtrations on it to analyze the Reeb dynamics of $(L(\tilde p,\tilde p-1),\lambda_{\tilde\psi})$. In this section we review the construction of ECH as well as the construction of the ECH spectral numbers, which are defined using an action filtration.

\subsection{Embedded Contact Homology}\label{subsec:ech}

Let $Y$ be a closed connected oriented three-manifold with a nondegenerate contact form $\lambda$ and Reeb vector field $R$. An \emph{orbit set} of $R$ is a finite set of pairs $(\alpha_i,m_i)$ where the $\alpha_i$ are distinct simple Reeb orbits, the $m_i$ are positive integers, and their \emph{total homology class} is zero, meaning that
\[
\sum_im_i[\alpha_i]=0\in H_1(Y;\Z)
\]
(ECH is can be defined in more general total homology classes; see \cite{echlec}.) For a $\lambda$-compatible (see Definition \ref{defn:lcacs}) almost-complex structure $J$ on $\R\times Y$ we define the \emph{embedded contact homology chain complex} $ECC_*(Y,\lambda,J)$ to be the $\Z_2$-vector space generated by orbit sets $\{(\alpha_i,m_i)\}$ for which $m_i=1$ whenever $\alpha_i$ is hyperbolic (it is possible to extend to $\Z$ coefficients, which we do not need to do; again see \cite{echlec}). We use multiplicative notation for orbit sets, e.g., $\{(\alpha_i,m_i)\}=\alpha_1^{m_1}\cdots\alpha_n^{m_n}$, and additive notation for the group operation in the chain groups, e.g. $x=x_1+\cdots+x_n$ for $x\in ECC_*(Y,\lambda,J)$, where the $x_i$ are generators.

For $\alpha=\{(\alpha_i,m_i)\}$, let $H_2(Y,\alpha,\emptyset)$ denote the set of 2-chains $Z$ for which
\[
\partial Z=\sum_im_i\alpha_i
\]
modulo those which are boundaries of 3-chains. Note $H_2(Y,\alpha,\emptyset)$ is affine over $H_2(Y;\Z)$.

Recall that $\xi=\ker\lambda$ is a 2-plane bundle. If $c_1(\xi)\in H^2(Y;\Z)$ is torsion, then by \cite[Proposition 1.6]{indexineq}, $ECC(Y,\lambda,J)$ is absolutely $\Z$-graded by the \emph{ECH index}
\[
I(\alpha)=c_\tau(Z_\alpha)+Q_\tau(Z_\alpha)+\sum_i\sum_{k=1}^{m_i}CZ_\tau(\alpha_i^k)
\]
where $Z_\alpha$ is any element of $H_2(Y,\alpha,\emptyset)$, $\tau$ is a trivialization of $\alpha^*\xi$, and $\alpha_i^k$ denotes the connected $k$-fold cover of the simple orbit $\alpha_i$. We recall the definitions of the components $c_\tau$, $Q_\tau$, and $CZ_\tau$ of the ECH index here. See \cite[\S2]{indexineq} for further explanation.


We denote by $c_\tau(Z_\alpha)$ the \emph{relative first Chern class} of $\xi|_{Z_\alpha}$ with respect to $\tau$. It is determined by choosing a surface $S$ representing $Z_\alpha$, a section of $\xi$ over $\partial S$ constant with respect to $\tau$, extending this section over $S$, and counting its zeroes with sign. In the computations in Section \ref{subsec:computeech}, we will use the following formula for $c_{\tau'}$ in terms of $c_\tau$, where $\tau'$ is another trivialization of $\alpha^*\xi$. Let $\tau_i,\tau'_i$ denote the restrictions of $\tau,\tau'$ to $\alpha_i$. Then
\begin{equation}\label{eqn:ctauprime}
c_\tau(Z_\alpha)-c_{\tau'}(Z_\alpha)=\sum_im_i(\tau'_i-\tau_i)
\end{equation}
Let $\overset{h.e.}{\sim}$ denote homotopy equivalence. We use $\tau'_i-\tau_i$ to denote the degree of the map
\[
\tau_i\circ{\tau'_i}^{-1}:\alpha_i\to Sp(2,\R)\overset{h.e.}{\sim}S^1
\]

We denote by $Q_\tau(Z_\alpha)$ the \emph{relative self-intersection number} of $Z_\alpha$. Let $\Sigma$ be a surface in $[0,1]\times Y$ for which
\begin{enumerate}
\item $\partial\Sigma$ is a union of $m_i$ positively oriented one-fold covers of $\{1\}\times\alpha_i$.
\item $\pi_Y(\Sigma)$ represents $Z_\alpha$.
\item $\Sigma$ is embedded on the interior of $[0,1]\times Y$, transverse to the boundary, and $\pi_Y|_S$ is an immersion near $\partial S$.
\item Near $\{1\}\times Y$, the projection of $S$ to a plane transverse to $\alpha_i$ consists of a union of rays which do not intersect and do not rotate with respect to $\tau$ as we traverse $\alpha_i$.
\end{enumerate}
Let $\Sigma,\Sigma'$ be two such surfaces which do not intersect near $\{1\}\times Y$. Then $Q_\tau(Z_\alpha)$ is the signed count of the intersections of $\Sigma$ and $\Sigma'$ (which necessarily arise only in their interiors). We again have a change-of-trivialization formula
\begin{equation}\label{eqn:Qtauprime}
Q_\tau(Z_\alpha)-Q_{\tau'}(Z_\alpha)=\sum_im_i^2(\tau'_i-\tau_i)
\end{equation}

Let $\gamma:\R/T\Z\to Y$ be a periodic orbit of $R$ and $\tau$ a trivialization of $\gamma^*\xi$. We denote by $CZ_\tau(\gamma)$ the \emph{Conley-Zehnder index} of $\gamma$ with respect to $\tau$. For all $t\in[0,T]$ there is a symplectic linear map $\phi_t:\xi|_{\gamma(0)}\to\xi|_{\gamma(t)}$ given by the restriction of the differential of the time $t$ Reeb flow to the contact planes. Under conjugation by $\tau$, the  $\phi_t$ trace out a path in $Sp(2,\R)$. Because $\lambda$ is nondegenerate, this path ends at a matrix which does not have $\pm1$ as an eigenvalue. Thus it has a well-defined intersection number with the \emph{Maslov cycle}, the symplectic matrices which do have $\pm1$ as an eigenvalue (in $Sp(2,\R)$ these are the parabolic matrices). This intersection number is the \emph{Conley-Zehnder index} of $\gamma$. If $\gamma$ is a simple elliptic orbit then the path $\phi_t$ is homotopic to rotation by a family of angles $2\pi\theta_t$, where $\theta_T=\rot_\tau(\gamma)$. Via the picture of $Sp(2,\R)$ as an open solid torus in \cite{segal}, we can directly compute the intersection number of this family with the Maslov cycle to obtain
\[
CZ_\tau(\gamma^k)=2\lfloor k\theta\rfloor+1
\]
If $\gamma$ is hyperbolic and $v$ is an eigenvector of $\phi_T$, then $\{\phi_t(v)\}$ is a family of vectors in $\R^2$ which rotate by angle $\pi k$ for some integer $k$, where $k$ is even when $\gamma$ is positive hyperbolic, and odd when $\gamma$ is negative hyperbolic. Again via the open solid torus model of $Sp(2,\R)$, we can explicitly compute the intersection number with the Maslov cycle to obtain
\[
CZ_\tau(\gamma)=k
\]

Finally we also have a change-of-trivialization formula for $CZ_\tau(\gamma)$, when $\gamma$ is simple:
\begin{equation}\label{eqn:CZtauprime}
CZ_\tau(\gamma^k)-CZ_{\tau'}(\gamma^k)=2k(\tau-\tau')
\end{equation}

Next we turn to the definition of the differential in ECH.

\begin{defn}\label{defn:lcacs} An almost-complex structure $J$ on $\R\times Y$ is \emph{$\lambda$-compatible} if it is $\R$-invariant, $J\xi=\xi$, rotating positively with respect to $d\lambda$, and $J\partial_s=R$, where $s$ is the $\R$ coordinate.
\end{defn}

\begin{defn}\label{defn:Jholomcurrent} A \emph{$J$-holomorphic current from $\alpha$ to $\beta$} is a finite set of pairs $(C_k,d_k)$ where the $C_k$ are distinct irreducible somewhere-injective $J$-holomorphic curves whose positive ends are at covers of the $\alpha_i$, whose negative ends are at covers of the $\beta_j$, the sum over $k$ of $d_k$ times the covering multiplicity of $\alpha_i$ ($\beta_j$) equals $m_i$ ($n_j$), and there are no other ends.
\end{defn}

Denote by $\M^J(\alpha,\beta)$ the set of $J$-holomorphic currents from $\alpha$ to $\beta$. Notice that there is an $\R$ action on $\M^J(\alpha,\beta)$ by postcomposing the maps parameterizing each component curve by translation in the $\R$ direction. The coefficient of $\beta$ in $\partial\alpha$ is given by
\[
\langle\partial\alpha,\beta\rangle:=\sum_{I(\alpha)-I(\beta)=1}\#\frac{\M^J(\alpha,\beta)}{\R}
\]
where $\#$ denotes the mod 2 count. That this count is finite is shown in \cite{echlec}; that $\partial^2=0$ is shown in \cite{gluing}, and that the resulting homology depends only on $Y$ and $\xi$ rather than on $\lambda$ and $J$ is shown in \cite{indepform}. Therefore we denote the homology of $(ECC_*(Y,\lambda,J),\partial)$ by
\[
ECH_*(Y,\xi)
\]
We call $ECH_*(Y,\xi)$ the \emph{embedded contact homology}, or \emph{ECH}, of $(Y,\xi)$.

\subsection{The ECH spectral numbers}\label{subsec:spectrum}

The ECH spectral numbers allow us to understand the relationship between symplectic action and ECH homology classes. The symplectic action of an orbit set is the linear combination of the actions of its component orbits, that is,
\[
\A(\{(\alpha_i,m_i)\}):=\sum_im_i\int_{\alpha_i}\lambda
\]
If there is a $J$-holomorphic current from $\alpha$ to $\beta$ for any $\lambda$-compatible $J$, then by Stokes' theorem, $\A(\alpha)\geq\A(\beta)$, with equality only if $\alpha=\beta$. In particular, if $\langle\partial\alpha,\beta\rangle\neq0$, then the action of $\alpha$ is greater than that of $\beta$. Therefore, for any $L\in\R$, the span of the generators $\alpha$ with $\A(\alpha)<L$ is a subcomplex, denoted by $ECC_*^L(Y,\lambda,J)$. Its homology is called the \emph{(action-)filtered ECH}, and it is shown in \cite{arnoldchord} that this homology is independent of $J$, though it does depend on the contact form $\lambda$ rather than just on the contact structure $\xi$.

There are maps induced by inclusion of chain complexes
\begin{equation}\label{eqn:L}
ECH_*^L(Y,\lambda)\to ECH_*(Y,\xi)
\end{equation}
which are also shown in \cite{arnoldchord} to be independent of $J$.

Let $\sigma$ be a nonzero class in $ECH_*(Y,\xi)$. Define the \emph{ECH spectral number of $\sigma$} as
\[
c_\sigma(Y,\lambda):=\inf\{L\;|\;\sigma\text{ is in the image of }ECH_*^L(Y,\lambda)\to ECH_*(Y,\xi)\}
\]
That is, $c_\sigma(Y,\lambda)=L_0$ when there is a cycle $x=\sum_ix_i$ representing $\sigma$ for which each orbit set $x_i$ has $\A(x_i)\leq L_0$, and $L_0$ is the least positive number with this property.

The following theorem is a special case of the original theorem.
\begin{thm}[Cristofaro-Gardiner-Hutchings-Ramos, {\cite[Theorem 1.3]{asymp}}]\label{thm:asymp} Let $Y$ be a closed connected three-manifold with contact form $\lambda$ and contact structure $\xi=\ker\lambda$ with $c_1(\xi)$ torsion. Let $I$ denote the absolute $\Z$-grading on $ECH_*(Y,\xi)$ by the ECH index. Let $\{\sigma_k\}_{k\geq1}$ be a sequence of nonzero homogeneous classes in $ECH_*(Y,\xi)$ with $\lim_{k\to\infty}I(\sigma_k)=\infty$. Then
\[
\lim_{k\to\infty}\frac{c_{\sigma_k}(Y,\lambda)^2}{I(\sigma_k)}=\vol(Y,\lambda)
\]
\end{thm}

We can use \cite[Theorem 1.3]{asymp} to obtain Reeb orbits in infinitely many ECH homology classes for which we have a bound on symplectic action.

\section{The knot filtration on ECH}\label{sec:kfech}

The previous section allowed us to discuss the action of periodic orbits of a Reeb vector field, which will help us understand the total action of the periodic orbits of $\psi$. In order to obtain a bound on the periods of these orbits, we need to investigate the intersection number of Reeb orbits with a page of an open book decomposition. We will generalize of the knot filtration defined by Hutchings in \cite{mean} for contact manifolds $Y$ with $H_1(Y)=0$ to contact manifolds $Y$ with $b_1(Y)=0$. To apply the filtration to annulus maps we will compute the embedded contact homology of the lens space $L(p,p-1)$ with its contact form from Lemma \ref{lem:stdlens} filtered by the sum of the filtrations by the binding components of the open book decomposition from Proposition \ref{prop:constructcontact}.

\subsection{Construction of the knot filtration}\label{subsec:constructkfech}

Let $Y$ be a closed oriented three manifold with $b_1(Y)=0$ and a contact form $\lambda$. Let $J$ be a $\lambda$-compatible almost-complex structure on $\R\times Y$ as above. Given an elliptic Reeb orbit $B$, there is a filtration on the ECH chain complex by a modified version of the linking number of each orbit set with $B$. Unlike the action filtration, this knot filtration is independent of the contact form up to the rotation number of $B$.

In the case when the chosen orbit is the binding of an open book decomposition, there is a related filtration on contact homology (see \cite{reebobd}).

Let $B$ be a simple elliptic Reeb orbit of $\lambda$ and let $B^p$ be a nullhomologous cover of $B$. Let $p\rot(B)\in\R-\Q$ be the rotation number of $B^p$ in the trivialization which has linking number zero with $B^p$ with respect to $\Sigma_{B^p}$, where $\Sigma_{B^p}$ is a Seifert surface for $B^p$. That is,
\[
\rot(B):=\frac{1}{p}\rot_{\tau{\Sigma_{B^p}}}(B^p)
\]
Note that when $p>1$, $\rot(B)$ is not truly a rotation number computed in an honest trivialization. Define the function
\[
\F_B(B^m\alpha):=m\rot(B)+lk(B,\alpha)
\]
for $\alpha$ any orbit set of $\lambda$ not including $B$.

\begin{lem}\label{lem:pffiltration} $\F_B$ is a filtration on $ECC_*(Y,\lambda,J)$ for any $\lambda$-compatible $J$. That is, the differential decreases $\F_B$.
\end{lem}
\begin{proof} The proof is similar to the proof of \cite[Lemma 5.1]{mean}, with adjustments as necessary to account for the role played by covers of $B$.

Let $\alpha_\pm$ be orbit sets of $(Y,\lambda)$ not including $B$. Assume that there is a $J$-holomorphic current $\mathcal{C}\in\M^J(B^{m_+}\alpha_+,B^{m_-}\alpha_-)$. Because the definition of $\F_B$ is linear in the orbits, we can assume that $\mathcal{C}$ consists of one irreducible somewhere injective $J$-holomorphic curve $C$ which is not the trivial cylinder $\R\times B$. Our goal is to show that
\begin{equation}\label{eqn:filtineq}
\F_B(B^{m_+}\alpha_+)\geq\F_B(B^{m_-}\alpha_-)
\end{equation}

From \cite[Corollary 2.6]{cylasymptotics} we get that when $s_0>0$ is a large enough, $C$ is transverse to $\{\pm s_0\}\times Y$. From \cite[Corollary 2.5]{cylasymptotics} we get that $C\cap\R\times B\cap((-\infty,-s_0]\times Y)$ and $C\cap\R\times B\cap([s_0,\infty)\times Y)$ are both empty.

For $s_0$ large enough in the above sense, let $\eta_\pm$ denote $C\cap\{\pm s_0\}\times Y$. Notice that $\eta_\pm$ splits as a pair of links: one which approaches $\alpha_\pm$ as $s_0\to\infty$ and one which approaches $B^{m_\pm}$ as $s_0\to\infty$. Choose $s_0$ large enough so that the link approaching $\alpha_\pm$ has linking number $lk(B,\alpha_\pm)$ with $B$. Call the other link $\zeta_\pm$. Then we have
\begin{equation}\label{eqn:linkingsum}
lk(B,\eta_\pm)=lk(B,\zeta_\pm)+lk(B,\alpha_\pm)
\end{equation}

Let $\zeta_+^1,\dots,\zeta_+^k$ denote the components of $\zeta_+$. Let $q_i$ denote the covering multiplicity of $\zeta_+^i$; in particular, $\sum_{i=1}^kq_i=m_+$. By \cite[Lemma 5.3 (b)]{echlec} and the discussion following its proof,
\[
lk(B,\zeta_+^i)=\frac{1}{p}\text{wind}_{\tau_{\Sigma_{B^p}}}(\zeta_+^i)\leq\frac{1}{p}\lfloor q_ip\rot(B)\rfloor=\lfloor q_i\rot(B)\rfloor
\]
therefore
\begin{equation}\label{eqn:zetapluslk}
lk(B,\zeta_+)=\sum_{i=1}^klk(B,\zeta_+^i) \leq\sum_{i=1}^k\lfloor q_i\rot(B)\rfloor \leq\sum_{i=1}^kq_i\rot(B) =m_+\rot(B)
\end{equation}
and similarly
\begin{equation}\label{eqn:zetaminuslk}
lk(B,\zeta_-)\geq m_-\rot(B)
\end{equation}

Finally,
\begin{equation}\label{eqn:intpos}
lk(B,\eta_+)-lk(B,\eta_-)=\#C\cap(\R\times B)\geq0
\end{equation}
where the inequality comes from intersection positivity for $J$-holomorphic curves. Therefore, by combining (\ref{eqn:intpos}), (\ref{eqn:linkingsum}), (\ref{eqn:zetapluslk}), and (\ref{eqn:zetaminuslk}), we obtain (\ref{eqn:filtineq}) as desired.
\end{proof}

Let $\ell\in\R$. Denote by
\[
ECH_*^{\F_B\leq\ell}(Y,\lambda,J)
\]
the homology of the subcomplex of $ECC_*(Y,\lambda,J)$ generated by admissible orbit sets $B^m\alpha$ with $\F_B(B^m\alpha)\leq\ell$. We call $ECH_*^{\F_B\leq\ell}(Y,\lambda,J)$ the \emph{knot filtered embedded contact homology of $(Y,\lambda)$}, or the \emph{embedded contact homology of $(Y,\lambda)$ filtered by $B$} when we want to emphasize the elliptic orbit $B$. We will show that this homology is an invariant of the contact structure up to contactomorphism, knot transverse to the contact structure, and its rotation number, and does not depend on the contact form.

However, we will need invariance of a slightly different filtration of the embedded contact homology of $(L(\tilde p,\tilde p-1),\lambda_{\tilde\psi})$. We are interested in putting a lower bound on $\alpha\cdot A_0$. As a sum of filtrations, $\F_{e_+}+\F_{e_-}$ is a filtration as well, and we will show in the proof of Proposition \ref{prop:pairineqscontact} that for an orbit set $\alpha$ not including either binding component $e_\pm$ of the open book decomposition $(B_{\tilde p},P_{\tilde p})$ of $L(\tilde p,\tilde p-1)$, we have $\F_{e_+}(\alpha)+\F_{e_-}(\alpha)=\alpha\cdot A_0$.  Therefore, we will compute $ECH_*^{\F_{e_+}+\F_{e_-}\leq\ell}(L(\tilde p,\tilde p-1),\lambda_{\tilde\psi},J)$. In order to do so, we need the following theorem.

\begin{thm}\label{thm:sumkffunctoriality} Let $Y$ be a closed oriented three-manifold with $b_1(Y)=0$, a contact structure $\xi$, a pair of knots $B_\pm$ transverse to $\xi$, and contact forms $\lambda$ and $\lambda'$ for $\xi$ for both of which $B_\pm$ are elliptic Reeb orbits. Let $\rot(B_\pm)\in\R-\Q$ and let $\ell\in\R$. Assume that for the same integer $p$, the orbits $B_\pm^p$ are nullhomologous and have rotation numbers $p\rot(B_\pm)$ for both $\lambda$ and $\lambda'$ in the trivializations which have linking number zero with $B_\pm^p$ with respect to $\Sigma_{B_\pm^p}$. Then
\[
ECH_*^{\F_{B_+}+\F_{B_-}\leq\ell}(Y,\lambda,J)=ECH_*^{\F_{B_+}+\F_{B_-}\leq\ell}(Y,\lambda',J')
\]
for any $\lambda$-compatible $J$ and $\lambda'$-compatible $J'$.
\end{thm}

In light of this theorem we are justified in using the notation
\[
ECH_*^{\F_{B_+}+\F_{B_-}\leq\ell}(Y,\xi,B_+,B_-,\rot(B_+),\rot(B_-))=ECH_*^{\F_{B_+}+\F_{B_-}\leq\ell}(Y,\lambda,J)
\]
where $\xi$ is any contact structure contactomorphic to $\ker\lambda$.

By very similar methods, we can also prove the following theorem, though we will not need it to prove Theorem \ref{thm:main}.

\begin{thm}\label{thm:kffunctoriality} Let $Y$ be a closed oriented three-manifold with $b_1(Y)=0$, contact structure $\xi$, a knot $B$ transverse to $\xi$, and contact forms $\lambda$ and $\lambda'$ for $\xi$, for which $B$ is an elliptic Reeb orbit. Let $\rot(B)\in\R-\Q$ and let $\ell\in\R$. Assume that for some integer $p$, the orbit $B^p$ is nullhomologous and has rotation number $p\rot(B)$ for both $\lambda$ and $\lambda'$, in the trivialization which has linking number zero with $B^p$ with respect to $\Sigma_{B^p}$. Then
\[
ECH^{\F_B\leq\ell}_*(Y,\lambda,J)=ECH^{\F_B\leq\ell}(Y,\lambda',J')
\]
for any $\lambda$-compatible $J$ and $\lambda'$-compatible $J'$.
\end{thm}

\begin{proof}

Theorems \ref{thm:sumkffunctoriality} and \ref{thm:kffunctoriality} are proved in the same manner as is Theorem \cite[Theorem 5.3]{mean} in \cite[\S6, \S7]{mean}. The idea of the proof is the following: in \cite[Proposition 6.2]{mean}, Hutchings proves the existence of chain maps between the ECH chain complexes of the contact manifolds at either end of an exact symplectic cobordism. Furthermore, if the image of some chain $\alpha$ under this chain maps are nonzero, then there must be a holomorphic current in the cobordism between $\alpha$ and its image under the cobordism map.

In the situation described in Theorem \cite[Theorem 5.3]{mean} as well as Theorem \ref{thm:kffunctoriality}, we can build a symplectic cobordism on $\R\times Y$ which admits these chain maps. The chain maps do not increase $\F_B$ because of intersection positivity between the holomorphic currents between sources and targets and the curve $\R\times B$. Therefore, these cobordism maps descend to ECH filtered by $B$. To prove Theorem \ref{thm:sumkffunctoriality}, we simply replace $\R\times B$ with $\R\times B_+\cup\R\times B_-$.

In \cite[Proposition 6.2]{mean} and in \cite[\S7]{mean} Hutchings also shows that if the ends of the cobordism are contactomorphic then the composition of the two resulting cobordism maps is chain homotopic to the identity, and this chain homotopy also preserves the knot filtration. Therefore the cobordism maps must both be isomorphisms which preserve the knot filtration.

\end{proof}

Notice that both Theorems \ref{thm:sumkffunctoriality} and \ref{thm:kffunctoriality} extend to the case when we consider two contactomorphic contact structures on $Y$ and two knots or pairs of knots, so long as the diffeomorphism of $Y$ which realizes the contactomorphism between the contact structures also identifies the knots and preserves their rotation numbers.

\subsection{Computation of $ECH_*(L(\tilde p,\tilde p-1),\xi_{\tilde\psi})$}\label{subsec:computeech}

We need to compute $ECH_*(L(\tilde p,\tilde p-1),\xi_{\tilde\psi})$ in order to apply \cite[Theorem 1.3]{asymp}. Our approach will also illuminate how to compute the ECH of $L(\tilde p,\tilde p-1)$ filtered by $\F_{e_+}+\F_{e_-}$.

\begin{prop}\label{prop:modelform} If $\frac{y_+}{-y_-+F}\in\R-\Q$, there is a nondegenerate contact form $\lambda_{\tilde p}$ on $L(\tilde p,\tilde p-1)$, where $\tilde p=y_+-y_-+F$, which satisfies
\begin{enumerate}
\item $\ker\lambda_{\tilde p}$ and $\ker\lambda_{\tilde\psi}$ are contactomorphic.
\item The orbits $e_\pm$ of $\lambda_{\tilde\psi}$ are both also simple nondegenerate elliptic Reeb orbits for forms $\lambda_{\tilde p}$, and it has no other simple Reeb orbits.
\item The nullhomologous covers $e_\pm^{\tilde p}$ of $e_\pm$ have rotation numbers $\frac{\tilde p}{y_+}-1$ and $\frac{\tilde p}{-y_-+F}-1$ when computed in the trivializations of $\ker\lambda_{\tilde p}$ which have linking number zero with $e_\pm^{\tilde p}$ with respect to their Seifert surfaces $S_\pm$.
\end{enumerate}
\end{prop}

First we check that the rotation numbers in the third conclusion above proposition are the same as those of $\lambda_{\tilde\psi}$, as computed in Proposition \ref{prop:constructcontact}.
\begin{lem}\label{lem:computerotpm} The rotation numbers of $e_\pm^{\tilde p}$ in the trivializations of $\ker\lambda_{\tilde\psi}$ which have linking number zero with $e_\pm^{\tilde p}$ with respect to their Seifert surfaces are $\frac{\tilde p}{y_+}-1$ and $\frac{\tilde p}{-y_-+F}-1$.
\end{lem}
\begin{proof} Recall the solid tori $T_\pm$ from the proof of Proposition \ref{prop:constructcontact}. Let $\nu_\pm$ denote the meridian $t_\pm=0$ on the boundary of the solid torus $T_\pm$, oriented with $\partial_{\mu_\pm}$. Let $D_\pm=\{t_\pm=0\}$ be the disks bounded by $\nu_\pm$, oriented so that $\nu_\pm=\partial D_\pm$. In $T_+$, let
\[
\Sigma_+:=\{(t_+,\rho_+,\mu_+)\in T_+\;|\;\tilde p\mu_+=t_+\}
\]
Orient $\Sigma_+$ so that $\partial\Sigma_+=e_+^{\tilde p}\sqcup-T_{1,\tilde p}$, where where $T_{1,\tilde p}$ is the $(1,\tilde p)$ torus knot on $\partial T_+$. The meridian $\nu_-$ of the core circle of $T_-$ is glued to $T_{1,\tilde p}$ to form $L(\tilde p,\tilde p-1)$. By gluing $D_-$ along $\nu_-$ to $\Sigma_+$, we obtain a Seifert surface for $e_+^{\tilde p}$. Similarly, in $T_-$, let
\[
\Sigma_-:=\{(t_-,\mu_-,\rho_-)\in T_-\;|\;\tilde p\mu_-=-t_-\}
\]
By similar reasoning, $\Sigma_-\cup_{\nu_+}D_+$ is a Seifert surface for $e_-^{\tilde p}$.

We can now compute the desired rotation numbers of $e_\pm^{\tilde p}$ by using the trivializations which have linking number zero with $e_\pm^{\tilde p}$ with respect to $\Sigma_\pm$. We know from Proposition \ref{prop:constructcontact} that the rotation numbers of $e_\pm^{\tilde p}$ in the trivializations of $\ker\lambda_{\tilde\psi}$ which have linking number zero with $e_\pm^{\tilde p}$ with respect to $\tilde pA_0$ are $\frac{\tilde p}{y_+}$ and $\frac{\tilde p}{-y_-+F}$. In $T_\pm$ coordinates, $A_0$ is the surface $\mu_\pm=0$. In one full flow about $e_\pm^{\tilde p}$, the trivializations which have linking number zero with $e_\pm^{\tilde p}$ with respect to $\tilde pA_0$ twist exactly once less in the $\partial_{\mu_\pm}$ direction than do the trivializations which have linking number zero with $e_\pm^{\tilde p}$ with respect to $\Sigma_\pm$. This is because the corresponding surfaces rotate exactly once less, as shown in Figure \ref{fig:rot+} for $\Sigma_+$. Therefore,
\[
\rot_{\tau_{\Sigma_\pm\cup D_\mp}}\left(e_\pm^{\tilde p}\right)=\rot_{\tau_{\tilde pA_0}}\left(e_\pm^{\tilde p}\right)-1
\]

\begin{figure}
\centering
\includegraphics{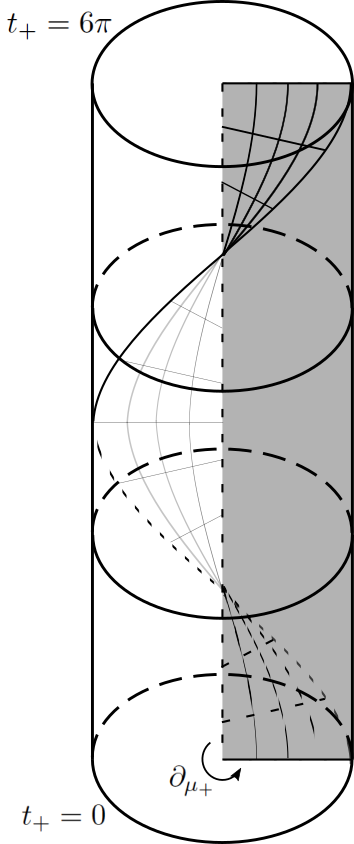}
\caption{A neighborhood of $e_+^3$ in the triple cover of $T_+$. The positive $\mu_+$ direction is counterclockwise in the $t_+=0$ circle. The grey surface is $3A_0$ and the mesh surface is $\Sigma_+$.}
\label{fig:rot+}
\end{figure}
\end{proof}

From now on we will use the following notation (compare the notation for $\rot(B)$ in the construction of the knot filtration in \S\ref{subsec:constructkfech}):
\[
\rot(e_+):=\frac{1}{y_+}-\frac{1}{\tilde p}\text{ and }\rot(e_-):=\frac{1}{-y_-+F}-\frac{1}{\tilde p}
\]

\begin{proof} \emph{(of Proposition \ref{prop:modelform})}
Let $E(a,b)$ denote the ellipsoid
\[
E(a,b):=\left\{(r_1,\theta_1,r_2,\theta_2)\in\C^2\;\middle|\;\frac{\pi}{a}r_1^2+\frac{\pi}{b}r_2^2\leq1\right\}
\]
Let $\mathfrak{q}_{\tilde p}:\partial E(a,b)\to L(\tilde p,\tilde p-1)$ be the quotient map as in \S\ref{subsec:stdlens}. Finally let $\lambda_{(a,b)}$ denote the restriction of the standard 1-form on $\R^4$ given on $\C^2-\{0\}$ by
\[
\frac{1}{2}r_1^2\,d\theta_1+\frac{1}{2}r_2^2\,d\theta_2
\]
to $\partial E(a,b)$. Notice that $\lambda_{(a,b)}$ has simple Reeb orbits
\[
\gamma^{1,2}_{(a,b)}:=\{r_{2,1}=0\}\cap\partial E(a,b)
\]
which, as computed in \cite{mean}, have actions $a$ and $b$, respectively, and rotation numbers $\frac{a}{b}$ and $\frac{b}{a}$, respectively, in their trivializations which have linking number zero with $\gamma^{1,2}$ with respect to their Seifert surfaces.

Let $\lambda_{\tilde p}$ on $L(\tilde p,\tilde p-1)$ be the contact form uniquely determined by
\[
\mathfrak{q}_{\tilde p}^*\lambda_{\tilde p}=\lambda_{(1,b_+)}
\]
as in the proof of Lemma \ref{lem:stdlens}, where
\[
b_+:=\frac{y_+}{-y_-+F}
\]

Because every orbit of $\lambda_{\tilde p}$ is covered $\tilde p\in\Z$ times by an orbit of $\lambda_{(1,b_+)}$, the former is nondegenerate precisely when the latter is, which is when $\frac{y_+}{-y_-+F}$ is irrational.

Next we check the conclusions of the proposition.
\begin{enumerate}
\item Because
\[
R_{(a,b)}=2\pi\left(\frac{1}{a}\partial_{\theta_1}+\frac{1}{b}\partial_{\theta_2}\right)
\]
$\lambda_{(a,b)}$ is adapted to the open book decomposition $\Pi_1$ of $\partial E(a,b)$. Because $\Pi_1$ factors through the open book decomposition $(H_{\tilde p},\Pi_{\tilde p})$ of $L(\tilde p,\tilde p-1)$ as shown in Lemma \ref{lem:stdlens}, $\lambda_{\tilde p}$ is adapted to open book decomposition which induces the abstract open book $(A,D_{\tilde p})$. Therefore $\ker\lambda_{\tilde p}$ and $\ker\lambda_{\tilde\psi}$ are contactomorphic by Theorem \ref{thm:aobtocontacto}.

\item The Reeb orbits of $\lambda_{\tilde p}$ are covered $\tilde p$ times by the Reeb orbits of $\lambda_{(1,b_+)}$. Because $e_\pm^{\tilde p}$ are the only simple Reeb orbits of the latter contact form, $e_\pm$ are the only simple Reeb orbits of the former contact form. Nondegeneracy also descends: the rotation numbers of the forms on the lens space can only differ from the rotation numbers of the forms on the ellipsoid boundary by the addition or subtraction of integers and then division by $\tilde p\in\Z$, operations which preserve irrationality.

\item Both $e_\pm^{\tilde p}=\mathfrak{q}_{\tilde p}(\gamma^{1,2}_{(1,b_+)})$ and $S_\pm$ can be represented by surfaces homologous to the images of the Siefert surfaces of $\gamma^{1,2}_{(1,b_+)}$ under $\mathfrak{q}_{\tilde p}$. Therefore, the rotation numbers of $e_\pm^{\tilde p}$ in the trivializations of $\ker\lambda_{\tilde p}$ with linking number zero with respect to $S_\pm$ are the same as those of $\gamma^{1,2}_{(1,b_+)}$ in the trivializations of $\ker\lambda_{(1,b_+)}$ with linking number zero with respect to their Seifert surfaces. These are
\begin{align*}
\rot_{S_+}(e_+^{\tilde p})&=\frac{1}{b_+}=\frac{-y_-+F}{y_+}=\frac{\tilde p-y_+}{y_+}=\frac{\tilde p}{y_+}-1
\\\rot_{S_-}(e_-^{\tilde p})&=\frac{b_+}{1}=\frac{y_+}{-y_-+F}=\frac{\tilde p-(-y_-+F)}{-y_-+F}=\frac{\tilde p}{-y_-+F}-1
\end{align*}

\end{enumerate}
\end{proof}

We will denote $\ker\lambda_{\tilde p}$ by $\xi_{\tilde p}$; it is contactomorphic to $\ker\lambda_{\tilde\psi}$ by Theorem \ref{thm:contactomorphic}. Therefore, we can compute the ECH of $(L(\tilde p,\tilde p-1),\lambda_{\tilde\psi})$ using $\lambda_{\tilde p}$ instead.

The generators of the chain complex for $\lambda_{\tilde p}$ are orbit sets of the form $e_+^{m_+}e_-^{m_-}$, with $m_+-d\tilde p=m_-$ for some $d\in\Z$. Because $H_2(L(\tilde p,\tilde p-1);\Z)=0$, there is a unique element $Z_{(m_+,d)}$ of $H_2(L(\tilde p,\tilde p-1),e_+^{m_+}e_-^{m_+-d\tilde p},\emptyset)$. Therefore there is a natural $\Z$ grading on $ECC_*(L(\tilde p,\tilde p-1),\lambda_{\tilde p},J)$ given by
\[
I\left(e_+^{m_+}e_-^{m_+-d\tilde p}\right):=I\left(e_+^{m_+}e_-^{m_+-d\tilde p},\emptyset,Z_{(m_+,d)}\right)
\]

\begin{prop}\label{prop:grading} Let $e_+^{m_+}e_-^{m_-}$, with $m_+-d\tilde p=m_-$ for some $d\in\Z$, be a generator of $ECC(L(\tilde p,\tilde p-1),\lambda_{\tilde p},J)$. We have
\[
I\left(e_+^{m_+}e_-^{m_+-d\tilde p}\right)=-\tilde pd^2+\sum_{i=1}^{m_+}\left(2\left\lfloor\frac{i}{y_+}\right\rfloor+1\right)+\sum_{j=1}^{m_+-d\tilde p}\left(2\left\lfloor\frac{j}{-y_-+F}\right\rfloor+1\right)
\]
\end{prop}

\begin{proof}
Recall that we refer to the trivialization which has linking number zero with $e_\pm$ with respect to $A_0$, where $A_0$ denotes the image of the zero page of $\Pi_{\tilde p}$, as $\tau_{A_0}$.

The contribution to the Conley-Zehnder index of $e_+^{m_+}e_-^{m_+-d\tilde p}$ is determined by the rotation numbers with respect to $\tau_{A_0}$. We compute these rotation numbers as follows:
\[
\rot_{\tau_{A_0}}(e_\pm)=\frac{1}{\tilde p}\rot_{\tau_{\tilde pA_0}}\left(e_\pm^{\tilde p}\right)=\frac{1}{\tilde p}\left(\rot_{\tau_{A_0^1}}\left(\gamma^{1,2}_{(1,b_+)}\right)\right)=\frac{1}{\tilde p}\left(\rot_{\tau_{S_\pm^1}}\left(\gamma^{1,2}_{(1,b_+)}\right)+1\right)
\]
From the computations in the proof of Proposition \ref{prop:modelform}, we get
\begin{align*}
\rot_{\tau_{A_0}}(e_+)&=\frac{1}{\tilde p}\left(\left(\frac{\tilde p}{y_+}-1\right)+1\right)=\frac{1}{y_+}
\\\rot_{\tau_{A_0}}(e_-)&=\frac{1}{\tilde p}\left(\left(\frac{\tilde p}{-y_-+F}-1\right)+1\right)=\frac{1}{-y_-+F}
\end{align*}
therefore the Conley-Zehnder contribution of $e_+^{m_+}e_-^{m_+-d\tilde p}$ to its index is
\[
\sum_{i=1}^{m_+}CZ_{\tau_{A_0}}(e_+^i)+\sum_{j=1}^{m_+-d\tilde p}CZ_{\tau_{A_0}}(e_-^j)=\sum_{i=1}^{m_+}\left(2\left\lfloor\frac{i}{y_+}\right\rfloor+1\right)+\sum_{j=1}^{m_+-d\tilde p}\left(2\left\lfloor\frac{j}{-y_-+F}\right\rfloor+1\right)
\]

Next we investigate the surfaces representing $Z_{(m_+,d)}$ to compute $c_{\tau_{A_0}}$ and $Q_{\tau_{A_0}}$. Pairs $e_+e_-$ bound the page $A_0$. In the proof of Lemma \ref{lem:computerotpm}, we found that $e_+^{\tilde p}$ is the boundary of $S_+=\Sigma_+\cup_{\nu_-}D_-$, and $e_-^{\tilde p}$ is the boundary of $S_-=\Sigma_-\cup_{\nu_+}D_+$. Therefore
\begin{itemize}
\item If $d>0$, then $e_+^{m_+}e_-^{m_+-d\tilde p}$ is the boundary of $dS_+\cup(m_+-d\tilde p)A_0$.
\item If $d=0$, then $e_+^{m_+}e_-^{m_+}$ is the boundary of $m_+A_0$.
\item If $d<0$, then $e_+^{m_+}e_-^{m_+-d\tilde p}$ is the boundary of $m_+A_0\cup dS_-$.
\end{itemize}

Therefore in order to compute the ECH index of a generator we need only to compute $c_{\tau_{A_0}}$ and $Q_{\tau_{A_0}}$ of the $H_2(L(\tilde p,\tilde p-1),e_\pm^{\tilde p},\emptyset)$ class of $dS_\pm$ and the $H_2(L(\tilde p,\tilde p-1),e_+e_-,\emptyset)$ class of $m_+A_0$ and add as necessary. To add, we use the following fact: if $Z\in H_2(Y,\alpha,\beta)$ and $W\in H_2(Y,\beta,\gamma)$ then both $c_\tau(Z+W)=c_\tau(Z)+c_\tau(W)$ and $Q_\tau(Z+W)=Q_\tau(Z)+Q_\tau(W)$, which are shown in \cite[\S3.3]{indexineq}. With this method we never have to worry about intersections between lifts of $S_\pm$ and $A_0$ to $[-1,1]\times L(\tilde p,\tilde p-1)$ when computing $Q_{\tau_{A_0}}$. 

The computations are of two types: those which can be computed directly in $L(\tilde p,\tilde p-1)$ and those which use the lift to $\partial E(1,b_+)$. First we directly compute $c_{\tau_{A_0}}(A_0)$ and $Q_{\tau_{A_0}}(A_0)$.

For the computation of $c_{\tau_{A_0}}(A_0)$, we choose the section $\partial_{\rho_+}$ of $\xi_{\tilde p}$ over $L(\tilde p,\tilde p-1)-\{e_\pm\}$. It is constant with respect to $\tau_{A_0}$, so in particular it is never zero on $A_0$. Therefore
\[
c_{\tau_{A_0}}(A_0)=0
\]

For the computation of $Q_{\tau_{A_0}}(A_0)$, we can represent the $H_2(L(\tilde p,\tilde p-1),e_+e_-,\emptyset)$ class of $A_0$ by a surface which in the $(1-\epsilon,1]$ range is an embedding whose image in a slice transverse to $e_\pm$ is a union of rays which do not intersect and do not rotate with respect to $\tau_{A_0}$ as follows. Take the map $\{1\}\times id$ to $\{1\}\times L(\tilde p,\tilde p-1)$ and ``push" the middle of the annulus into $(-1,1)\times L(\tilde p,\tilde p-1)$ by a map which is quadratic with respect to the radial direction on $A_0$. For example, if $A_0$ is parameterized by $[-1,1]_x\times(\R/2\pi\Z)_y$ then the embedding is $\{x^2\}\times id$. This surface is constant with respect to $\tau_{A_0}$ by the definition of $\tau_{A_0}$. Therefore we can use it to compute relative self-intersection number.

This surface does not intersect itself because we can simply push the middle of the annulus deeper into the $(-1,1)$ direction, e.g. by using $\frac{3}{2}x^2-\frac{1}{2}$ instead of $x^2$. Therefore
\[
Q_{\tau_{A_0}}(A_0)=0
\]

To compute $c_{\tau_{A_0}}(S_\pm)$ and $Q_{\tau_{A_0}}(S_\pm)$ we will pass through the computations for their lifts in $\partial E(1,b_+)$. Denote by $A_0^1$ the zero page of $\Pi_1$ and $S_\pm^1$ the Seifert surfaces $\theta_{1,2}=0$ for $\gamma^{1,2}_{(1,b_+)}$ in $\partial E(1,b_+)$.

Let $\tilde pS_\pm^1$ denote the equivalence class of $S_\pm^1$ under the action of $G_{\tilde p,\tilde p-1}$, the group introduced in Lemma \ref{lem:stdlens} to define $L(\tilde p,\tilde p-1)$. Instead of computing the values which $c_{\tau_{A_0}}$ and $Q_{\tau_{A_0}}$ take on $S_\pm$ directly, we will use instead the homologous surfaces $\mathfrak{q}_{(1,b_+)}(\tilde pS_\pm^1)$. As in \cite[\S2.2]{indexineq}, the relative first Chern class changes by the first Chern class when the homology class changes:
\[
c_\tau(Z)-c_\tau(Z')=c_1(Z-Z')
\]
where $c_1$ is the first Chern class of $\xi$. In our case, $H_2(L(\tilde p,\tilde p-1))=0$, so this difference is zero. The relative self-intersection number will not change, either. Following \cite[Lemma 2.5]{indexineq}, the difference between $Q_\tau(Z)$ and $Q_\tau(Z')$ is twice the intersection number of $Z-Z'$ with the $H_1(L(\tilde p,\tilde p-1))$ class in which the generators of the ECH chain complex live. Because our generators are homologous to zero, this difference is zero.

We have
\begin{align*}
c_{\tau_{A_0}}(S_\pm)&=c_{\tau_{A_0}}(\mathfrak{q}_{(1,b_+)}(\tilde pS_\pm^1))
\\&=c_{\tau_{A_0^1}}(\tilde p S_\pm^1)
\\&=\tilde p c_{\tau_{A_0^1}}(S_\pm^1)
\\&=\tilde p\left(c_{\tau_{S_\pm^1}}(S_\pm^1)-1\right)\text{ by (\ref{eqn:ctauprime})}
\end{align*}
As in \cite[\S3.7]{echlec}, we have $c_{\tau_{S_\pm^1}}(S_\pm^1)=1$, therefore
\[
c_{\tau_{A_0}}(dS_\pm)=dc_{\tau_{A_0}}(S_\pm)=0
\]

For the relative self-intersection number, no disk in $\tilde pS_\pm$ intersects any other, and they all have equal self-intersection numbers in $[-1,1]\times E(1,b_+)$. Therefore
\begin{align*}
Q_{\tau_{A_0}}(dS_\pm)&=Q_{\tau_{A_0}}(\mathfrak{q}_{(1,b_+)}(\tilde p dS_\pm^1))
\\&=\tilde pQ_{\tau_{A_0^1}}(d S_\pm^1) \text{ by the sentence above}
\\&=\tilde p\left(Q_{\tau_{S_\pm^1}}(dS_\pm^1)+d^2(-1)\right)
\end{align*}
The final equality comes from \ref{eqn:Qtauprime}, because the multiplicity of the positive end is $d$ and the trivializations over $\gamma^{1,2}_{(1,b_+)}$ which have linking number zero with $\gamma^{1,2}_{(1,b_+)}$ with respect to $A_0^1$ and with respect to $S_\pm^1$ differ by one. As in \cite[\S3.7]{echlec}, $Q_{\tau_{S_\pm^1}}(dS_\pm^1)=0$ because the disks $S_\pm^1$ can be perturbed to not intersect themselves even in $E(1,b_+)$ when the perturbation is constant with respect to $\tau_{S_\pm^1}$. Therefore
\[
Q_{\tau_{A_0}}(dS_\pm)=-\tilde pd^2
\]
\end{proof}

This index is relatively easy to compute for a given generator, but we still do not understand the global structure of the chain complex. It turns out to be entirely combinatorial. The study of the ECH of contact forms whose Reeb vector field is parallel to a family of tori was initiated by Hutchings and Sullivan in \cite{T3}, and an analysis similar to theirs applies to many contact forms on lens spaces. We introduce one such description of the Reeb dynamics of $\lambda_{\tilde p}$ in the proof of the next proposition.

\begin{prop}\label{prop:indexevenintegers} The indices of the generators of $ECC_*(L(\tilde p,\tilde p-1),\lambda_{\tilde p},J)$ are in bijection with the nonnegative even integers.
\end{prop}
\begin{proof} From Proposition \ref{prop:grading}, we know that
\[
I\left(e_+^{m_+}e_-^{m_+-d\tilde p}\right)=-\tilde pd^2+\sum_{i=1}^{m_+}\left(2\left\lfloor\frac{i}{y_+}\right\rfloor+1\right)+\sum_{j=1}^{m_+-d\tilde p}\left(2\left\lfloor\frac{j}{-y_-+F}\right\rfloor+1\right)
\]

We will obtain the bijection $I:\Z_{\geq0}^2\to2\Z_{\geq0}$ by counting lattice points in a polygonal region. Let $k_+=(1,0)$ and $k_-=(1,\tilde p)$. To each generator $e_+^{m_+}e_-^{m_+-d\tilde p}$ we associate the following lattice point in $\Z^2=H_1(T^2;\Z)$
\[
V_{(m_+,d)}=(d,m_+)
\]
It is straightforward to check that $V_{(m_+,d)}$ is in the northwest of the four skew quadrants determined by the axes spanned by $k_\pm$ (inclusive of axes). In particular, the map given by $e_+^{m_+}e_-^{m_+-d\tilde p}\mapsto V_{(m_+,d)}$ is a bijection to the northwest quadrant.

Let $L_{(m_+,d)}$ denote the line through $V_{(m_+,d)}$ of slope $y_+$. We claim that $I\left(e_+^{m_+}e_-^{m_+-d\tilde p}\right)$ equals twice the number of lattice points contained in the triangle enclosed by $L_{(m_+,d)}$ and the axes spanned by $k_\pm$, minus two.

The contribution to the Conley-Zehnder index from $e_+^{m_+}$ is given by twice the number of lattice points in the triangle bounded by the horizontal axis, the line $L_{(m_+,d)}$, and the vertical line through $V_{(m_+,d)}$, including the boundary points, except that the points on the vertical line are counted once rather than twice and $V_{(m_+,d)}$ is not counted. The reasoning is just as in the case of the ellipsoid, see \cite{echlec}.

The contribution to the Conley-Zehnder index from $e_-^{m_+-d\tilde p}$ is given by twice the number of lattice points in the triangle bounded by the skew axis spanned by $k_-$, the line $L_{(m_+,d)}$, and the vertical line through $V_{(m_+,d)}$, including the boundary points, except that the points on the vertical line are counted once rather than twice and $V_{(m_+,d)}$ is not counted. The reasoning is just as in the case of the ellipsoid, see \cite{echlec}, after composing the entire picture by the automorphism $\begin{pmatrix}-\tilde p&1\\1&0\end{pmatrix}$ of the lattice.

Finally, we claim that $-\tilde pd^2$ removes the overcount of all points strictly outside the northwest quadrant. There are three cases to check:
\begin{itemize}
\item When $d>0$, the vertical line through $V_{(m_+,d)}$ hits the horizontal axis to the right of the origin, forming a triangle in the northeast quadrant with the axes. This triangle has height $d\tilde p$ and width $d$, therefore area $\frac{1}{2}\tilde pd^2$. Because the triangle is a lattice polygon, we can use Pick's theorem to obtain
\[
\tilde pd^2=B+2I-2
\]
where $B$ denotes the number of boundary points of the triangle and $I$ denotes the number of interior points.

Let $V$ denote the number of boundary points on the vertical line which are below the skew axis, $H$ the number of boundary points on the horizontal axis which are not on the vertical line nor are the origin, and $S$ the number of boundary points on the skew axis. Then $B=V+H+S$. Meanwhile, the contribution to the Conley-Zehnder index from $e_+^{m_+}$ overcounts by $V+2I+2H$.

Because $S=H+2$,
\[
V+2I+2H=V+H+(S-2)+2I=B+2I-2
\]
therefore the contribution to the Conley-Zender index from $e_+^{m_+}$ overcounts by $\tilde pd^2$.

\item When $d=0$, the vertical line through $V_{(m_+,0)}=(m_+,0)$ passes through the origin, therefore there are no lattice points in either triangle outside the southeast quadrant, which correspond to the fact that the index is totally determined by the Conley-Zehnder indices of the orbits.

\item The case when $d<0$ is similar to the $d>0$ case, except that the triangle is now in the southwest quadrant. Figure \ref{fig:toricindexcomp} indicates the relevant regions in the case $d=-1$.

\end{itemize}

\begin{figure}
\centering
\includegraphics[width=100mm]{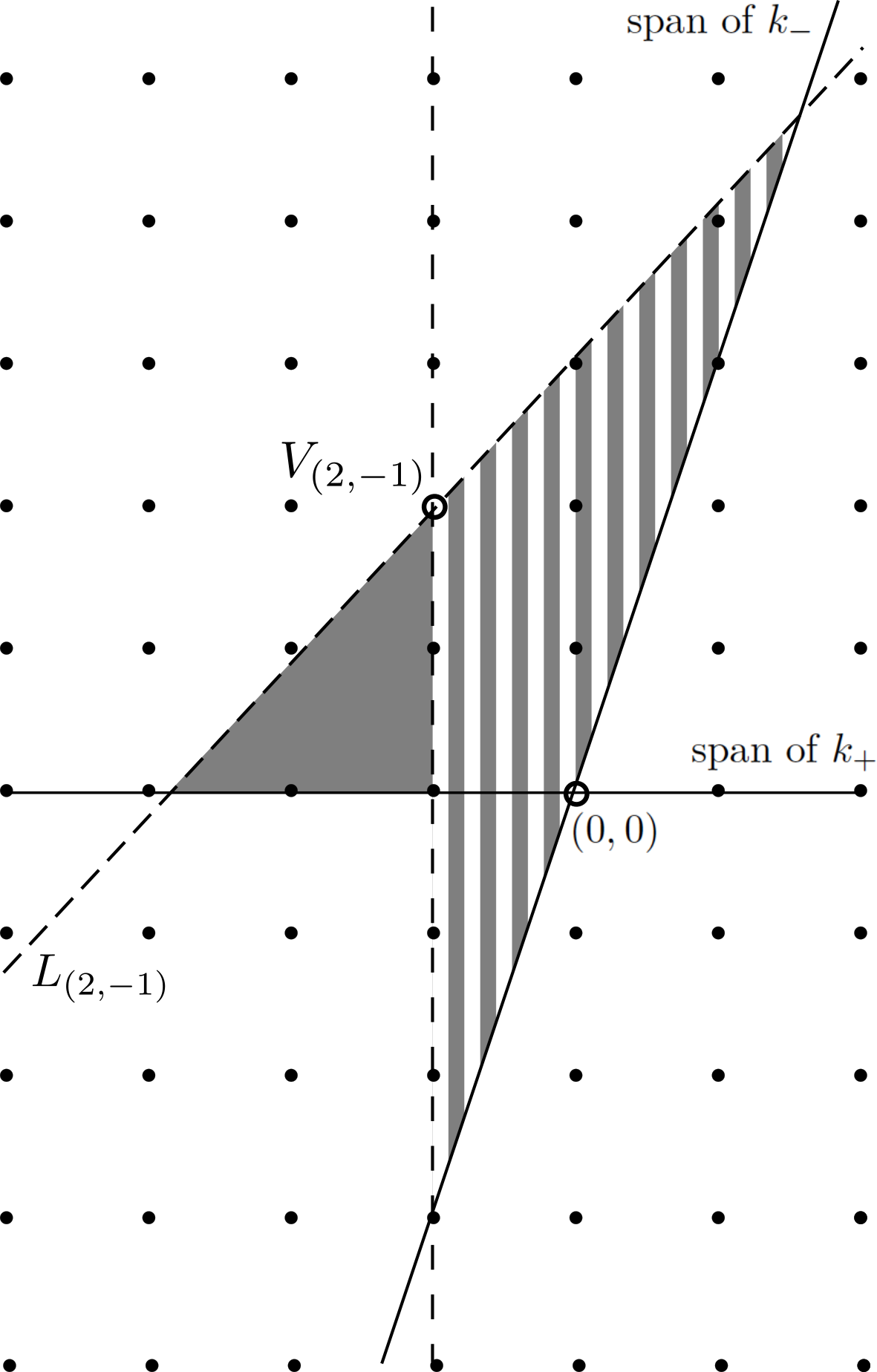}
\caption{The manifold in question is $L(3,2)$ with $y_+>1$. The solid triangle is used to compute $CZ^I_{\tau_{A_0}}\left(e_+^2\right)$ while the dashed triangle is used to compute $CZ^I_{\tau_{A_0}}\left(e_-^5\right)$. The triangle below the horizontal axis has been overcounted, corresponding to $Q_{\tau_{3A_0}}(S_-)$.}
\label{fig:toricindexcomp}
\end{figure} 

It remains to show that the index realizes the bijection between ECH generators and nonnegative even integers. To do this we introduce the following ordering on the points in the northwest quadrant. They can be indexed by the order that they are included in the half space to the right of and below a line of slope $y_+$ as it is moved to the left and up. A line of the irrational slope only ever contains one lattice point, so as the line moves it incorporates all the lattice points in the northwest quadrant, one at a time.

The index of a generator $e_+^{m_+}e_-^{m_+-d\tilde p}$ is given by twice the number of lattice points which come strictly before $V_{(m_+,d)}$ in the ordering by inclusion in the half space below a line of slope $y_+$. Therefore the ECH index is the composition of the bijection from $e_+^{m_+}e_-^{m_+-d\tilde p}$ to $V_{(m_+,d)}$ to its place in the ordering multiplied by two, starting at zero.

\end{proof}

We can immediately compute the embedded contact homology of $(L(\tilde p,\tilde p-1),\xi_{\tilde\psi})$.

\begin{cor}\label{cor:echcomp}
\[
ECH_*(L(\tilde p,\tilde p-1),\xi_{\tilde\psi})=\begin{cases}\Z/2\Z&\text{ if }*\in2\Z_{\geq0}
\\0&\text{ else}\end{cases}
\]
\end{cor}
\begin{proof} $ECH_*(L(\tilde p,\tilde p-1),\xi_{\tilde\psi})=ECH_*(L(\tilde p,\tilde p-1),\xi_{\tilde p})$, and the differential in the latter vanishes as there are no odd index generators.
\end{proof}

\subsection{Computation of $ECH_*^{\F_{e_+}+\F_{e_-}\leq \ell}(L(\tilde p,\tilde p-1),\xi_{\tilde\psi},e_+,e_-,\rot(e_+),\rot(e_-))$}\label{subsec:computekfech}

We use the bijection between generators and vertices in the plane introduced in the proof of Proposition \ref{prop:indexevenintegers} to compute the homology of $ECC_*(L(\tilde p,\tilde p-1),\lambda_{\tilde p},J)$ filtered by $\F_{e_+}+\F_{e_-}$.

When we set $B=e_\pm$, respectively, we obtain
\begin{align*}
\F_{e_+}\left(e_+^{m_+}e_-^{m_-}\right)&=m_+\rot(e_+)+lk(e_+,e_-)
\\&=m_+\left(\frac{1}{y_+}-\frac{1}{\tilde p}\right)+\frac{m_-}{\tilde p}
\end{align*}
and
\begin{align*}
\F_{e_-}\left(e_+^{m_+}e_-^{m_-}\right)&=lk(e_-,e_+)+m_-\rot(e_-)
\\&=\frac{m_+}{\tilde p}+m_-\left(\frac{1}{-y_-+F}-\frac{1}{\tilde p}\right)
\end{align*}
which both follow from
\[
lk(e_+,e_-)=lk(e_-,e_+)=\frac{S_+\cdot e_-}{\tilde p}=\frac{1}{\tilde p}\left(\Sigma_+\cdot e_-+D_-\cdot e_-\right)=\frac{0+1}{\tilde p}
\]

In the lattice description, we can compute using elementary geometry that the value $\F_{e_+}\left(e_+^{m_+}e_-^{m_+-d\tilde p}\right)$ is the horizontal distance between the $x$-intercept of $L_{(m_+,d)}$ and $V_{(m_+,d)}$, and similarly, that the value $\F_{e_-}\left(e_+^{m_+}e_-^{m_+-d\tilde p}\right)$ is the horizontal distance between $V_{(m_+,d)}$ and the intersection between the skew axis and the line $L_{(m_+,d)}$. (Using these ideas, we could at this point also compute the ECH of $L(\tilde p,\tilde p-1)$ filtered by the orbits $e_\pm$ separately.) Therefore, the sum filtration
\[
\F_{e_+}\left(e_+^{m_+}e_-^{m_-}\right)+\F_{e_-}\left(e_+^{m_+}e_-^{m_-}\right)=\frac{m_+}{y_+}+\frac{m_-}{-y_-+F}
\]
can be expressed as the width of the generator.

Given real numbers $a$ and $b$, let $N_k(a,b)$ denote the $k^\text{th}$ term in the sequence of nonnegative integer linear combinations of $a$ and $b$, listed with repetition and in increasing order, starting with $N_0(a,b)=0$. We encode the values of the sum filtration using the sequence of $N_k$s as follows. Given $k$, there are $m_\pm(k)$ for which $I(e_+^{m_+(k)}e_-^{m_-(k)})=2k$. Let $w(k)$ denote the integer for which
\[
N_{w(k)}\left(\frac{1}{y_+},\frac{1}{-y_-+F}\right)=\F_{e_+}\left(e_+^{m_+}e_-^{m_-}\right)+\F_{e_-}\left(e_+^{m_+}e_-^{m_-}\right)
\]
Notice that if $k\neq k'$ then $w(k)\neq w(k')$, because the width of the triangle representing each generator  strictly increases as we move the line through $V_{(m_+,d)}$ of slope $y_+$ to the left and up (which corresponds to increasing the ECH index). In addition, $w(k)\geq k$, because $w(k)<k$ would imply that the $k$ generators with index lower than $I(e_+^{m_+(k)}e_-^{m_-(k)})$ have a lower filtration value, which is not possible because filtration values strictly increase with index. Unless $\tilde p=1$, when $k$ is large enough $w(k)>k$, because not all pairs $(m_+,m_-)$ can satisfy $m_+\equiv m_-\mod\tilde p$.

 The above discussion proves
\begin{prop}\label{prop:sumkfech}
\[
ECH_{2k}^{\F_{e_+}+\F_{e_-}\leq\ell}(L(\tilde p,\tilde p-1),\xi_{\tilde\psi},e_+,e_-,\rot(e_+),\rot(e_-))=\begin{cases}\Z/2\Z&\hspace{-.15in}\text{ if }\ell\geq N_{w(k)}\left(\frac{1}{y_+},\frac{1}{-y_-+F}\right)
\\0&\hspace{-.15in}\text{ else}\end{cases}
\]
\end{prop}

\begin{ex} Refine the example illustrated in Figure \ref{fig:toricindexcomp} by setting $y_+=1+\frac{e}{30}$ and $-y_-+F=2-\frac{e}{30}$. Then the first twelve values of $w(k)$, in order of increasing $k$, are $0, 4, 5, 12, 13, 14, 15, 25, 26, 27, 28, 30$. These were found by computing the widths of the generators in indices up to $22$ and comparing them to the first twenty-nine values of the sequence with $k^\text{th}$ entry $N_k\left(\frac{1}{1+\frac{e}{30}},\frac{1}{2-\frac{e}{30}}\right)$.
\end{ex}

\section{Proof of the main theorem}\label{sec:proofmainthm}

We prove the main theorem. Our proof is inspired by the proof of \cite[Theorem \ref{thm:disk}]{mean}, however, there are several significant differences, which we list here and discuss at the appropriate stage in the proof.

Our Lemma \ref{lem:Nkinequality}, an analogue of \cite[Lemma 3.2]{mean} with $w(k)$ substituted for $k$, relies on our understanding of the relationship between $w(k)$ and the knot filtration on ECH, developed in \S\ref{subsec:computekfech}.

In Proposition \ref{prop:penultimate}, it is only possible to find an upper bound for the infimum of the mean action of periodic orbits of $\psi$ asymptotically, that is, by increasing $y_+$ arbitrarily high. This is because of the way the inequality we obtain in Lemma \ref{lem:Nkinequality} differs from the inequality obtained in \cite[Lemma 3.2]{mean}. Heuristically, this is because we need the ratio between the boundary rotation numbers to be close to one in order to prove that the periodic orbit identified in Proposition \ref{prop:pairineqscontact} is not the empty orbit.

The construction of the contact manifold in Proposition \ref{prop:constructcontact} only works when $y_+-y_-\in\Z$, $F\in\Z$, and to extend to $\psi$ for which these quantities are rational requires Lemma \ref{lem:rational}, which has no parallel in \cite{mean}.

Finally, the last part of the proof is much more complicated than the analogous section of \cite{mean}, the proof of Theorem \ref{thm:disk} assuming \cite[Proposition 2.2]{mean}; we discuss the specific reasons for this when we introduce that part of the proof.

\subsection{Initial bound in contact-geometric setting}\label{subsec:firstupperbound} 

\begin{prop}\label{prop:pairineqscontact} Let $\lambda$ be a contact form on $L(p,p-1)$ which is contactomorphic to the contact form $\lambda_p$ from Lemma \ref{lem:stdlens}. Suppose that both binding components $b_\pm$ of the open book decomposition $(H_p,\Pi_p)$ are elliptic for $\lambda$. Further suppose that the rotation numbers of $b_\pm^p$ in the trivializations which have linking number zero with $b_\pm$ with respect to their respective Seifert surfaces are $p\rot(b_\pm)$, respectively, where
\[
\left(\rot(b_+)+\frac{1}{p}\right)^{-1}+\left(\rot(b_-)+\frac{1}{p}\right)^{-1}=p
\]
Then for all $\epsilon>0$, for all sufficiently large integers $k$ there is an orbit set $\alpha_k$ not including either binding component and nonnegative integers $m_{k,\pm}$ for which
\begin{align}
&I\left(b_+^{m_{k,+}}\alpha_k b_-^{m_{k,-}}\right)=2k \label{eqn:indexbalphab}
\\&\A(\alpha_k)\leq\sqrt{2k(\vol(L(p,p-1),\lambda)+\epsilon)}-m_{k,+}\A(b_+)-m_{k,-}\A(b_-) \label{eqn:actionbalphab}
\\&\alpha_k\cdot A_0\geq N_{w(k)}\left(\rot(b_+)+\frac{1}{p},\rot(b_-)+\frac{1}{p}\right)-m_{k,+}\rot(b_+)-m_{k,-}\rot(b_-) \label{eqn:intersectionbalphab}
\end{align}
\end{prop}
\begin{proof}
Let $x_k$ be a cycle in $ECC_{2k}(L(p,p-1),\lambda,J)$ representing the generator of the group $ECH_{2k}(L(p,p-1),\xi)$. Then $\{[x_k]\}$ is a sequence satisfying the hypotheses of Theorem \ref{thm:asymp}, so for $k$ sufficiently large
\[
\frac{c_{[x_k]}(L(p,p-1),\lambda)^2}{2k}\leq\vol(L(p,p-1),\lambda)+\epsilon
\]
So for all such $k$, there is a finite set of orbit sets $x_{k_i}$ with $I(x_{k_i})=2k$ for which
\[
x_k=\sum_{i}x_{k_i}
\]
and each orbit set $x_{k_i}$ satisfies
\[
\A(x_{k_i})\leq c_{[x_k]}(L(p,p-1),\lambda)\Rightarrow\A(x_{k_i})\leq\sqrt{2k(\vol(L(p,p-1)+\epsilon)}
\]

$x_{k_i}$ can be written in the form $b_+^{m_{k_i,+}}\alpha_{k_i}b_-^{m_{k_i,-}}$ for some orbit set $\alpha_{k_i}$ not including either $b_\pm$, from which we get
\[
\A(\alpha_{k_i})\leq\sqrt{2k(\vol(L(p,p-1),\lambda)+\epsilon)}-m_{k_i,+}\A(b_+)-m_{k_i,-}\A(b_-)
\]

By Proposition \ref{prop:sumkfech}, Theorem \ref{thm:sumkffunctoriality}, and the fact that $\lambda$ and $\lambda_p$ share contact structures, the elliptic orbits $b_\pm$, and their rotation numbers, a cycle $x_k$ representing the generator of $ECH_{2k}(L(p,p-1),\xi)$ must have at least one $x_{k_i}$ satisfying
 \begin{equation}\label{eqn:sumbound}
\F_{e_+}(x_{k_i})+\F_{e_-}(x_{k_i})\geq N_{w(k)}\left(\rot(b_+)+\frac{1}{p},\rot(b_-)+\frac{1}{p}\right)
\end{equation}
For all $k$, choose one such $x_{k_i}$ and write it in the form $b_+^{m_{k,+}}\alpha_k b_-^{m_{k,-}}$. Notice that $\alpha_k$ satisfies both (\ref{eqn:indexbalphab}) and (\ref{eqn:actionbalphab}).

Expanding the left hand side of (\ref{eqn:sumbound}) and cancelling the contributions from $lk(b_\pm,b_\mp)$ against some of the contributions from $\rot(b_\pm)$ as in \S\ref{subsec:computekfech} gives
\begin{align}
\F_{e_+}(b_+^{m_{k,+}}\alpha_k b_-^{m_{k,-}})+\F_{e_-}(b_+^{m_{k,+}}\alpha_k b_-^{m_{k,-}})&=m_{k,+}\left(\rot(b_+)+\frac{1}{p}\right)+m_{k,-}\left(\rot(b_-)+\frac{1}{p}\right) \nonumber
\\&\;\;\;\;+lk(b_+,\alpha)+lk(b_-,\alpha) \label{eqn:sumfiltbalphab}
\end{align}

By definition of linking number, if $S_{b_\pm}$ are Seifert surfaces for $b_\pm^p$,
\[
lk(b_\pm,\alpha)=\frac{\alpha\cdot S_{b_\pm}}{p}
\]
Because $H_2(L(p,p-1);\Z)=0$, the closed surface $-S_{b_+}\cup_{b_+^p}nA_0\cup_{b_-^p}-S_{b_-}$ has intersection number zero with any one-cycle. Therefore
\[
lk(b_+,\alpha)+lk(b_-,\alpha)=\frac{1}{p}\left(\alpha\cdot S_{b_+}+\alpha\cdot S_{b_-}\right)=\frac{\alpha\cdot pA_0}{p}=\alpha\cdot A_0
\]
which, combined with (\ref{eqn:sumbound}) and (\ref{eqn:sumfiltbalphab}), proves (\ref{eqn:intersectionbalphab}).

\end{proof}

\subsection{Final bound in annulus setting}\label{subsec:finalbound}

We need to reinterpret the lower bound (\ref{eqn:intersectionbalphab}) in terms of the index $2k$.

\begin{lem}\label{lem:Nkinequality} Given positive real numbers $a$ and $b$ with $a+b\in\Z$ and $\frac{a}{b}$ irrational, there are constants $c_1$ and $c_2$ such that for every $k\in\Z_{\geq0}$,
\begin{equation}\label{eqn:Nkinequality}
N_{w(k)}\left(\frac{1}{a},\frac{1}{b}\right)^2\geq\frac{2k(a+b)}{ab}-c_1k^\frac{1}{2}+c_2
\end{equation}
\end{lem}

\begin{proof} As discussed in \S\ref{subsec:computekfech}, $N_{w(k)}\left(\frac{1}{a},\frac{1}{b}\right)$ is the width of the triangle bounded by the vertical axis, the skew axis of slope $a+b$, and the line of slope $a$ which passes through the point $V=\left(\frac{m_+-m_-}{a+b},m_+\right)$, where $m_\pm$ are chosen so that there are precisely $k+1$ integral lattice points contained in this triangle, inclusive of the points on the boundary.

As in the proof of Proposition \ref{prop:indexevenintegers}, we can count these lattice points in another way:
\begin{align}
2k&=-(a+b)\left(\frac{m_+-m_-}{a+b}\right)^2+\sum_{i=1}^{m_+}\left(2\left\lfloor\frac{i}{a}\right\rfloor+1\right)+\sum_{j=1}^{m_-}\left(2\left\lfloor\frac{j}{b}\right\rfloor+1\right) \nonumber
\\&\leq-(a+b)\left(\frac{m_+-m_-}{a+b}\right)^2+\sum_{i=1}^{m_+}\left(\frac{2i}{a}+1\right)+\sum_{j=1}^{m_-}\left(\frac{2j}{b}+1\right) \nonumber
\\&=-\frac{m_+^2}{a+b}+\frac{2m_+m_-}{a+b}-\frac{m_-^2}{a+b}+m_++\frac{m_+(m_++1)}{a}+m_-+\frac{m_-(m_-+1)}{b} \nonumber
\\&=m_+^2\left(\frac{1}{a}-\frac{1}{a+b}\right)+\frac{2m_+m_-}{a+b}+m_-^2\left(\frac{1}{b}-\frac{1}{a+b}\right)+m_+\left(1+\frac{1}{a}\right)+m_-\left(1+\frac{1}{b}\right) \nonumber
\\&=\frac{bm_+^2}{a(a+b)}+\frac{2m_+m_-}{a+b}+\frac{am_-^2}{b(a+b)}+m_+\left(1+\frac{1}{a}\right)+m_-\left(1+\frac{1}{b}\right) \label{eqn:2ksimplified}
\end{align}
Notice that
\begin{align*}
N_{w(k)}\left(\frac{1}{a},\frac{1}{b}\right)^2&=\left(\frac{m_+}{a}+\frac{m_-}{b}\right)^2
\\&=\frac{m_+^2}{a^2}+\frac{2m_+m_-}{ab}+\frac{m_-^2}{b^2}
\end{align*}
Therefore, if we multiply both sides of (\ref{eqn:2ksimplified}) by $\frac{a+b}{ab}$, we get
\begin{align}
\frac{2k(a+b)}{ab}&\leq \frac{m_+^2}{a^2}+\frac{2m_+m_-}{ab}+\frac{m_-^2}{b^2}+\frac{(a+b)m_+}{ab}\left(1+\frac{1}{a}\right)+\frac{(a+b)m_-}{ab}\left(1+\frac{1}{b}\right) \nonumber
\\&\leq N_{w(k)}\left(\frac{1}{a},\frac{1}{b}\right)^2+c_0N_{w(k)}\left(\frac{1}{a},\frac{1}{b}\right) \label{eqn:quadraticbound}
\end{align}
The constant $c_0$ is because any nonnegative linear combination of $m_\pm$ has an upper bound in terms of some constant times $N_{w(k)}\left(\frac{1}{a},\frac{1}{b}\right)$, because $N_{w(k)}\left(\frac{1}{a},\frac{1}{b}\right)$ is itself a nonnegative linear combination of the $m_\pm$.

Use $N$ to denote $N_{w(k)}\left(\frac{1}{a},\frac{1}{b}\right)$. We can simplify (\ref{eqn:quadraticbound}):
\begin{align*}
N^2+c_0N&\geq\frac{2k(a+b)}{ab}
\\N^2+c_0N+\frac{c_0^2}{4}&\geq\frac{2k(a+b)}{ab}+\frac{c_0^2}{4}
\\N&\geq\sqrt{\frac{2k(a+b)}{ab}+\frac{c_0^2}{4}}-\frac{c_0}{2}
\\N^2&\geq\frac{2k(a+b)}{ab}+\frac{c_0^2}{4}-c_0\sqrt{\frac{2k(a+b)}{ab}+\frac{c_0^2}{4}}+\frac{c_0^2}{4}
\end{align*}
From here the estimate (\ref{eqn:Nkinequality}) follows.
\end{proof}

Next we prove a weaker version of (\ref{eqn:goal}), restricted to those $\psi$ to which Proposition \ref{prop:constructcontact} applies, and for which $y_+$, $y_->>0$.

From now on we will use the notation
\[
\hm(a,b)=\frac{2}{\frac{1}{a}+\frac{1}{b}}=\frac{2ab}{a+b}
\]
to denote the harmonic mean of $a$ and $b$. We will also sometimes refer to ``the orbits of $(\psi,y_++N)$" to refer to the orbits of $\psi$ with total and mean actions computed using the action function normalized to be $y_++N$ on $\partial_+A$ rather than $y_+$.

\begin{prop}\label{prop:penultimate} Let $\psi$ be an area-preserving diffeomorphism of $(A,\omega)$ which is rotation by $2\pi y_\pm$ near $\partial_\pm A$, whose flux applied to the class of the $(x,0)$ curve is $F\in\Z$, for which $y_+-y_-\in\Z$, both $y_+$ and $-y_-+F$ are irrational, and whose action function $f$ is positive.

Let $\A_N$ denote the total action computed with $f_{(\psi,y_++N,\beta)}$ rather than with $f_{(\psi,y_+,\beta)}$.
If
\[
\V(\tilde\psi)<\max\{y_+,-y_-+F\}
\]
then for all sufficiently large integers $N$,
\begin{equation}\label{eqn:Ngoal}
\inf\left\{\frac{\A_N(\gamma)}{\ell(\gamma)}\;\middle|\;\gamma\in\P(\psi)\right\}\leq\sqrt{\hm(y_++N,-y_-+F+N)(\V(\tilde\psi)+N)}
\end{equation}
\end{prop}

\begin{proof} Let $\mcO(\lambda_{\tilde\psi})$ denote the set of orbit sets of the contact form $\lambda_{\tilde\psi}$ constructed by applying Proposition \ref{prop:constructcontact} to $\tilde\psi$ which do not include either binding component. Let $\mcO(\psi)$ denote the set of tuples $\{(\gamma_i,m_i)\}$ where the $\gamma_i$ are simple periodic orbits of $\psi$, the $m_i$ are positive integers, and if $\gamma_i$ is sent to $\alpha_i$ under the bijection from Proposition \ref{prop:constructcontact} then $\{(\alpha_i,m_i)\}\in\mcO(\lambda_{\tilde\psi})$. We extend the action, period, and mean action to $\mcO(\psi)$ in the obvious ways, e.g. $\A(\{(\gamma_i,m_i)\})=\sum_im_i\A(\gamma_i)$. In order to show (\ref{eqn:Ngoal}) it is enough to show the analogous inequality over $\mcO(\psi)$:
\begin{equation}\label{eqn:NOgoal}
\inf\left\{\frac{\A_N(\{(\gamma_i,m_i)\})}{\ell(\{(\gamma_i,m_i)\})}\;\middle|\;\{(\gamma_i,m_i)\}\in\mcO(\psi)\right\}\leq\sqrt{\hm(y_++N,-y_-+F+N)(\V(\tilde\psi)+N)}
\end{equation}
This is because
\[
\frac{m_i\A_N(\gamma_i)}{m_i\ell(\gamma_i)}=\frac{\A_N(\gamma_i)}{\ell(\gamma_i)}
\]
and because for any sequences $a_1,\dots,a_l$ and $\ell_1,\dots,\ell_l$ there must be some $i$ for which
\[
\frac{a_i}{\ell_i}\leq\frac{\sum_ia_i}{\sum_i\ell_i}
\]
which can be proved for $l=2$ by
\[
\frac{a_1+a_2}{\ell_1+\ell_2}<\frac{a_1}{\ell_1}\text{ and }\frac{a_2}{\ell_2}\Rightarrow a_2\ell_1<a_1\ell_2\text{ and }a_1\ell_2<a_2\ell_1
\]
and for $l>2$ by induction.

Choose $\epsilon>0$ so that
\begin{equation}\label{eqn:epsilonchoice}
\V(\tilde\psi)+\frac{\epsilon}{2}<\max\{y_+, -y_-+F\}
\end{equation}

We claim that Proposition \ref{prop:pairineqscontact} applies to $p=\tilde p$, $\lambda=\lambda_{\tilde\psi}$, $b_\pm=e_\pm$, and $\rot(b_+)=\frac{1}{y_+}-\frac{1}{\tilde p}, \rot(b_-)=\frac{1}{-y_-+F}-\frac{1}{\tilde p}$. This is because $\lambda_{\tilde\psi}$ is adapted to the open book decomposition $(B_{\tilde p},P_{\tilde p})$, which induces the same abstract open book $(A,D_{\tilde p})$ as does $(H_{\tilde p},\Pi_{\tilde p})$. Becuase $\lambda_{\tilde p}$ is adapted to $(H_{\tilde p},\Pi_{\tilde p})$, $\lambda_{\tilde\psi}$ and $\lambda_{\tilde p}$ are contactomorphic by Theorem \ref{thm:contactomorphic}.

Let $\alpha_k$ be the orbit set obtained by applying Proposition \ref{prop:pairineqscontact} to $\lambda_{\tilde\psi}$. By combining (\ref{eqn:intersectionbalphab}) and (\ref{eqn:Nkinequality}), we get the lower bound
\begin{equation}\label{eqn:lowerboundint}
\alpha\cdot A_0\geq\sqrt{\frac{2k\tilde p}{y_+(-y_-+F)}-c_1k^\frac{1}{2}+c_2}-\frac{m_{k,+}}{y_+}-\frac{m_{k,-}}{-y_-+F}
\end{equation}

We claim that the right hand side of (\ref{eqn:lowerboundint}) is positive when $k$ is large enough. Notice that positivity gives us $\alpha_k\neq\emptyset$, which was not guaranteed by Proposition \ref{prop:pairineqscontact}. Here is where the hypothesis (\ref{eqn:Vmax}) is crucial: the choice in \ref{eqn:epsilonchoice} is what allows us to assume $C<1$ in (\ref{eqn:C}) below.

The intuition behind the argument that $\alpha_k\neq\emptyset$ is the following. When we replace $y_+$ with $y_++N$, we also replace $-y_-+F$ with $-y_--N+F+2N=-y_-+F+N$. As $N$ gets large, the rotation numbers of the orbits $e_\pm$ for the contact form on $L(\tilde p+2N,\tilde p+2N-1)$ constructed as in Proposition \ref{prop:constructcontact} become very small, and so contribute less and less to the value of the  knot filtration on the orbit set identified in Proposition \ref{prop:pairineqscontact}. Lemma \ref{lem:Nkinequality} allows us to use (\ref{eqn:lowerboundint}) instead of (\ref{eqn:intersectionbalphab}), which is necessary to show that as $N$ increases, the lower bound on the value of the knot filtration on this orbit set does not go to zero as quickly as the contributions from the rotation numbers do.

The inequalities which prove this intuition are the following. Let $m=\min\{y_+,-y_-+F\}$ and $M=\max\{y_+,-y_-+F\}$. By applying (\ref{eqn:actionbalphab}) to $\lambda_{\tilde\psi}$, we get
\begin{equation}\label{eqn:actionpositive}
m_{k,+}+m_{k,-}\leq\sqrt{2k(\vol(L(\tilde p,\tilde p-1),\lambda_{\tilde\psi})+\epsilon)}-\A(\alpha_k)\leq\sqrt{2k(\vol(L(\tilde p,\tilde p-1),\lambda_{\tilde\psi})+\epsilon)}
\end{equation}
from which, by our choice of $\epsilon$ in (\ref{eqn:epsilonchoice}), we obtain
\begin{equation}\label{eqn:C}
m_{k,+}+m_{k,-}\leq C\sqrt{4kM}
\end{equation}
for some $C<1$. The right hand side of (\ref{eqn:lowerboundint}) is positive if
\begin{equation}\label{eqn:positive}
\frac{m_{k,+}}{y_+}+\frac{m_{k,-}}{-y_-+F}<\sqrt{\frac{2k\tilde p}{y_+(-y_-+F)}-c_1k^\frac{1}{2}+c_2}
\end{equation}
In order to show (\ref{eqn:positive}), we use (\ref{eqn:C}) to obtain the following upper bound on the left hand side of (\ref{eqn:positive}):
\begin{align*}
\frac{m_{k,+}}{y_+}+\frac{m_{k,-}}{-y_-+F}&\leq\frac{m_{k,+}+m_{k,-}}{m}
\\&\leq\frac{C\sqrt{4kM}}{m}
\end{align*}
Therefore, when $k$ is large enough, to show (\ref{eqn:positive}) it suffices to show that
\begin{align}
\frac{C\sqrt{4kM}}{m}&<\sqrt{\frac{2k\tilde p}{y_+(-y_-+F)}}\nonumber
\\\frac{2C^2M}{m^2}&<\frac{m+M}{mM} \nonumber
\\2C^2M^2&< m(m+M)\label{eqn:mM}
\end{align}
Now replace $y_+$ with $y_++N$, $-y_-+F$ with $-y_-+F+N$, $\tilde p$ with $\tilde p+2N$, and $\V(\tilde\psi)$ with $\V(\tilde\psi)+N$ for $N\in\Z$, which are the changes which occur when we replace $(\psi,y_+)$ with $(\psi,y_++N)$. (\ref{eqn:mM}) becomes
\[
2C^2(M^2+2MN+N^2)< mM+m^2+2mN+MN+mN+2N^2
\]
which holds when $N>>0$ because the $N^2$ term takes over, and $C<1$. These replacements are precisely the modifications which occur when we replace $(\psi,y_+)$ with $(\psi,y_++N)$.

Notice that the hypotheses of Propositions \ref{prop:constructcontact} and \ref{prop:pairineqscontact} hold for $(\psi,y_++N)$ if they hold for $(\psi,y_+)$, therefore we obtain an upper bound on the action and lower bound on the intersection number analogous to those we had for $(\psi,y_+)$. Note that we also now replace $\A$ with $\A_N$ in \ref{eqn:lowerboundint}, however we do not need to replace $\mcO(\lambda_{\tilde\psi})$ because the orbits of the contact form built from $(\psi,y_+)$ are the same as those built from $(\psi,y_++N)$.

Now that we know $\alpha_k\cdot A_0\geq0$, we can divide our upper bound on $\A_N(\alpha_k)$ by our lower bound on $\alpha_k\cdot A_0$ to obtain
\begin{equation}\label{eqn:firstdivision}
\frac{\A_N(\alpha_k)}{\alpha_k\cdot A_0}\leq\frac{\sqrt{2k(2(\V(\tilde\psi)+N)+\epsilon)}-m_{k,+}-m_{k,-}}{\sqrt{\frac{2k(\tilde p+2N)}{(y_++N)(-y_-+F+N)}-c_1k^\frac{1}{2}+c_2}-\frac{m_{k,+}}{y_++N}-\frac{m_{k,-}}{-y_-+F+N}}
\end{equation}

\textbf{Claim:} When extended to a function defined for $(m_+,m_-)\in\R_{\geq0}^2$, the right hand side of \ref{eqn:firstdivision} is maximized at $(0,0)$.

Consider a function of the form
\[
(x,y)\mapsto\frac{A-bx-cy}{D-ex-fy}
\]
where $A$, $b$, $c$, $D$, $e$, $f$ are constants in $x$ and $y$, and where both of the intercepts with the axes by the line where the numerator is zero occur at positive values of $x$ and $y$. It must have a maximum at $(0,0)$ if the denominator is positive whenever the numerator is positive and $x>0$, $y>0$. 

This is precisely what we have just shown when $N$ is large enough, by showing that (\ref{eqn:positive}) holds whenever (\ref{eqn:actionpositive}) holds. Therefore we can update (\ref{eqn:firstdivision}) by evaluating at $m_{k,\pm}=0$ to maximize its right hand side:
\[
\frac{\A_N(\alpha_k)}{\alpha_k\cdot A_0}\leq\frac{\sqrt{2k(2(\V(\tilde\psi)+N)+\epsilon)}}{\sqrt{\frac{2k(\tilde p+2N)}{(y_++N)(-y_-+F+N)}-c_1k^\frac{1}{2}+c_2}}
\]
By sending $k\to\infty$, we obtain
\begin{align}
\inf_k\left\{\frac{\A_N(\alpha_k)}{\alpha_k\cdot A_0}\right\}&\leq\frac{\sqrt{2k(2(\V(\tilde\psi)+N)+\epsilon)}}{\sqrt{\frac{2k(\tilde p+2N)}{(y_++N)(-y_-+F+N)}}} \nonumber
\\\inf\left\{\frac{\A_N(\alpha)}{\alpha\cdot A_0}\;\middle|\;\alpha\in\mcO(\lambda_{\psi})-\{e_\pm\}\right\}&\leq\frac{\sqrt{2(\V(\tilde\psi)+N)+\epsilon}}{\sqrt{\frac{\tilde p+2N}{(y_++N)(-y_-+F+N)}}} \label{eqn:bigsqrt}
\end{align}
Notice that the quantity inside the square root in the denominator of the right hand side of (\ref{eqn:bigsqrt}) simplifies to
\begin{align*}
\frac{\tilde p+2N}{(y_++N)(-y_-+F+N)}&=\frac{(y_++N)+(-y_-+F+N)}{(y_++N)(-y_-+F+N)}
\\&=\frac{1}{-y_-+F+N}+\frac{1}{y_++N}
\end{align*}
which, combined with (\ref{eqn:bigsqrt}) and sending $\epsilon\to0$, proves the inequality
\[
\inf\left\{\frac{\A_N(\alpha)}{\alpha\cdot A_0}\;\middle|\;\alpha\in\mcO(\lambda_{\psi})-\{e_\pm\}\right\}\leq\sqrt{\hm(y_++N,-y_-+F+N)(\V(\tilde\psi)+N)}
\]
Applying the bijection from Proposition (\ref{prop:constructcontact}) gives the desired inequality (\ref{eqn:NOgoal}), which gives us (\ref{eqn:Ngoal}) as discussed at the beginning of the proof.
\end{proof}

Before we prove Theorem \ref{thm:main}, we first prove two lemmas, which will allow us to prove our theorem for $\psi$ which do not satisfy all of the assumptions of Proposition \ref{prop:constructcontact}. The first will allow us to remove the assumption that the action function is positive. The second will allow us to replace the requirements that $y_+-y_-\in\Z$ and $F\in\Z$ with $y_+-y_-\in\Q$ and $F\in\Q$, which we will further weaken later in the proof.

\begin{lem}\label{lem:Nshift} If (\ref{eqn:goal}) holds for $(\psi,y_++N)$, then (\ref{eqn:goal}) holds for $(\psi,y_+)$.
\end{lem}
\begin{proof} Because $f_{(\psi,y_++N,\beta)}=f_{(\psi,y_+,\beta)}+N$, we have
\begin{equation}\label{eqn:shiftgoal}
\inf\left\{\frac{\A(\gamma)}{\ell(\gamma)}\;\middle|\;\gamma\in\P(\psi)\right\}+N=\inf\left\{\frac{\A_N(\gamma)}{\ell(\gamma)}\;\middle|\;\gamma\in\P(\psi)\right\}\leq\V_N(\tilde\psi)=\V(\tilde\psi)+N
\end{equation}
where $\A_N$ and $\V_N$ denote the total action and Calabi invariant, respectively, computed using $f_{(\psi,y_++N,\beta)}$.
\end{proof}

\begin{lem}\label{lem:rational} Let $\psi$ be an area-preserving diffeomorphism of $(A,\omega)$ which rotates by $2\pi y_\pm$ near $\partial_\pm A$. Then if $\gamma$ is a periodic orbit of $\psi^q$ which is covered by a periodic orbit $\gamma'$ of $\psi$, then
\begin{equation}\label{eqn:meanactionq}
\frac{\A(\gamma)}{\ell(\gamma)}=q\frac{\A(\gamma')}{\ell(\gamma')}
\end{equation}
and
\begin{equation}\label{eqn:calabiq}
\V(\tilde\psi^q)=q\V(\tilde\psi)
\end{equation}
\end{lem}
\begin{proof} Note $\tilde\psi^q$ is rotation by $qy_+$ near $x=1$. Let $f_q$ denote $f_{(\psi^q,qy_+,\beta)}$. Let $\gamma=(\gamma_1,\dots,\gamma_l)$ be a periodic orbit of $\psi^q$, and let $\gamma'=(\gamma_{1,1},\dots,\gamma_{1,q},\dots,\gamma_{l,1},\dots,\gamma_{l,q})$ be a lift to a periodic orbit of $\psi$, where $\gamma_{i,q}=\gamma_i$. Let $\eta$ be a path from $\gamma_1$ to $\partial_+A$. The total action of $\gamma$ is
\begin{align*}
\sum_{i=1}^lf_q(\gamma_i)&=\sum_{i=1}^l\left(qy_++\int_{\psi^{q(i-1)}\eta} df_q\right)
\\&=qly_++\sum_{i=1}^l\left(\int_{\psi^{q(i-1)}\eta}\left({\psi^q}^*\beta-\beta\right)\right)
\\&=qly_++\int_{\psi^{ql}\eta}\beta-\int_{\psi^{ql-1}\eta}\beta+\cdots+\int_{\psi\eta}\beta-\int_\eta\beta
\\&=\A(\gamma')
\end{align*}

We obtain (\ref{eqn:meanactionq}) by combining $\A(\gamma)=\A(\gamma')$ with $\ell(\gamma')=q\ell(\gamma)$. To show (\ref{eqn:calabiq}), notice
\[
df_q={\psi^q}^*\beta-\beta={\psi^q}^*\beta-{\psi^{q-1}}^*\beta+\cdots+\psi^*\beta-\beta
\]
Define $f_{q,i}$ by $df_{q,i}={\psi^{i}}^*\beta-{\psi^{i-1}}^*\beta$ and $f_{q,i}|_{\partial_+A}=y_+$. By Lemma \ref{lem:calabiindep}, we have $\int_Af_{q,i}\omega=2\V(\tilde\psi)$ for all $i$, which proves (\ref{eqn:calabiq}).

\end{proof}

\begin{rmrk} If $\tilde\psi$ is an area-preserving diffeomorphism of $(A,\omega)$ for which $y_+-y_-\in\Q$, $F\in\Q$, then there is some $q\in\Z$ for which $qy_+-qy_-\in\Z$, $qF\in\Z$. If the other hypotheses of Proposition \ref{prop:penultimate} apply to $\tilde\psi^q$, then we get
\[
q\inf\left\{\frac{\A(\gamma)}{\ell(\gamma)}\;\middle|\;\gamma\in\P(\psi)\right\}\leq\inf\left\{\frac{\A_q(\gamma)}{\ell_q(\gamma)}\;\middle|\;\gamma\in\P(\psi^q)\right\}\leq\V(\tilde\psi^q)=q\V(\tilde\psi)
\]
and dividing by $q$ gives us (\ref{eqn:goal}) for $\tilde\psi$. This is how we will obtain (\ref{eqn:goal}) for a broader class of maps than those which satisfy the hypotheses of Proposition \ref{prop:penultimate}.
\end{rmrk}

Finally we prove Theorem \ref{thm:main}. The proof is split into seven cases, each of which requires its own delicate analysis. There are two reasons for this, corresponding to the two tasks we have to accomplish: apply Proposition \ref{prop:penultimate} to some power $\psi^q$ of $\psi$ if $\psi^q$ satisfies its hypotheses, and then improve (\ref{eqn:Ngoal}) to (\ref{eqn:goal}) for $(\psi^q,qy_++N)$. (From (\ref{eqn:goal}) for $(\psi^q,qy_++N)$ we obtain (\ref{eqn:goal}) for $(\psi,y_+)$ by applying Lemmas \ref{lem:Nshift} and \ref{lem:rational}.) We accomplish these tasks by perturbing our maps $\psi$ to a sequence of maps for which the inequality (\ref{eqn:Ngoal}) for the perturbed map implies (\ref{eqn:goal}) for $(\psi^q,qy_++N)$, if (\ref{eqn:Ngoal}) is not already stronger. Throughout we have to take care that the orbits picked out by the inequality (\ref{eqn:Ngoal}) for the perturbed map actually correspond to orbits of the original rather than the perturbed map.

When one of the boundary rotation numbers of $\psi$ is rational, we cannot apply Proposition \ref{prop:penultimate} to any power of $\psi$. To address this we show that the rational case (1a) can be reduced to the irrational case (2a). Hutchings had a similar issue in \cite{mean}.

When both boundary rotation numbers are irrational, we focus on bootstrapping (\ref{eqn:Ngoal}) to (\ref{eqn:goal}) for $(\psi^q,qy_++N)$ when (\ref{eqn:Ngoal}) is weaker, as in all cases which fall under (2a). We are able to choose perturbations which both change $\psi$ to a form to which we can apply Proposition \ref{prop:penultimate} and which allow us to bootstrap. Case (2a) is more complicated than the analogous (irrational) part of Hutchings' proof because we have two boundary components rather than one. We have to use different perturbations of $\psi$ based on the relationship between the Calabi invariant and the minimum of $y_+$ and $-y_-+F$, leading to the four sub-cases.

In cases (1b) and (2b) (\ref{eqn:Ngoal}) is stronger than (\ref{eqn:goal}) for $(\psi^q,qy_++N)$, and we do not have to bootstrap. This is also a new phenomenon. We can have
\[
\min\{y_+,-y_-+F\}<hm(y_+,-y_-+F)\leq\V(\tilde\psi)<\max\{y_+,-y_-+F\}
\]
in which case (\ref{eqn:Ngoal}) is stronger than (\ref{eqn:goal}) for $(\psi^q,qy_++N)$. This cannot happen in the case of the disk because it has only one boundary component.

\begin{proof}[Proof of Theorem \ref{thm:main}] We will prove the theorem in cases. Let $m=\min\{y_+,-y_-+F\}$ and $M=\max\{y_+,-y_-+F\}$ as in the proof of Proposition \ref{prop:penultimate}. Let 
\[
y_m=\begin{cases}y_+&\text{ if }y_+=m\\y_-&\text{ if }-y_-+F=m\end{cases} \text{ and }y_M=\begin{cases}y_+&\text{ if }y_+=M\\y_-&\text{ if }-y_-+F=M\end{cases}
\]

There are seven cases we must consider:
\begin{enumerate}
\item $y_m\in\Q$
\begin{enumerate}
\item $\V(\tilde\psi)<m$
\item $m\leq\V(\tilde\psi)$
\end{enumerate}
\item $y_m\in\R-\Q$
\begin{enumerate}
\item $\V(\tilde\psi)<\hm(y_+,-y_-+F)$
\begin{enumerate}
\item $m=M$
\item $\V(\tilde\psi)<m\neq M$
\item $y_m=-y_-$ and $-y_-+F\leq\V(\tilde\psi)$
\item $y_m=y_+\leq\V(\tilde\psi)$
\end{enumerate}
\item $\hm(y_+,-y_-+F)\leq\V(\tilde\psi)$
\end{enumerate}
\end{enumerate}

\subsubsection*{Case (1a), $y_m\in\Q$, $\V(\tilde\psi)<m$}

In this case we cannot apply Proposition \ref{prop:penultimate} to any power of $\psi$ directly. Therefore we will reduce to case (2a).

Choose $\epsilon>0$ so that $\V(\tilde\psi)<m-\epsilon$ and so that $y_m-\epsilon\in\R-\Q$. Choose $D$ so that
\begin{equation}\label{eqn:MDm}
m\leq M-D<m+1
\end{equation}
Choose $\delta,\delta'$ so that when $-1\leq x\leq-1+\delta'$, $\psi$ is rotation by $y_-$, and when $1-\delta\leq x\leq1$, $\psi$ is rotation by $y_+$.

If $y_+=y_M$, let $b_\epsilon:[-1,1]\to[-D-\epsilon,\epsilon]$ be a smooth nonincreasing function which is identically $\epsilon$ near $-1$, identically zero from slightly before $-1+\delta'$ until slightly after $1-\delta$, identically $-D-\epsilon$ near $1$, and for which $\int_{-1}^1b(x)\,dx=0$.

If $y_-=y_M$, let $b_\epsilon:[-1,1]\to[-\epsilon,D+\epsilon]$ be a smooth nonincreasing function which is identically $D+\epsilon$ near $-1$, identically zero from slightly before $-1+\delta'$ until slightly after $1-\delta$, identically $-\epsilon$ near $1$, and for which $\int_{-1}^1b(x)\,dx=0$.

Let $\tau_\epsilon$ be the area-preserving diffeomorphism of $(A,\omega)$ given by
\[
\tau_\epsilon(x,y)=(x,y+2\pi b_\epsilon(x))
\]
and let $\hat\psi$ be the area-preserving diffeomorphism $\tau_\epsilon\circ\psi$.

We claim that $\hat\psi$ now falls under Case (2a), and that (\ref{eqn:goal}) for $\hat\psi$ proves (\ref{eqn:goal}) for $\psi$. Specifically, we need to show that $\hat y_m$ is irrational and that $\V(\tilde{\hat\psi})$ is less than the harmonic mean of $\hat M=M-D-\epsilon$ and $\hat m=m-\epsilon$. We also need to show that the infimum of mean action over the orbits of $\hat\psi$ must be attained by an orbit which is also an orbit for $\psi$.

We chose $\epsilon$ to make $\hat y_m=y_m-\epsilon$ irrational.

To put a bound on the Calabi invariant and to show that the infimum is attained by a shared orbit, we need to compute the average of the action function $\hat f=f(\hat\psi,y_+-\epsilon,\beta)$ and its values on the regions $-1\leq x\leq-1+\delta', 1-\delta\leq x\leq1$ where $\hat\psi$ differs from $\psi$.

First we compute $\hat f$ when $1-\delta\leq x\leq 1$. Because $\psi^*\beta=\beta$ when $x$ is greater than $1-\delta$, we know that $d\hat f=\tau_\epsilon^*\beta-\beta$ in that region. Therefore,
\begin{align*}
\hat f(x,y)&=\int_1^xtb_\epsilon'(t)\,dt+\hat f(1,y)
\\&=xb_\epsilon(x)-(-\epsilon)-\int_1^xb_\epsilon(t)\,dt+y_+-\epsilon
\\&=y_++xb_\epsilon(x)+\int_x^1b_\epsilon(t)\,dt
\end{align*}
Because $b_\epsilon$ is nonincreasing, $\partial_x\hat f=xb_\epsilon'(x)$ is nonpositive. Therefore $\hat f$ achieves its minimum on $1-\delta\leq x\leq1$ when $x=1$. This minimum is $y_+-D-\epsilon$ if $y_+=y_M$ and $y_+-\epsilon$ if $y_+=y_m$.

Next we compute $\hat f$ when $-1+\delta'\leq x\leq 1-\delta$. Let $\eta_{x_1,x_2,y}$ denote the path from $(x_1,y)$ to $(x_2,y)$ parameterized by $t\mapsto(t,y)$. Because $d\tau_\epsilon=0$ between $-1+\delta'$ and $1-\delta$, we have $d\hat f=df$ in that region. Therefore,
\begin{align*}
\hat f(x,y)&=\int_{\eta_{1-\delta,x,y}}df+\hat f(1-\delta,y)
\\&=f(x,y)-f(1-\delta,y)+y_++(1-\delta)b_\epsilon(1-\delta)+\int_{1-\delta}^1b_\epsilon(t)\,dt
\\&=f(x,y)-y_++y_++(1-\delta)(0)+\int_{1-\delta}^1b_\epsilon(t)\,dt
\\&=f(x,y)+\int_{1-\delta}^1b_\epsilon(t)\,dt
\end{align*}
Notice that orbits of $\hat\psi$ in the region $-1+\delta'\leq x\leq1-\delta$ therefore have mean action within $\delta(D+\epsilon)$ of the mean action of the corresponding orbits of $\psi$.

Finally we compute $\hat f$ when $-1\leq x\leq-1+\delta'$. As when $x$ is greater than $1-\delta$, $b_\epsilon$ determines $d\hat f$:
\begin{align*}
\hat f(x,y)&=\int_{-1+\delta'}^xtb_\epsilon'(t)\,dt+\hat f(-1+\delta,y)
\\&=xb_\epsilon(x)-(-1+\delta')b_\epsilon(-1+\delta')-\int_{-1+\delta'}^xb_\epsilon(t)\,dt+f(-1+\delta,y)+\int_{1-\delta}^1b_\epsilon(t)\,dt
\\&=xb_\epsilon(x)-(-1+\delta')(0)+\int_x^1b_\epsilon(t)\,dt-y_-+F
\\&=-y_-+F+xb_\epsilon(x)+\int_x^1b_\epsilon(t)\,dt
\end{align*}
Because $b_\epsilon$ is nonincreasing, $\partial_x\hat f=xb_\epsilon'(x)$ is nonnegative. Therefore $\hat f$ achieves its minimum on $-1\leq x\leq-1+\delta'$ when $x=-1$. This minimum is $-y_-+F-D-\epsilon$ if $y_-=y_M$ and $-y_-+F-\epsilon$ if $y_-=y_m$.

We can put the following upper bound on $\V(\tilde{\hat\psi})$:
\begin{align*}
\V(\tilde{\hat\psi})&=\frac{1}{2}\int_A\hat f\omega
\\&=\frac{1}{2}\left(\int_{-1}^{-1+\delta'}\left(-y_-+F+xb_\epsilon(x)+\int_x^1b_\epsilon(t)\,dt\right)\,dx+\int_{-1+\delta'}^{1-\delta}\left(f+\int_{1-\delta}^1b_\epsilon(t)\,dt\right)\,dx\right.
\\&\;\;\;\;\left.+\int_{1-\delta}^1\left(y_++xb_\epsilon(x)+\int_x^1b_\epsilon(t)\,dt\right)\,dx\right)
\\&=\frac{1}{2}\left(\delta'(-y_-+F)+\int_{-1}^{-1+\delta'}xb_\epsilon(x)\,dx+\int_{-1}^{-1+\delta'}\int_x^1b_\epsilon(t)\,dt\,dx+2\V(\tilde\psi)\right.
\\&\;\;\;\;\left.-\delta'(-y_-+F)-\delta(y_+)+(2-\delta-\delta')\int_{1-\delta}^1b_\epsilon(t)\,dt+\delta(y_+)+\int_{1-\delta}^1xb_\epsilon(x)\,dx\right.
\\&\;\;\;\;\left.+\int_{1-\delta}^1\int_x^1b_\epsilon(t)\,dt\,dx\right)
\\&=\frac{1}{2}\left(2\V(\tilde\psi)+\int_{-1}^1xb_\epsilon(x)\,dx+\int_{-1}^{-1+\delta'}\int_x^1b_\epsilon(t)\,dt\,dx+\int_{1-\delta}^1\int_x^1b_\epsilon(t)\,dt\,dx\right.
\\&\;\;\;\;\left.+(2-\delta-\delta')\int_{1-\delta}^1b_\epsilon(t)\,dt\right)
\\&\leq\V(\tilde\psi)+\frac{1}{2}\left(\delta'(0)+\delta(0)+(2-\delta-\delta')(0)\right)
\\&\leq\V(\tilde\psi)
\end{align*}

The harmonic mean of $\hat M$ and $\hat m$ is greater than $\V(\tilde{\hat\psi})$, because both $\hat M$ and $\hat m$ are.

Finally, the orbits of $\hat\psi$ and $\psi$ can be split into those lying in the range $-1+\delta'\leq x\leq1-\delta$, where the orbits of $\hat\psi$ are in bijection with the orbits of $\psi$, and those for which $-1\leq x\leq-1+\delta'$ or $1-\delta\leq x\leq1$, in which case the two diffeomorphisms have different orbits. We don't want the bound obtained by proving that (\ref{eqn:goal}) holds for $\hat\psi$ in Case (2a) to be identifying an orbit of $\hat\psi$ which is not also shared by $\psi$.

However, this cannot happen. These new orbits have action greater than $m-\epsilon$ and $M-D-\epsilon$, respectively. (\ref{eqn:goal}) implies there exists an orbit with mean action less than or equal to
\begin{equation}\label{eqn:Vlessm}
\V(\tilde{\hat\psi})\leq\V(\tilde\psi)
\end{equation}
The right hand side of (\ref{eqn:Vlessm}) is less than or equal to $M-D-\epsilon$ and $m-\epsilon$ because of our choices of $D$ in (\ref{eqn:MDm}) and $\epsilon$.

Now we know that $\hat\psi$ falls under Case (2a). It remains to show that (\ref{eqn:goal}) for $\hat\psi$ proves (\ref{eqn:goal}) for $\psi$. The mean action of orbits of $\hat\psi$ at which the infimum is attained is within $\delta(D+\epsilon)$ of the mean action of the corresponding orbits of $\psi$. Therefore by sending $\delta\to0$, (\ref{eqn:goal}) for $\hat\psi$ using Case (2a) suffices to prove (\ref{eqn:goal}) for $\psi$.

\subsubsection*{Case (1b) $y_m\in\Q$, $m\leq\V(\tilde\psi)$}

Every $(x,y)\in A$ close enough to the boundary corresponding to the minimum of $y_+$ and $-y_-+F$ is part of a periodic orbit of $\psi$ with mean action $m$. Because $m\leq\V(\tilde\psi)$, these periodic orbits prove (\ref{eqn:goal}). Compare to Remark \ref{rmrk:rationalboundary}.

\subsubsection*{Case (2a)(i) $y_m\in\R-\Q$, $\V(\tilde\psi)<m=M$}

Assume $m=M$ (we will use $M$ to denote both). Choose $\epsilon>0$ so that $\V(\tilde\psi)+\epsilon<M$ and $\V(\tilde\psi)+\epsilon\in\R-\Q$.

Choose $\delta>0$ so that when $x$ is within distance $\delta$ of either of $\pm1$, $\psi$ is rotation by $y_\pm$. Let $b_i:[-1,1]\to[-M+\V(\tilde\psi)+\epsilon,M-\V(\tilde\psi)-\epsilon]$ be a smooth nonincreasing function which is identically $M-\V(\tilde\psi)-\epsilon$ near $-1$, identically zero from before $-1+\delta$ to after $1-\delta$, identically $-M+\V(\tilde\psi)+\epsilon$ near $1$, and for which $B_i=\int_{-1}^1b_i(x)\,dx$ satisfies $F+B_i\in\Q$. Notice that
\begin{equation}\label{eqn:lbbi}
\delta(-M+\V(\tilde\psi)+\epsilon)\leq B_i\leq\delta(M-\V(\tilde\psi)-\epsilon)
\end{equation}

Let $\tau_i$ be the area-preserving diffeomorphism of $(A,\omega)$ given by
\[
\tau_i(x,y)=(x,y+2\pi b_i(x))
\]
and let $\psi_i=\tau_i\circ\psi$.

Notice that the flux of $\psi_i$ applied to the $(x,0)$ curve is $F+B_i$. There is some $q$ for which $q(F+B_i)\in\Z$. We will apply (\ref{eqn:Ngoal}) to $\psi_i^q$ and use it to obtain (\ref{eqn:goal}) for $\psi^q$ by sending $\epsilon\to0$. In order for (\ref{eqn:Ngoal}) applied to $\psi_i^q$ to say anything about the orbits of $\psi$, we need to know that the orbits of $\psi_i^q$ which are not in bijection with the orbits of $\psi^q$ have mean action so great as to not be picked up by the upper bound on the infimum of the mean action of orbits of $\psi_i^q$.

First we compute $f_i=f(\psi_i,M-\epsilon,\beta)$. When $1-\delta\leq x\leq1$, we have $df_i=\tau_i^*\beta-\beta$, therefore
\begin{align*}
f_i(x,y)&=\int_1^xtb'_i(t)\,dt+f_i(1,y)
\\&=xb_i(x)-(-M+\V(\tilde\psi)+\epsilon)-\int_1^xb_i(t)\,dt+M-M+\V(\tilde\psi)+\epsilon
\\&=M+xb_i(x)+\int_x^1b_i(t)\,dt
\end{align*}
Notice that because $b_i'$ is nonpositive, $f_i$ achieves its minimum on $1-\delta\leq x\leq1$ when $x=1$. Therefore in this range, $f_i(x,y)\geq \V(\tilde\psi)+\epsilon$.

Next we compute $f_i$ when $-1+\delta\leq x\leq 1-\delta$. In this region, $df_i=df$, therefore
\begin{align*}
f_i(x,y)&=\int_{\eta_{1-\delta,x,y}}df+f_i(1-\delta,y)
\\&=f(x,y)-M+M+(1-\delta)(0)+\int_{1-\delta}^1b_i(t)\,dt
\\&=f(x,y)+\int_{1-\delta}^1b_i(t)\,dt
\end{align*}
Notice that in this range, the mean action of an orbit of $\psi_i$ is within $\delta(M-\V(\tilde\psi)-\epsilon)$ of the mean action of the corresponding orbit of $\psi$.

Finally we compute $f_i$ when $-1\leq x\leq -1+\delta$. In this region, $df_i=\tau_i^*\beta-\beta$, therefore
\begin{align*}
f_i(x,y)&=\int_{-1+\delta}^xtb_i'(t)\,dt+f_i(-1+\delta,y)
\\&=xb_i(x)-(-1+\delta)(0)-\int_{-1+\delta}^xb_i(t)\,dt+M+\int_{1-\delta}^1b_i(t)\,dt
\\&=M+xb_i(x)+\int_x^1b_i(t)\,dt
\end{align*}
Notice that because $b_i'$ is nonpositive and $x$ is negative in this range, $f_i$ achieves its minimum on $-1\leq x\leq-1+\delta$ when $x=-1$. Therefore in this range, $f_i(x,y)\geq\V(\tilde\psi)+\epsilon+B_i$.

Next we obtain an upper bound for $\V(\tilde\psi_i)$.
\begin{align*}
\V(\tilde\psi_i)&=\frac{1}{2}\int_Af_i\omega
\\&=\frac{1}{2}\left(\int_{-1}^{-1+\delta}\left(M+xb_i(x)+\int_x^1b_i(t)\,dt\right)\,dx+\int_{-1+\delta}^{1-\delta}\left(f+\int_{1-\delta}^1b_i(t)\,dt\right)\,dx\right.
\\&\;\;\;\;\left.+\int_{1-\delta}^1\left(M+xb_i(x)+\int_x^1b_i(t)\,dt\right)\,dx\right)
\\&=\frac{1}{2}\left(\delta M+0+\delta(M-\V(\tilde\psi)-\epsilon)+2\V(\tilde\psi)-2\delta M+(2-2\delta)(0)+\delta M+0+\delta(0)\right)
\\&=\V(\tilde\psi)+\delta(M-\V(\tilde\psi)-\epsilon)
\end{align*}

For conciseness, let $\V_\epsilon$ denote $\V(\tilde\psi)+\epsilon$ throughout the remainder of this case. Applying (\ref{eqn:Ngoal}) to $\psi_i^q$ gives us
\begin{align}
\inf&\left\{\frac{\A_N(\gamma)}{\ell(\gamma)}\;\middle|\;\gamma\in\P(\psi_i^q)\right\}\leq\sqrt{\hm(q\V_\epsilon+N,q(\V_\epsilon+B_i)+N)(q\V(\tilde\psi_i)+N)} \nonumber
\\&\leq\sqrt{\hm(q\V_\epsilon+N,q(\V_\epsilon+\delta(M-\V_\epsilon))+N)\left(q\left(\V(\tilde\psi)+\delta(M-\V_\epsilon)\right)+N\right)} \label{eqn:2ai}
\end{align}

We claim that the upper bound in (\ref{eqn:2ai}) is lower than the minimum mean action of all orbits which aren't in bijection with orbits of $(\psi^q,qy_++N)$. By combining the lower bound (\ref{eqn:lbbi}) with our lower bounds on $f_i$ near the boundary, we know that the mean action of any of these new orbits is at least $q\left(\V_\epsilon-\delta(M-\V_\epsilon)\right)+N$. Meanwhile, the right hand side of (\ref{eqn:2ai}) is at most the geometric mean of
\[
q\left(\V_\epsilon+\delta(M-\V_\epsilon)\right)+N=\max\{q\V_\epsilon+N,q\left(\V_\epsilon+\delta(M-\V_\epsilon)\right)+N\}
\]
and $q\left(\V(\tilde\psi)+\delta(M-\V_\epsilon)\right)+N$. Let $V:=q\V(\tilde\psi)+N$. Therefore, in order to show that $q\left(\V_\epsilon-\delta(M-\V_\epsilon)\right)+N$ is greater than the right hand side of (\ref{eqn:2ai}), we need to show
\begin{equation}\label{eqn:2aiprime}
V+q\epsilon-q\delta(M-\V_\epsilon)>\sqrt{(V+q\epsilon+q\delta(M-\V_\epsilon))(V+q\delta(M-\V_\epsilon))}
\end{equation}

The arithmetic mean is greater than the geometric mean, so (\ref{eqn:2aiprime}) will follow if the following holds:
\begin{align*}
V+q\epsilon-q\delta(M-\V_\epsilon)&>\frac{1}{2}\left(V+q\epsilon+q\delta(M-\V_\epsilon)+V+q\delta(M-\V_\epsilon)\right)
\\&=V+\frac{q\epsilon}{2}+q\delta(M-\V_\epsilon)
\end{align*}
which is equivalent to
\begin{equation}\label{eqn:epdel}
\frac{\epsilon}{2}>2\delta(M-\V_\epsilon)
\end{equation}
which holds whenever $\delta$ is very small compared to $\epsilon$.

We also showed earlier that the mean action of orbits of $(\psi_i^q,qy_++N)$ with $-1+\delta\leq x\leq1-\delta$ is at least the mean action of the corresponding orbit of $(\psi^q,qy_++N)$ minus $q\delta(M-\V_\epsilon)$. Therefore (\ref{eqn:2ai}) implies the following for the unperturbed map $\psi$:
\begin{align*}
&\inf\left\{\frac{\A_N(\gamma)}{\ell(\gamma)}\;\middle|\;\gamma\in\P(\psi^q)\right\}
\\&\leq\sqrt{\hm(q\V_\epsilon+N,q(\V_\epsilon+\delta(M-\V_\epsilon))+N)\left(q\left(\V(\tilde\psi)+\delta(M-\V_\epsilon)\right)+N\right)}+q\delta(M-\V_\epsilon)
\end{align*}
Taking $\epsilon,\delta\to0$ gives
\[
\inf\left\{\frac{\A_N(\gamma)}{\ell(\gamma)}\;\middle|\;\gamma\in\P(\psi^q)\right\}\leq q\V(\tilde\psi)+N
\]
which implies
\[
\inf\left\{\frac{\A(\gamma)}{\ell(\gamma)}\;\middle|\;\gamma\in\P(\psi^q)\right\}\leq q\V(\tilde\psi)
\]
via (\ref{eqn:shiftgoal}). Applying Lemma \ref{lem:rational} allows us to obtain (\ref{eqn:goal}).

\subsubsection*{Case (2a)(ii) $y_m\in\R-\Q$, $\V(\tilde\psi)<m\neq M$}

Assume $\V(\tilde\psi)<m\neq M$. If $y_+=y_M$ then construct a perturbed diffeomorphism as $\psi_i$ was constructed in Case (2a)(i), using as $b$ a smooth nonincreasing function which is identically zero until after $1-\delta$ and then identically $m-M$ near 1; if $y_-=y_M$, use a smooth nonincreasing function which is identically $M-m$ near $-1$ and identically zero from before $-1+\delta$.

We are now in Case (2a)(i), and have only changed our Calabi invariant and the actions of the orbits which we haven't touched by at most $\delta(M-m)$. All new orbits will have action at least $m-\delta(M-m)$, which we can arrange to be greater than the upper bound of $\V(\tilde\psi)+\delta(M-m)$ given by applying Case (2a)(i) to the perturbed diffeomorphism, by making $\delta$ small enough in light of $\V(\tilde\psi)<m$. Sending $\delta\to0$ then gives (\ref{eqn:goal}).

\subsubsection*{Case (2a)(iii) $y_m\in\R-\Q$, $y_m=y_-$, $-y_-+F\leq\V(\tilde\psi)<\hm(y_+,-y_-+F)$}

Assume $-y_-+F\leq\V(\tilde\psi)+\epsilon<\hm(y_+,-y_-+F)$ and $\V(\tilde\psi)+\epsilon\in\R-\Q$.

Construct a perturbed diffeomorphism $\psi_{iii}$ as $\psi_i$ was constructed in Case (2a)(i), using as $b_{iii}$ a smooth nonincreasing function which is identically zero until after $1-\delta$ and then identically $-y_++\V(\tilde\psi)+\epsilon$ near 1. Also choose $F+\int_{-1}^1b_{iii}(x)\,dx\in\Q$. Because $y_m\in\R-\Q$, it follows that $-y_-+F+\int_{-1}^1b_{iii}(x)\,dx\in\R-\Q$.

As in Case (2a)(i), we can compute that we have only changed our Calabi invariant and the actions of the orbits which carry through to ${\psi_{iii}}^q$ from $\psi^q$ by at most $q\delta(y_+-\V(\tilde\psi)-\epsilon)$. Let $\V_\epsilon$ denote $\V(\tilde\psi)+\epsilon$ throughout this case. All new orbits of ${\psi_{iii}}^q$ have action at least $q\V_\epsilon+N$, which we want to be greater than the bound
\[
\sqrt{\hm\left(q\V_\epsilon+N,q\left(-y_-+F+\delta(y_+-\V_\epsilon)\right)+N\right)\left(q\left(\V(\tilde\psi)+\delta(y_+-\V_\epsilon)\right)+N\right)}
\]
given by applying Proposition \ref{prop:penultimate} to ${\psi_{iii}}^q$ so that any orbits of ${\psi_{iii}}^q$ which do not correspond to orbits of $\psi^q$ cannot be picked out by the bound. Because the arithmetic mean is greater than both the geometric and harmonic means, it suffices to show that $q\V_\epsilon+N$ is greater than the appropriate combination of arithmetic means. This follows by making $\delta$ very small compared to $\epsilon$ (compare to the reasoning we used to show (\ref{eqn:2aiprime}) in Case (2a)(i); this bound is actually sharper, since $y_m<\V(\tilde\psi)$).

Therefore the infimum guaranteed by (\ref{eqn:Ngoal}) refers to orbits which correspond to orbits of $\psi$. Because the mean action of these orbits as orbits of $({\psi_{iii}}^q,qy_++N)$ differ by their mean action as orbits of $\psi^q$ by at most $q\delta(y_+-\V_\epsilon)$, we obtain
\begin{align*}
&\inf\left\{\frac{\A_N(\gamma)}{\ell(\gamma)}\;\middle|\;\gamma\in\P(\psi^q)\right\}
\\&\leq\sqrt{\hm\left(q\V_\epsilon+N,q\left(-y_-+F+\delta(y_+-\V_\epsilon)\right)+N\right)\left(q\left(\V(\tilde\psi)+\delta(y_+-\V_\epsilon)\right)+N\right)}
\\&\;\;\;\;\;\;+\delta(y_+-\V(\tilde\psi)-\epsilon)
\end{align*}
from which, by sending $\delta\to0$, $\epsilon\to0$, we obtain
\begin{equation}\label{eqn:2aiii}
\inf\left\{\frac{\A_N(\gamma)}{\ell(\gamma)}\;\middle|\;\gamma\in\P(\psi^q)\right\}\leq\sqrt{\hm\left(q\V(\tilde\psi)+N,q\left(-y_-+F\right)+N\right)\left(q\V(\tilde\psi)+N\right)}
\end{equation}
Notice now that $q\V(\tilde\psi)+N\geq\hm(q\V(\tilde\psi)+N,q\left(-y_-+F\right)+N)$, therefore the right hand side of (\ref{eqn:2aiii}) is less than or equal to $q\V(\tilde\psi)+N$, giving us
\[
\inf\left\{\frac{\A_N(\gamma)}{\ell(\gamma)}\;\middle|\;\gamma\in\P(\psi^q)\right\}\leq q\V(\tilde\psi)+N
\]
Undoing the rotation as in (\ref{eqn:shiftgoal}) and then applying Lemma \ref{lem:rational} gives us (\ref{eqn:goal}).

\subsubsection*{Case (2a)(iv) $y_m\in\R-\Q$, $y_m=y_+<\V(\tilde\psi)<\hm(y_+,-y_-+F)$}

Assume $y_+\leq\V(\tilde\psi)+\epsilon<\hm(y_+,-y_-+F)$ and $\V(\tilde\psi)+\epsilon\in\R-\Q$.

In this case, to construct the perturbation $\psi_{iv}$ use a smooth nonincreasing function $b_{iv}$ which is identically $-y_-+F-\V(\tilde\psi)-\epsilon>0$ near $-1$, identically zero from before $-1+\delta$, and for which $F+\int_{-1}^1b(x)\,dx\in\Q$.

Notice that if $-1+\delta\leq x\leq1$, the action function $f_{iv}$ of $\psi_{iv}$ equals $f$. (Therefore in particular, the orbits of $\psi_{iv}$ which correspond to orbits of $\psi$ share the same mean action.) When $-1\leq x\leq-1+\delta$,
\begin{align*}
f_{iv}(x,y)&=\int_{-1+\delta}^xtb_{iv}'(t)\,dt+f_{iv}(x,y)
\\&=xb_{iv}(x)-(-1+\delta)(0)-\int_{-1+\delta}^xb_{iv}(t)\,dt+(-y_-+F)
\\&=-y_-+F+xb_{iv}(x)+\int_x^{-1+\delta}b_{iv}(t)\,dt
\end{align*}
Notice that $b_{iv}$ is nonincreasing, so $f_{iv}$ achieves its minimum on $-1\leq x\leq-1+\delta$ when $x=-1$. This minimum is $\V(\tilde\psi)+\epsilon+\int_{-1}^1b_{iv}(t)\,dt$, which is greater than $\V(\tilde\psi)+\epsilon$.

From here the analysis goes through in an analogous manner to that in Case (2a)(iii).

\subsubsection*{Case (2b) $y_m\in\R-\Q$, $\hm(y_+,-y_-+F)\leq\V(\tilde\psi)$}

We simply perturb $\psi$ so that $y_M$ is irrational and the flux is rational. An application of Proposition \ref{prop:penultimate} combined with the usual shift by $N$ and division by $q$ then immediately gives the result, because the perturbed Calabi invariant will be greater than the harmonic mean in the upper bound, so greater than the upper bound. The same checks (change in Calabi invariant, mean action of orbits which are untouched, and mean action of orbits which are new) as in the previous cases are required and apply.

If $y_+=y_M$ and for $\epsilon>0$, $\V(\tilde\psi)<y_+-\epsilon$, and $y_+-\epsilon\in\R-\Q$, we use the perturbation $b:[-1,1]\to[-\epsilon,0]$ which is nonincreasing, zero until after $1-\delta$, $-\epsilon$ near $1$, and for which $F+\int_{-1}^1b(x)\,dx\in\Q$.

If $y_-=y_M$ and for $\epsilon>0$, $\V(\tilde\psi)<-y_-+F-\epsilon$, and $y_-+\epsilon\in\R-\Q$, we use the perturbation $b:[-1,1]\to[0,\epsilon]$ which is nonincreasing, $\epsilon$ near $-1$, zero from before $-1+\delta$, and for which $F+\int_{-1}^1b(x)\,dx\in\Q$.

\end{proof}

\begin{appendices}

\section{Criterion for diffeomorphisms of the annulus and disk to be unrelated}\label{app:new}

We will attempt to devise a map $\kappa:A\to\D^2$ for which the following conditions hold.
\begin{itemize}
\item Denote by $\psi_\kappa$ some extension of $\kappa\circ\psi\circ\kappa^{-1}$ to the whole disk (\emph{a priori} it is only defined on the image of $\kappa$). We need $\psi_\kappa$ to satisfy the hypotheses of Theorem \ref{thm:disk}.
\item The existence of an orbit of $\psi_\kappa$ satisfying (\ref{eqn:ineqdisk}) implies the existence of an orbit of $\psi$ satisfying (\ref{eqn:goal}); the only way to do this is to be sure that any orbits of $\psi_\kappa$ which have no relationship to any orbit of $\psi$ have mean action greater than $\V(\tilde\psi)$.
\end{itemize}
If such a $\kappa$ exists, then our Theorem \ref{thm:main} follows as a corollary of Theorem \ref{thm:disk}. However, while we do not prove that given $\psi$ there is no such map $\kappa$ (which depends on $\psi$), we do show that the best candidate for $\kappa$ applying to all diffeomorphisms of the annulus cannot satisfy the second requirement for a large class of $\psi$. Specifically, this class consists of all those $\psi$ for which (\ref{eqn:12fv}) holds.

The first condition indicates that $\kappa$ ought to send the boundary component of $A$ corresponding to $\max\{y_+,-y_-+F\}$ to the boundary of $\D^2$. Therefore, throughout the rest of this appendix we assume that $y_+=\max\{y_+,-y_-+F\}$; the case when $-y_-+F$ can be treated in exactly the same way.

The second condition means that unless we plan to tailor $\kappa$ very carefully to each $\psi$ in turn, we want the orbits of $\psi_\kappa$ which don't correspond to orbits of $\psi$ to have easily computable mean action, and we also want $\psi_\kappa$ to have an easily computable action function and Calabi invariant. The only way to be guaranteed we can compute these is to ask for $\kappa$ to have image
\[
\{(r,\theta)\in\D^2|r_0\leq r\leq1\}
\]
in polar coordinates on $\D^2$, where $r_0\in[0,1]$, for $\kappa$ to restrict to
\begin{equation}\label{eqn:kappabetaD}
(1,y)\mapsto(1,y)\text{ and }(-1,y)\mapsto(r_0,y)
\end{equation}
on $\partial A$, and for
\begin{equation}\label{eqn:kappaomegaD}
\kappa^*\left(\frac{1}{\pi}r\,dr\wedge d\theta\right)=\frac{1-r_0^2}{2}\omega
\end{equation}
so that if $\psi$ is a symplectomorphism, so is $\psi_\kappa$. The most obvious choice is the map $\kappa(x,y)=\left(\sqrt{\frac{1-r_0^2}{2}x+\frac{1+r_0^2}{2}},y\right)$, though it is not necessary that $\kappa$ be exactly of this form. We can extend $\kappa\circ\psi\circ\kappa^{-1}$ by defining it to be a rotation by $y_-$ on the disk of radius $r_0$. Its boundary rotation number is $y_+$.

Let $f_\kappa$ denote the action function of $(\psi_\kappa,y_+)$. If
\begin{equation}\label{eqn:zeroineq}
f_\kappa(0,0)\leq\V(\psi_\kappa,y_+)
\end{equation}
then the fixed point of $\psi_\kappa$ at the origin (as well as any periodic orbits in the disk of radius $r_0$ in the case when $y_-\in\Q$) satisfies the conclusion of Theorem \ref{thm:disk}, so we can learn nothing about the periodic orbits of $\psi$ from those of $\psi_\kappa$.

\begin{prop}\label{prop:appendix} If (\ref{eqn:12fv}) holds, then (\ref{eqn:zeroineq}) holds. That is, if
\[
\frac{1}{2}F\leq\V(\tilde\psi)
\]
then Theorem \ref{thm:main} does not follow from \cite[Theorem 1.2]{mean} by filling with a disk the boundary along which $f$ takes the value $\min\{f(1,y),f(-1,y)\}$.
\end{prop}

\begin{proof}

We will show that when $r_0=0$, (\ref{eqn:zeroineq}) is equivalent to (\ref{eqn:12fv}). When $r_0>0$, it can be derived in exactly the same way that (\ref{eqn:zeroineq}) is equivalent to
\[
\frac{1}{2}(1-r_0^2)F-2r_0^2(-y_-+F)\leq(1-r_0^2)\V(\tilde\psi)
\]
which follows from (\ref{eqn:12fv}). Therefore from now on we assume $r_0=0$.

We will use the primitive $\beta_{\D^2}=\frac{r^2}{2\pi}\,d\theta$ and the notation $\omega_{\D^2}$ for $\frac{1}{\pi}r\,dr\wedge d\theta$. Let $\eta_0$ denote the curve $\{y=0\}$ in $A$, oriented in the direction of decreasing $x$. On the left hand side of (\ref{eqn:zeroineq}), we have
\begin{align*}
f_\kappa(0,0)-f_\kappa(1,0)&=\int_{\kappa(\eta_0)}df_\kappa
\\&=\int_{\kappa(\eta_0)}\left(\psi_\kappa^*\beta_{\D^2}-\beta_{\D^2}\right)
\\&=\int_{\eta_0}\left(\kappa^*{\kappa^{-1}}^*\psi^*\kappa^*\beta_{\D^2}-\kappa^*\beta_{\D^2}\right)
\\&=\int_{\eta_0}\left(\psi^*\kappa^*\beta_{\D^2}-\kappa^*\beta_{\D^2}\right)
\\&=\frac{1}{2}F-\int_0^{2\pi y_+}\kappa^*\beta_{\D^2}|_{\partial_+A}-\int_{2\pi y_-}^0\kappa^*\beta_{\D^2}|_{\partial_-A}
\\&=\frac{1}{2}F-y_++0
\end{align*}
where the final term is zero because of (\ref{eqn:kappabetaD}), (\ref{eqn:kappaomegaD}), and Stokes' theorem. Therefore, because $f_\kappa(1,0)=y_+$,
\[
f_\kappa(0,0)=\frac{1}{2}F
\]

On the right hand side of (\ref{eqn:zeroineq}), we have
\[\V(\psi_\kappa,y_+)=\int_{\D^2}f_\kappa\omega_{\D^2}=\int_{\kappa_*(A)}f_\kappa\,d\beta_{\D^2}=\int_A\kappa^*(f_\kappa\,d\beta_{\D^2})=\int_A(f_\kappa\circ\kappa)\,d(\kappa^*\beta_{\D^2})
\]
Notice that
\[
d(f_\kappa\circ\kappa)=\kappa^*df_\kappa=\kappa^*((\kappa^{-1})^*\psi^*\kappa^*\beta_{\D^2}-\beta_{\D^2})=\psi^*\kappa^*\beta_{\D^2}-\kappa^*\beta_{\D^2}
\]
therefore the Calabi invariant of $\psi_\kappa$ on $\D^2$ can be computed as the Calabi invariant of $\psi$ using $\kappa^*\beta_{\D^2}$ and $\kappa^*\omega_{\D^2}$ rather than $\beta$ and $\omega$. By Lemma \ref{lem:calabiindep}, the Calabi invariant depends only on the restriction to $\partial A$ of the one-form with which it is computed. Therefore we may compute $\V(\psi_\kappa,y_+)$ using the following one-form, which agrees with $\kappa^*\beta_{\D^2}$ on $\partial A$:
\[
\beta'=\frac{1}{2}\beta+\frac{1}{4\pi}\,dy
\]
For the differential of the action function, we get
\[
df_{(\psi,y_+,\beta')}=\psi^*\beta'-\beta'=\psi^*\left(\frac{1}{2}\beta+\frac{1}{4\pi}\,dy\right)-\frac{1}{2}\beta-\frac{1}{4\pi}\,dy=\frac{1}{2}\,df+\frac{1}{4\pi}(\psi^*dy-dy)
\]
Let $\tilde\psi(x,y)=(\tilde\psi_1(x,y),\tilde\psi_2(x,y))$. Then
\[
\psi^*dy=\frac{\partial\tilde\psi_2}{\partial x}\,dx+\frac{\partial\tilde\psi_2}{\partial y}\,dy
\]
Therefore,
\[
d(\tilde\psi_2-y)=\frac{\partial\tilde\psi_2}{\partial x}\,dx+\left(\frac{\partial\tilde\psi_2}{\partial y}-1\right)\,dy=\psi^*dy-dy
\]
This gives us
\[
d\left(\frac{1}{2}f+\frac{1}{4\pi}(\tilde\psi_2-y)\right)=df_{(\psi,y_+,\beta')}
\]
We can also check that $f_{(\psi,y_+,\beta')}(1,y)=y_+$. Therefore
\[
f_{(\psi,y_+,\beta')}(x,y)=\frac{1}{2}f(x,y)+\frac{1}{4\pi}(\tilde\psi_2(x,y)-pr_2(x,y))
\]

From $f_{(\psi,y_+,\beta')}$ we can compute $\V(\psi_\kappa,y_+)$:
\begin{align}
\V(\psi_\kappa,y_+)&=\frac{\int_Af_{(\psi,y_+,\beta')}d\beta'}{\int_Ad\beta'} \nonumber
\\&=\frac{\int_A\left(\frac{1}{2}f(x,y)+\frac{1}{4\pi}(\tilde\psi_2(x,y)-pr_2(x,y))\right)\frac{1}{2}\omega}{\int_A\frac{1}{2}\omega} \nonumber
\\&=\frac{\int_A\frac{1}{2}f\frac{1}{2}\omega}{\int_A\frac{1}{2}\omega}+\frac{\int_A\frac{1}{4\pi}(\tilde\psi_2-pr_2)\frac{1}{2}\omega}{\int_A\frac{1}{2}\omega} \nonumber
\\&=\frac{1}{2}\V(\tilde\psi)+\frac{1}{16\pi^2}\int_0^{2\pi}\left(\int_{-1}^1(\tilde\psi_2-pr_2)\,dx\right)dy \nonumber
\\&=\frac{1}{2}\V(\tilde\psi)+\frac{1}{16\pi^2}\int_0^{2\pi}2\pi F\,dy \label{eqn:areaundergraph}
\\&=\frac{1}{2}\V(\tilde\psi)+\frac{1}{4}F \nonumber
\end{align}
(To obtain (\ref{eqn:areaundergraph}), notice that because $\psi$ is area-preserving, the integral $\int_{-1}^1\tilde\psi_2-pr_2\,dx$ measures the area under the graph of the function $\psi_2(x,0)$ with area form $dx\wedge dy$ on $\tilde A$.)

After multiplying by two, (\ref{eqn:zeroineq}) is equivalent to (\ref{eqn:12fv}):
\[
\frac{1}{2}F\leq\V(\tilde\psi)
\]
Notice that when we add a full $2\pi$ rotation to $\psi$, the flux changes by two while the Calabi invariant changes by one. Therefore (\ref{eqn:12fv}) is a property of $\psi$ rather than of the lift $\tilde\psi$.

If (\ref{eqn:12fv}) holds then the hypotheses of Theorem \ref{thm:disk} also hold. This is because
\[
\V(\tilde\psi_\kappa,y_+)=\frac{1}{2}\V(\tilde\psi)+\frac{1}{4}F\overset{\ref{eqn:12fv}}{\leq}\V(\tilde\psi)<y_+
\]
so we can apply Theorem \ref{thm:disk}.

\end{proof}

It remains to understand whether or not (\ref{eqn:12fv}) holds for a robust class of $\tilde\psi$. That is, we would like Theorem \ref{thm:main} to not only be a nontrivial extension of \cite[Theorem 1.2]{mean}, but to apply to far more symplectomorphisms than the strategy outlined at the beginning of this appendix applies to. As we have seen in the proof of Proposition \ref{prop:appendix}, this is equivalent to (\ref{eqn:12fv}) holding for ``many" $\psi$, in some sense of the word ``many."

Notice that for rigid rotations, the equality $\frac{1}{2}F=\V(\tilde\psi)$ holds. Therefore, it's not unreasonable to suspect that under some natural finite measure on the group $G$ of pairs $(\psi,y_+)$, the inequality (\ref{eqn:12fv}) could hold for a set of half measure. We will not go so far, but will at least convince ourselves that (\ref{eqn:12fv}) holds for enough pairs $(\psi,y_+)$ so that Theorem \ref{thm:main} significantly extends \cite[Theorem 1.2]{mean}, by considering the case where
\begin{equation}\label{eqn:constxpsi}
\tilde\psi(x,y)=(x,y+2\pi g(x))\text{ and }\beta=\frac{x}{2\pi}\,dy
\end{equation}
In this case we get $df=xg'(x)\,dx$. Therefore
\[
f(x,y)=\int_1^xtg'(t)\,dt+f(1,y)=xg(x)+\int_x^1g(t)\,dt-g(1)+g(1)=xg(x)+\int_x^1g(t)\,dt
\]
We can integrate $f$ to obtain
\[
\V(\tilde\psi)=\frac{1}{2}\int_{-1}^1\left(xg(x)+\int_x^1g(t)\,dt\right)\,dx
\]
We can integrate $g$ to obtain
\[
F=\int_{-1}^1g(x)\,dx
\]
Therefore, all $\tilde\psi$ of the form (\ref{eqn:constxpsi}) for which
\[
\int_{-1}^1g(x)\leq\int_{-1}^1\left(xg(x)+\int_x^1g(t)\,dt\right)\,dx
\]
fall into the class of area-preserving diffeomorphisms of $A$ for which Theorem \ref{thm:main} is new. We can apply the hypotheses so long as $\frac{1}{2}\int_{-1}^1\left(xg(x)+\int_x^1g(t)\,dt\right)\,dx<\max\{g(1),-g(-1)+\int_{-1}^1g(x)\,dx\}$. For example, $g(x)=cx^n$ for $c\in\R_{>0}, n\in\Z$ are specific examples to which we can apply the theorem and for which it is new.

\end{appendices}

\section*{Acknowledgement}

We thank Michael Hutchings for his guidance throughout this project. We also thank James Conway for his suggestions in the case where $y_+-y_-$ is rational, and Barney Bramham and Umberto Hryniewicz for helpful discussions and comments.

\bibliography{references_W_MAA}
\bibliographystyle{plain}

\end{document}